\documentclass[12pt]{amsart}

\usepackage[T1]{fontenc}\usepackage{graphicx}
\usepackage{amsmath, fullpage}
\usepackage{amssymb}
\usepackage{eucal}
\usepackage{mathrsfs} 
 \usepackage{mathptmx, tikz} 
\usepackage{geometry}

\usepackage{comment}

\usepackage{mathtools}
\DeclarePairedDelimiter{\ceil}{\lceil}{\rceil}

\newtheorem{theorem}{Theorem}[section]
\newtheorem{lemma}{Lemma}[section]
\newtheorem{proposition}{Proposition}[section]
\newtheorem{corollary}{Corollary}[section]

\theoremstyle{definition}
\newtheorem{definition}{Definition}[section]

\theoremstyle{remark}
\newtheorem{example}{Example}[section]
\newtheorem{remark}{Remark}[section]
\newtheorem{assump}{Assumption}[section]
\numberwithin{equation}{section}

 \DeclareSymbolFont{largesymbols}{OMX}{yhex}{m}{n}
  \DeclareMathAccent{\widehat}{\mathord}{largesymbols}{"62}

\newcommand{\vertiii}[1]{{\left\vert\kern-0.25ex\left\vert\kern-0.25ex\left\vert #1 
    \right\vert\kern-0.25ex\right\vert\kern-0.25ex\right\vert}}

\newcommand{\RR}{\mathbf{R}}

\newcommand{\PP}{\mathbf{P}}
\newcommand{\N}{\mathbf{N}}

\newcommand{\EE}{\mathbf{E}}
\newcommand{\Z}{\mathbf{Z}}

\title[dd]{Quantitative positivity of transition densities for random perturbations of Hamiltonian systems}
\author{Shimaa Elesaely, David P. Herzog, Kyle L. Liss}
\begin{document}
\maketitle

\begin{abstract}
	We study a class of diffusion processes arising from random perturbations of conservative Hamiltonian systems. Under a set of abstract hypotheses---including basic structural assumptions on the Hamiltonian, a weak Lyapunov structure, and a quantitative notion of hypoellipticity---we prove that transition densities satisfy a sharp, uniform pointwise lower bound over Hamiltonian sublevel sets in the small noise limit $\epsilon \to 0$. By applying our general theorem, we obtain quantitative minorization estimates for a variety of models including Langevin dynamics, chains of oscillators coupled to heat bathes at different temperatures, and finite-dimensional fluid models such as stochastically forced Galerkin truncations of the Navier-Stokes equations and the Lorenz '96 system. As a corollary, assuming a stronger Lyapunov structure, our main result yields a sharp exponential rate of convergence to equilibrium for $0 < \epsilon \ll 1$ in a weighted total variation norm.   A central feature of our approach is that it does not require knowledge of the explicit form of the invariant measure, nor even its existence, and hence is broadly applicable to deduce minorization for physically relevant systems where invariant measures are inaccessible.

\end{abstract}
 \section{Introduction}
 	Often motivated by applications to Markov Chain Monte Carlo (MCMC) or the Boltzmann equation, recent research in the area of ergodic diffusion processes has focused on obtaining explicit rates of convergence to equilibrium with respect to key parameters in the system~\cite{AAMN_24, ADSW_21, BGH_21, BS_25, CLW_23, CGHS_21, DMS_15, EGZ_19, GS_14, Hera_06,LSR_10, LSS_20, RS_18, Tal_02, Vil_06, Vil_09}.  A well-studied example is (second-order) Langevin dynamics, which is a stochastic differential equation (SDE) of the form 
\begin{align}
 		\label{eqn:LD}
 		dx_t^\gamma &= v_t^\gamma \, dt \\
 		\nonumber dv_t^\gamma& = - \gamma v_t^\gamma \, dt - \nabla U(x_t^\gamma) \, dt+ \sqrt{2\gamma } \, dB_t,
 	\end{align} 
 	where the process $z_t^\gamma:=(x_t^\gamma, v_t^\gamma)$ evolves on an open subset $\mathcal{X}$ of $\RR^n \times \RR^n$, $U:\RR^n \to [0,\infty]$ is the potential, $\gamma>0$ is the friction parameter and $B_t$ is a standard Brownian motion on $\RR^n$.  A central problem for the dynamics~\eqref{eqn:LD} is to quantify its rate of convergence to equilibrium with respect to the parameter $\gamma > 0$ for a general class of potential functions~\cite{BS_25, CLW_23, CGHS_21, EGZ_19, RS_18, Tal_02, Vil_09} (see also~\cite{LSS_20} in the case of adaptive Langevin dynamics).  In the setting of \eqref{eqn:LD}, various convergence theories have emerged~\cite{AAMN_24, DMS_15, EGZ_19, Hera_06, Vil_09}, each constructing a notion of distance (measuring convergence to equilibrium) that boasts some advantage(s) over the others.  Such distances are often equivalent to either $L^2(\mu)$ or $H^1(\mu)$, where $\mu$ is the stationary measure, and the methods employed are more functional-analytic than probabilistic, relying on the explicit form of the invariant measure $\mu$ and Poincar\'{e} inequalities to extract convergence bounds~\cite{AAMN_24, DMS_15, Hera_06, Vil_09}.  See~\cite{EGZ_19}, however, for a more probabilistic method using quantitative couplings and a constructed Wasserstein distance.  
 	
 	In the case of Langevin dynamics~\eqref{eqn:LD}, superiority of a convergence theory can be judged both by the generality of the class of potentials it can treat as well as how the convergence bounds scale with respect to $\gamma$.  Currently, functional-analytic approaches using distances equivalent to either $L^2(\mu)$ or $H^1(\mu)$ have proven to be particularly effective at capturing sharp dependence on $\gamma$~\cite{BS_25, CLW_23, CGHS_21, DMS_15, RS_18}, which is arguably why they have become prominent. However, these methods often rely on the explicit form of the invariant Gibbs measure for \eqref{eqn:LD}.  Thus, they can be difficult to extend to other physically relevant settings where no such closed form is available. A simple example in the above setting~\eqref{eqn:LD} is when each coordinate is coupled to a heat bath at a different temperature, so the isotropic noise $\sqrt{2\gamma}dB_t$ in \eqref{eqn:LD} is replaced by anisotropic noise of the form $\sqrt{2\gamma T}dB_t$, where $T = \mathrm{diag}(T_1,T_2,\ldots,T_n)$ is an $n\times n$ diagonal matrix with $T_i >0$ for all $i$. By contrast, more classical probabilistic methods such as couplings or, in particular, Doeblin's minorization method~\cite{HM_11, MSH_02} are highly flexible and apply in great generality, but have lost some favor, primarily because it is hard to see how the convergence bounds depend on parameters.  Furthermore, even if one does produce explicit bounds using such techniques~\cite{EGZ_19, MSH_02}, they do not scale optimally with respect to $\gamma$.           
 	
 	The principal problem underlying the difficulty in applying probabilistic approaches is the absence of local, quantitative positive lower bounds on the transition densities at the time scale on which equilibration occurs. The case of interest in this paper is the small-$\gamma$ limit, where convergence takes place on the $\gamma^{-1}$ scale.  Positive lower bounds on the transition density are usually proven using a combination of control theory and smoothing of the Markov semigroup ensured by hypoellipticity of the driving second-order operators~\cite{RB_06}.  However, in small noise limits, these bounds are largely existential with the exception of some cases that use Malliavin calculus~\cite{E_18} or where the process is simple enough that the density is explicitly computable~\cite{BGHKM_25}. The goal of this paper is to resolve this problem by producing optimal, quantitative lower bounds on the transition density with respect to a small parameter for a wide class of diffusion processes with underlying conservation laws.  In particular, while thus far we have used Langevin dynamics as a motivating example, our results apply more generally, including to systems of strongly coupled oscillators and to random perturbations of finite-dimensional fluid-like equations such as the Lorenz '96 model and stochastically forced Galerkin truncations of the Navier--Stokes' equations.
 	
 	We now state a corollary (Theorem~\ref{thm:LD} below) of our main abstract result (Theorem~\ref{thm:minorization}) as it applies to our motivating example of Langevin dynamics. We suppose that the potential $U:\RR^n\rightarrow [0, + \infty]$ in \eqref{eqn:LD} satisfies the following set of hypotheses.
 	\begin{assump}
 		\label{assump:Ug}~~
 		\item[(U1)] 
 		The set $\mathcal{Y}:= \{ x\in \RR^n \, : \, U(x)< \infty\}$ is open and connected and $U\in C^\infty(\mathcal{Y})$;
 		\item[(U2)] For any $R>0$, the open set
 		$$  U_R = \{ x \in \RR^n : U(x) < R \}$$
 		has a compact closure; 
 		\item[(U3)]  $e^{-U(x)} \in L^1(\mathcal{Y}, \, dx)$. 
 	\end{assump}
 	
 	Defining $\mathcal{X}= \mathcal{Y} \times \RR^n$, it follows from Assumption~\ref{assump:Ug} that  for fixed $\gamma \in (0,1)$, the Markov process $\{ z_t^\gamma\}_{t\geq 0}$ defined by~\eqref{eqn:LD} is non-explosive.  Furthermore, by H\"{o}rmander's theorem~\cite{Hor_67}, the associated transition measures $\{ \mathcal{P}_t^\gamma(z, \, \cdot \,) \}_{t>0, z\in \mathcal{X}}$ are absolutely continuous with respect to Lebesgue measure on $\mathcal{X}$ with a smooth transition density; that is, for all $t>0$, $z\in \mathcal{X}$, and $\gamma\in (0,1)$,
 	\begin{align*}
 		\mathcal{P}_t^\gamma(z, \, dw) = p_t^\gamma(z,w) \, dw,
 	\end{align*}
 	where $p_t^\gamma (z, \cdot) \in L^1(\mathcal{X})$ and $(t,z,w) \mapsto p_t^\gamma(z, w) \in C^\infty((0, \infty) \times \mathcal{X}^2)$.  For $z=(x,v)$, setting 
 	\begin{align}
 		H(z)= H(x,v) = \frac{|v|^2}{2}+ U(x) \qquad \text{ and } \qquad H_R = \{ z \in \mathcal{X} \, : \, H(z) < R \},
 	\end{align}
 	we will prove:
 	\begin{theorem}[Quantitative minorization for Langevin dynamics]
 		\label{thm:LD}
 		Suppose that the potential $U:\RR^n\rightarrow [0, + \infty]$ satisfies Assumption~\ref{assump:Ug}.  Then, for every $R>0$ large enough but independent of $\gamma \in (0,1)$, there exist $t_0>0$ and $\lambda >0$, both also independent of $\gamma \in (0,1)$, such that 
 		\begin{align}
 			\label{eqn:thmminor}
 			\inf_{z,w \in H_R} p_{\frac{t_0}{\gamma}}^\gamma(z,w) \geq \lambda
 		\end{align}
 		for all $\gamma \in (0,1)$. 
 	\end{theorem}
 	
 	Theorem~\ref{thm:LD} gives quantitative lower bounds for the transition density for Langevin dynamics~\eqref{eqn:LD} over Hamiltonian sublevel sets $H_R$, $R\gg 1$, on the time scale $\gamma^{-1}$. Using the appropriate Lyapunov structure as in~\cite{CGHS_21} with respect to $0 < \gamma \ll 1$, Theorem~\ref{thm:LD} can be used to extract an optimal quantitative, geometric rate of convergence to equilibrium in a weighted total variation distance on the $\gamma^{-1}$ time scale, provided the potential $U$ satisfies further, natural structural assumptions~\cite{HerMat_19}.  
 	
 	\begin{remark}\label{rem:shearing}
 		The $\gamma^{-1}$ timescale is optimal for \eqref{eqn:LD} in the sense of convergence to global equilibrium because solutions are limited to crossing Hamiltonian level sets on a slow timescale dictated by the noise strength.  However, \eqref{eqn:LD} also contains a fast, $\gamma$-independent timescale on which solutions explore a neighborhood of a fixed level set, and it is quite possible that transition densities first relax on a timescale much shorter than $\gamma^{-1}$ to a function of $H$ before their slow evolution to global equilibrium. As it relates to Theorem~\ref{thm:LD}, this would be linked to proving strict positivity of the density $p_t^\gamma$ for $t \ll o(\gamma^{-1})$ over Hamiltonian \textit{level} sets, rather than sublevel sets. Unfortunately, our proof of Theorem~\ref{thm:LD} gives us no such information and it remains an interesting direction of future research. Enhanced convergence rates to metastable states is well understood in some simple two-dimensional Hamiltonian settings \cite{Shear1, Shear2, Shear3, Shear4, Shear5, Shear6}, but remain widely open for the high-dimensional systems we study here. 
 	\end{remark}
 	
 	While the system~\eqref{eqn:LD} has a unique explicit stationary distribution $\mu$ on Borel subsets of $\mathcal{X}$ given by
 	\begin{align}
 		\mu(dx \, dv ) \propto \exp(-H(x,v)) \, dx\, dv,
 	\end{align}    
 	we will see below that our main result (Theorem~\ref{thm:minorization}), though at least partially motivated by the structure of Langevin dynamics~\eqref{eqn:LD}, does not require knowing the specific form of the stationary measure. In particular, we can just as easily handle the case of Langevin dynamics with anisotropic noise discussed earlier (see Remark~\ref{rem:heatbathes} for more details). See also Example~\ref{ex:osc} below, which treats a chain of strongly coupled oscillators with ends held at different temperatures. In fact, our approach does not even require the existence of a normalizable invariant measure.  As a concrete example, we will produce the same results for degenerately-damped, fluid-like models where, although one might expect the existence of a stationary distribution, it is not yet known rigorously whether one exists.  See Examples~\ref{ex:Lor}-\ref{ex:NSE} below. 
 	
 	 Just as in the case of \eqref{eqn:LD} above, in all of the examples we consider in this paper, making use of a suitable Lyapunov structure (see Assumption~\ref{assump:Lp} and Theorem~\ref{thm:main} below) extends the strict positivity we prove to a sharp exponential convergence rate to equilibrium in a weighted total variation norm. In this way, our results bridge the gap between the quantitative strength of the functional-analytic approaches to convergence and the flexibility of probabilistic methods based on minorization.

 	\subsection{Method of proof}  To produce a \emph{qualitative} versus a \emph{quantitative} lower bound on the transition density, one typically uses continuity of the transition density coupled with positivity of the density at all points (see, for example, \cite{RB_06, GHHM_18, HerMat_15}).  In the case of Langevin dynamics~\eqref{eqn:LD}, it can be shown that for every $t>0$, $\gamma \in (0,1)$ and $z,w\in \mathcal{X}$,
 	\begin{align}
 		\label{eqn:LDexist}
 		p_t^\gamma(z,w) >0.
 	\end{align} 
 	Using the fact that for each $t>0$, $p_t^\gamma(\cdot,\cdot) \in C^\infty( \mathcal{X}^2)$, we can then use compactness of $\overline{H_R}$ to conclude the existence of $\lambda_{\gamma,R} >0$ such that
 	\begin{align}
 		\label{eqn:LDexist1}
 		\inf_{z,w \in H_R} p_{t/\gamma}^\gamma(z,w)  \geq \lambda_{\gamma,R}
 	\end{align}
 	for all $\gamma \in (0,1)$ and $R\gg1$. However, making the existential bound~\eqref{eqn:LDexist1} more quantitative with respect to $\gamma$ using the usual methods from control theory via the Support Theorems~\cite{SV_72ii, SV_72i} seems currently out of reach. This is true despite efforts to make the controls more explicit using scalings~\cite{GHHM_18}.   
 	
 One way around this difficulty is to start from the observation that the dynamics~\eqref{eqn:LD} ``concentrates'' on large enough sublevel sets $H_R$. That is, for $\delta >0$ small but independent of $\gamma \in (0,1)$, one can show
 	\begin{align}
 		\label{eqn:lbmeasure}
 		|A_{t/\gamma}^\delta |:= |\{(z,w) \in H_R^2 \, : \,  p_{t/\gamma}^\gamma(z,w)\geq \delta \}| \geq \delta,
 	\end{align}
 	where $| \cdot |$ denotes Lebesgue measure on $\mathcal{X}^2$.  Although the set $A_{t/\gamma}^\delta$ has little known structure with respect to $\gamma \in (0,1)$, one may hope that combining \eqref{eqn:lbmeasure} with quantitative-in-$\gamma$ spatial regularity of $p_{t/\gamma}^\gamma(\cdot ,\cdot)$ could yield a result like Theorem~\ref{thm:LD}.  While this is not exactly the approach employed in this paper, the principle that positive lower bounds in measure coupled with some form of quantitative regularity can produce uniform positivity is fundamental to the proof of Theorem~\ref{thm:LD}.

 	In the context of quantitative convergence to equilibrium with respect to the total variation distance, the idea above was first noticed and applied successfully to random perturbations of finite-dimensional, fluid-like models with small viscosity $\nu$ and highly degenerate noise~\cite{BL_21}. The key tools required included a quantitative version of H\"{o}rmander's theorem~\cite{Hor_67} and analogues of the classical elliptic theories of De Giorgi~\cite{De_57} and Moser~\cite{Mos_61} generalized to these hypoelliptic, quantitative settings. Hypoelliptic variants of De Giorgi–Nash–Moser theory had already been developed to establish Hölder regularity and Harnack inequalities for kinetic Fokker–Planck equations with rough coefficients~\cite{golse2015h, golse2016harnack}, and many of the arguments in~\cite{BL_21} were adapted from these works. The strategy in~\cite{BL_21} was to, starting from a \textit{time-independent} version of \eqref{eqn:lbmeasure}, first prove a quantitative, pointwise lower bound on compact sets for the stationary density $\rho^\nu$ (which existed due to an assumed Lyapunov structure). This pointwise information was used to prove exponential convergence to equilibrium on the $\nu^{-1}$ timescale in the $L^2(\rho^\nu(x) \, dx)$ norm, which was then upgraded to quantitative total variation convergence using a quantitative Moser theorem ($L^2\implies L^\infty$).

 	While many of the elements present in~\cite{BL_21} are used and generalized in this paper, especially the parameter-dependent H\"{o}rmander theorems (Theorem~\ref{thm:Hssmooth3} and Theorem~\ref{thm:tdhor} below) and the quantitative Moser theorem (Theorem~\ref{thm:moser2}), the path we take must be different because we cannot, under our general set of assumptions in Section~\ref{sec:mainresults}, utilize the invariant measure as it may not exist. The motivation in \cite{BL_21} for using an approach based on studying the stationary measure was that transition densities lack uniform time regularity when rescaled to the natural long timescale. For instance, in \eqref{eqn:LD}, the time-rescaled density $p_{t/\gamma}^\gamma$ moves rapidly along Hamiltonian level sets and satisfies $\partial_t p_{t/\gamma}^\gamma \approx \gamma^{-1}$. The main observation we make that allows us to get around this issue and work directly at the time-dependent level is to consider time-averages of the transition density on the long time scale. For example, in the case of Langevin dynamics we consider
 	\begin{align}
 		\bar{p}_t^\gamma(z,w) :=\frac{1}{t-1}\int_1^t p^\gamma_{s/\gamma}(z,w) \, ds
 	\end{align} 
 	for $t>1$ independent of $\gamma$. The key feature that the average quantity $\bar{p}_t^\gamma$ boasts over $p_{t/\gamma}^\gamma$ is that, by construction, it is just as regular in time as it is in space. In particular, it can be shown that there exists $s > 0$ independent of $\gamma \in (0,1)$ such that $\|\bar{p}^\gamma_{\cdot}\|_{H^{s}_{\mathrm{loc}}((2,\infty)\times\RR^d)}$ is bounded uniformly in $\gamma$, a fact not true for $p_{t/\gamma}^\gamma$.
 	
 	From this point, we will show first that the time-averaged density satisfies a positive uniform lower bound with respect to the parameter over compact sets.  This follows the De Giorgi iteration process~\cite{De_57}, which converts uniform upper bounds to uniform lower bounds. Translating this to the language of Markov processes, the uniform lower bounds on the time averaged density over compact sets imply that all compact sets are quantitatively \emph{petite}.  Abstractly, we then use the time-averaged local quantitative Sobolev control and quantitative uniform upper bounds on the density obtained via Moser iteration to show the claimed uniform lower bounds on the time-$t$ density by adapting the \emph{small set} argument in~\cite{MT_12}.  Importantly, this last step obtains both a quantitative small set and quantitative control on return times to that set, which can then be used to produce a uniform lower bound on the density.          
\subsection{Organization} 
 The organization of the paper is as follows.  In Section~\ref{sec:mainresults}, we set some notation, outline assumptions and state the main result (Theorem~\ref{thm:minorization}) to be proven in the paper.  In Section~\ref{sec:examples}, we apply Theorem~\ref{thm:minorization} to several concrete examples and show, in particular, how Theorem~\ref{thm:LD} follows from Theorem~\ref{thm:minorization}.  Section~\ref{sec:outline} outlines and further explains in detail the steps of the proof of Theorem~\ref{thm:minorization}, which is then carried out in Sections~\ref{sec:smallset}-\ref{sec:lowerb}.  \\
 	
 	\noindent {\bf Acknowledgments:} K.L thanks the National Science Foundation for its support through the RTG grant
 	DMS-2038056 and D.P.H. acknowledges support from grant DMS-2246491 from the National Science Foundation, grant MP-TSM-00002755 from the Simons Foundation and the Duke Mathematics Department during the Fall of 2022, when this work was started. The authors thank Jonathan Mattingly for many fruitful discussions on this topic. 
 	
\section{Notation, Assumptions and Main Results}
\label{sec:mainresults}

In this paper, we study a general class of stochastic differential equations (SDEs) of the following form: 
\begin{align}
\label{eqn:SDEmain}
    \begin{cases}
        dx_t^\epsilon = [-\epsilon Z(x_t^\epsilon) + Z_0(x_t^\epsilon)] \, dt + \sqrt{2\epsilon} \sum_{j=1}^r Z_j(x_t^\epsilon) \circ \, dB_t^j\\
        x_0^\epsilon= x\in \mathcal{X},
    \end{cases}
\end{align}
where $\epsilon \in (0,1)$ is a small parameter, $Z, Z_0, \ldots, Z_r$ are $C^\infty$ vector fields defined on an open subset $\mathcal{X}$ of $\RR^d$, and $B^1_t, \ldots, B_t^r$ are standard, independent real-valued Brownian motions defined on a probability space $(\Omega, \mathcal{F}, \PP, \EE)$.  The $\circ$ in~\eqref{eqn:SDEmain} indicates the Stratonovich convention of the stochastic integral, which is made merely for convenience in relating the dynamics in~\eqref{eqn:SDEmain} to its infinitesimal generator $L_\epsilon$, which in this case has the form
\begin{align}
\label{eqn:Ldef}
L_\epsilon= -\epsilon Z + Z_0 + \epsilon \sum_{j=1}^r Z_j^2 
\end{align}
when acting on sufficiently smooth test functions.  

Note that in~\eqref{eqn:SDEmain}, \eqref{eqn:Ldef} and in the rest of the paper, if $T_d(\mathcal{O})$ denotes the space of $C^\infty$ vector fields on an open subset $\mathcal{O}\subset \RR^d$, we often express $X\in T_d(\mathcal{O})$ in coordinates as 
 \begin{align}
 \label{eqn:vf}
     X= \sum_{j=1}^d X^j \frac{\partial}{\partial x^j}, 
 \end{align}
 where $X^j \in C^\infty(\mathcal{O})$, $j=1,2,\ldots,d$. 
On the other hand, if $x\in \mathcal{O}$, we adopt the convention that $X(x)$ denotes the column vector
 \begin{align}
     X(x)= (X^1(x), \ldots, X^d(x))^T \in \RR^d, 
 \end{align}
 where the $X^i$ are as in~\eqref{eqn:vf}.
 Thus, in particular, $Z(x_t^\epsilon), Z_0(x_t^\epsilon), \ldots, Z_r(x_t^\epsilon)$ in~\eqref{eqn:SDEmain} are the column vectors $Z(x), Z_0(x), \ldots, Z_r(x)$ evaluated at $x=x_t^\epsilon$, while $L_\epsilon$ in~\eqref{eqn:Ldef} is a second-order differential operator.

 For any $X,Y\in T_d(\mathcal{O})$, we let $\text{div}(X)\in C^\infty(\mathcal{O})$ denote the divergence of $X$ and $[X,Y]\in T_d(\mathcal{O})$ denote their commutator; that is, 
 \begin{align}
 \text{div}(X)&= \text{div}\bigg(\sum_{j=1}^d X^j \frac{\partial}{\partial x_j}\bigg) := \sum_{j=1}^d \frac{\partial X^j}{\partial x_j},\\
     \label{eqn:liebracket} [X,Y] &:= \sum_{j=1}^d \bigg(\sum_{k=1}^d X^k\frac{\partial Y^j}{\partial x^k} -Y^k \frac{\partial X^j}{\partial x^k} \bigg) \frac{\partial}{\partial x^j}.
 \end{align}
 
 Fundamental to the structure of the vector fields $-\epsilon Z+ Z_0, Z_1, \ldots, Z_r$ in equation~\eqref{eqn:SDEmain} is a H\"{o}rmander bracket spanning condition satisfied in a uniform sense with respect to the parameter $\epsilon \in  (0,1)$.  In order to describe it, for $\mathcal{O}\subset \RR^d$ open, we let $S_{F}(\mathcal{O})$ denote the set of $u_\epsilon \in C^\infty(\mathcal{O})$ parameterized by $\epsilon \in (0,1)$ such that for every multi-index $\alpha$, there exists a constant $C_\alpha >0$ such that 
 \begin{align*}
\sup_{\epsilon \in (0,1)} \|D^\alpha u_\epsilon \|_{L^\infty(\mathcal{O})} \leq C_\alpha.
 \end{align*} 
 We let $S_d(\mathcal{O})$ denote the set of parameterized vector fields $X_\epsilon \in T_d(\mathcal{O})$ with coefficients belonging to $S_{F}(\mathcal{O})$.  Observe that the space $S_d(\mathcal{O})$ is closed under the Lie bracket.   
\begin{definition}
 \label{def:unifhor}
     Suppose that $\{X_{0, \epsilon}; X_{1, \epsilon}; \ldots; X_{k, \epsilon} \}$ is an ordered list of parameterized vector fields contained in $S_d(\mathcal{O})$.  We say that the family $\{X_{0, \epsilon}; X_{1, \epsilon}; \ldots; X_{k, \epsilon} \}$ satisfies the \emph{uniform parabolic H\"{o}rmander condition on} $\mathcal{O}$ \emph{with respect to} $\epsilon \in (0,1)$ if the list  
     \begin{align}
     \label{eqn:horlist}
    &\begin{aligned}
        &X_{j_1, \epsilon}, &&\quad j_1\in\{1,2, \ldots, k\}, \\
        &[X_{j_1, \epsilon}, X_{j_2, \epsilon}], &&\quad j_i \in\{ 0,1,\ldots, k\},\\
        &[X_{j_1, \epsilon}, [X_{j_2, \epsilon}, X_{j_3, \epsilon}]], &&\quad j_i \in \{0,1,\ldots, k\},\\
        &\qquad \vdots && \qquad \vdots
    \end{aligned}
\end{align}
contains a finite list $Y_{1, \epsilon}, Y_{2, \epsilon}, \ldots, Y_{\ell, \epsilon} \in S_d(\mathcal{O})$, the length $\ell$  of which is independent of $\epsilon \in (0,1)$, such that for every $j=1,2,\ldots, d$, there exist $\{a_{jk, \epsilon}\}_{k=1}^\ell \subset S_F(\mathcal{O})$ for which 
\begin{align}
e_j= \sum_{k=1}^\ell a_{jk, \epsilon}(x) Y_{k, \epsilon}(x)
\end{align} 
for all $x\in \mathcal{O}$ and $\epsilon \in (0,1)$, where $e_j$ denotes the standard basis vector.  

 \end{definition}

\begin{example}
\label{ex:simple1}
Consider $d=2$ with coordinates $z=(x,v)\in \RR^2$, and set $X_{0, \epsilon}=v\partial_x -  \epsilon v  \partial_v\in S_2(\RR^2)$, $X_{1, \epsilon}= \partial_v\in S_2(\RR^2)$.  Then $[X_{1,\epsilon}, X_{0, \epsilon}]= \partial_x - \epsilon \partial_v$, so $\{X_{0, \epsilon}; X_{1,\epsilon}\}$ satisfies the uniform parabolic H\"{o}rmander condition on any open set $\mathcal{O}\subset \RR^2$.  Indeed, note that 
\begin{align*}
e_1 = [X_{1,\epsilon}, X_{0, \epsilon}](z) + \epsilon X_{1,\epsilon}(z) \qquad \text{ and } \qquad e_2 = X_{1, \epsilon}(z)
\end{align*}  
for any $z\in \RR^2$ and the coefficients $1$ and $\epsilon$ belong to $S_F(\RR^2)$. 
\end{example}

  We will further assume that part of the dynamics~\eqref{eqn:SDEmain} has a conserved quantity $H$ which induces the state space $\mathcal{X}$ of the solution $\{x_t^\epsilon\}_{t\geq  0}$ of~\eqref{eqn:SDEmain}.  In particular, we fix $H:\RR^d\rightarrow [0, +\infty]$ and define subsets $H_R, R>0,$ and $\mathcal{X}$ of $\RR^d$ by
\begin{align}
H_{R}= \{ x\in \RR^d \, : \, H(x) < R \} \qquad \text{ and } \qquad
 \mathcal{X}:= \{ x\in \RR^d \, : \, H(x) < \infty \}.
\end{align}
Throughout, we will make the following assumptions on the function $H$ and the vector fields $Z, Z_0, Z_1, \ldots, Z_r$. 
\begin{assump}
\label{assump:H} ~~~~~
\begin{itemize}    
\item[(H1)]  $H$ is independent of $\epsilon \in (0,1)$. 
    \item[(H2)] For every $R>0$, $H_{R}$ is open and $\overline{H_{R}}$ is compact. 
    \item[(H3)] $H$ is a smooth function on the set where it is finite; that is, $H\in C^\infty(\mathcal{X})$.
    \item[(H4)]  $\mathcal{X}$ is connected.
    \end{itemize}
\end{assump}

\begin{assump}
\label{assump:V}
~~~
\begin{itemize}
    \item[(V1)] $Z, Z_0, Z_1, Z_2, \ldots, Z_r \in T_d(\mathcal{X})$ are independent of $\epsilon \in (0,1)$.
    \item[(V2)] $\text{div}(Z_0)=\text{div}(Z_1)=\cdots =\text{div}(Z_r)=0$ on $\mathcal{X}$ and $\frac{1}{2}\| \text{div}(Z) \|_{L^\infty(\mathcal{X})}< \inf_{x\in \mathcal{X}} \text{div}(Z)$. 
    \item[(V3)] $Z_0$ conserves $H$; that is, $Z_0 H=0$ on $\mathcal{X}$.   
    \item[(V4) ]  There exists $\eta, d_*>0$ independent of $\epsilon \in (0,1)$ such that    
    \begin{align*}
   &e^{-\eta H} \textstyle{\sum_{j=1}^r} (|Z_j H|^2 + |Z_j^2 H|+1) \in L^1(\mathcal{X}, dx)   \quad \text{ and }\\
   & \quad  \eta \textstyle{\sum_{j=1}^r (Z_j H)^2} \leq ZH +\textstyle{ \sum_{j=1}^r} Z_j^2 H \leq ZH + d_*/2 .
    \end{align*}    
 
    \item[(V5)] For every $R>0$, the ordered collection $\{ -\epsilon Z+ Z_{0}; Z_1; \ldots; Z_r\}$ satisfies the uniform parabolic H\"{o}rmander condition with respect to $\epsilon \in (0,1)$ on $H_{R}$.   
\end{itemize}
\end{assump}
\begin{remark}
    We allow for $H$ to take the value $+\infty$ in order to treat conserved quantities containing singularities, such as Lennard--Jones or Coloumb interaction potentials in the case of randomly perturbed Hamiltonian systems, where $H$ is the Hamiltonian.  In the context of finite-dimensional models in fluid mechanics, e.g. finite-dimensional Galerkin projections of the Navier-Stokes' equations, we can take $H(x)= |x|^2$ and $\mathcal{X}=\RR^d$.  In this case, $H_{R}$ is simply the ball of radius $\sqrt{R}>0$ centered at the origin.    
\end{remark}

 \begin{remark}
 Observe that Assumption~\ref{assump:V} (V2) implies $\inf_{x\in \mathcal{X}} \text{div}(Z)>0$.  Thus, when considering equation~\eqref{eqn:SDEmain}, $-Z(x_t^\epsilon)$ is a damping term.  This is also reflected in the inequalities in Assumption~\ref{assump:V} (V4) as explained in the next remark.  
 \end{remark}
 
 \begin{remark}
 \label{rem:subinv}
 Perhaps the least obvious assumption in the list (V1)-(V5) is condition (V4).  The first condition in (V4) is an integrability condition on the measure $e^{-\eta H} dx$. The inequality 
 \begin{align}
 \label{eqn:subinvs}
 \eta \textstyle{\sum_{j=1}^r (Z_j H)^2} \leq ZH +\textstyle{ \sum_{j=1}^r} Z_j^2 H
 \end{align}
  translates to $e^{-\eta H} \, dx$ being `close' to a subinvariant measure for the Markov process $\{ x_t^\epsilon \}_{t\geq 0}$.  Indeed, note that if $L^*_\epsilon$ denotes the formal $L^2$ adjoint of $L_\epsilon$, then by equation~\eqref{eqn:subinvs} and Assumption~\ref{assump:V} (V1)-(V3) we have 
 \begin{align*}
 L^*_\epsilon (e^{-\eta H}) =\epsilon (  \text{div}(Z) + Z +  \textstyle{\sum_{j=1}^r} Z_j^2)e^{-\eta H} \leq \epsilon \text{div}(Z) e^{-\eta H}.   
 \end{align*}  
 In the context of random perturbations of Hamiltonian dynamics, such a measure (for some $\eta >0$) can be an invariant measure.  However, in the context of various fluid models, the form of invariant measures is often not known. Nevertheless, one still has such `subinvariance' structure.  
 
 Finally, the global inequalities
  \begin{align}
 \label{eqn:subinvs2}
 \eta \textstyle{\sum_{j=1}^r (Z_j H)^2} \leq ZH +\textstyle{ \sum_{j=1}^r} Z_j^2 H \leq ZH + d_*/2
 \end{align} 
together induce basic Lyapunov structure, giving strong existence and uniquess of equation~\eqref{eqn:SDEmain}.  Indeed, taking $V=H+1$, we observe that by~\eqref{eqn:subinvs2}  and Assumption~\ref{assump:V} (V1), (V3) 
 \begin{align}
 L_\epsilon V = -\epsilon Z H + \epsilon \textstyle{\sum_{j=1}^r} Z_j^2 H \leq 2 \epsilon  \textstyle{\sum_{j=1}^r} Z_j^2 H \leq d_* \epsilon. \end{align}
 It follows by this bound Assumption~\ref{assump:H} that $\{ x_t^\epsilon \}_{t\geq 0}$ is non-explosive.  That is, if 
\begin{align*}
\tau_n^\epsilon = \inf \{ t \geq 0 \, : \, x_t^\epsilon \notin H_{n } \}
\end{align*}  
and $\tau^\epsilon = \lim_{n\rightarrow \infty} \tau_n^\epsilon$, then for every $x\in \mathcal{X}$ it follows that
\begin{align}
\PP_x \{\tau^\epsilon < \infty \}=0.
\end{align}

 \end{remark}

In addition to the solution of~\eqref{eqn:SDEmain} being non-explosive, Assumptions~\ref{assump:H}-\ref{assump:V} furthermore imply that the solution $\{x_t^\epsilon\}_{t\geq 0}$ is a Markov process with corresponding Markov semigroup, which we denote by $\{ \mathcal{P}^\epsilon_t\}_{t\geq 0}$.  The semigroup $\{ \mathcal{P}_t^\epsilon\}_{t\geq 0}$ acts on bounded, Borel measurable measurable functions $\phi: \mathcal{X}\rightarrow \RR$ by
\begin{align}
\mathcal{P}_t^\epsilon \phi(x) :=\EE_x \phi(x_t^\epsilon) 
\end{align}
and dually on probability measures $\nu$ defined on the Borel subsets $\mathcal{B}(\mathcal{X})$ of $\mathcal{X}$ by
\begin{align}
 \nu\mathcal{P}_t^\epsilon (A) = \int_\mathcal{X} \nu(dx) \mathcal{P}_t^\epsilon \mathbf{1}_A(x), \,\,\, A \in \mathcal{B}(\mathcal{X}). 
\end{align}  
We use the notation 
\begin{align}
\mathcal{P}_t^\epsilon(x, A)= \mathcal{P}_t^\epsilon \mathbf{1}_{A}(x), \,\,\, x\in \mathcal{X},\,\,\, A \in \mathcal{B}(\mathcal{X}),
\end{align}
to denote the Markov transitions of $\{x_t^\epsilon\}_{t\geq 0}$.  

Note also that, under Assumptions~\ref{assump:H}-\ref{assump:V} (especially Assumption~\ref{assump:V} (V5)) for each $x\in \mathcal{X}$ and $t>0$, the measure $\mathcal{P}_t^\epsilon(x,\,\cdot \,)$ is absolutely continuous with respect to Lebesgue measure on $\mathcal{X}$ with a smooth transition density $p_{t}^\epsilon(x,y)$; that is, for each $t>0$ and $x\in \mathcal{X}$
\begin{align}
\mathcal{P}_t^\epsilon(x, dy) = p_{t}^\epsilon(x,y) \, dy
\end{align}
and $(t,x,y) \mapsto p_{t}^\epsilon(x,y) \in C^\infty((0, \infty) \times \mathcal{X}^2)$.  This follows from an application of H\"{o}rmander's hypoellipticity theorem~\cite{Hor_67} as in~\cite{RB_06}.

In addition to the assumptions above, we will also need to know some behavior about $\{x_t^\epsilon\}_{t\geq 0}$ near certain boundaries in space.  This employs the notion of \emph{regular points} of diffusions, which we now describe in detail in the context of a general, path-continuous stochastic process $\Xi=( X_t)_{t\geq 0}$ with state space $\RR^d$ defined on $(\Omega, \mathcal{F}, \PP, \EE)$.  This concept is important as it will allow us to solve various boundary value PDEs related to the process $\{ x_t^\epsilon\}_{t\geq 0}$ in the classical sense.  

Suppose that $\mathcal{O}\subset \RR^d$ is a bounded open set with continuous boundary $\partial \mathcal{O}$.  Define the first positive exit time of $\Xi$ from the closure $\overline{\mathcal{O}}$:   
\begin{align}
\tau_{\overline{\mathcal{O}}}^{\Xi}= \inf \{ t >0 \, : \, X_t \notin \overline{\mathcal{O}}\}.
\end{align}      
We call the set $\mathcal{O}$ \emph{boundary regular} for $\Xi$ if for all $x_* \in \partial \mathcal{O}$
\begin{align}
\PP_{x_*} \big\{ \tau_{\overline{\mathcal{O}}}^{\Xi} > 0 \big\} =0.
\end{align}
That is, $\mathcal{O}$ is boundary regular for $\Xi$ if the process instantaneously leaves the closure $\overline{\mathcal{O}}$ when started at any boundary point. 

\begin{remark}
\label{rem:BR}
There are very simple examples of hypoelliptic diffusions with domains $\mathcal{O}$ which are not boundary regular.   We provide here the same example as in~\cite{Ok_13}. Consider the open square $\mathcal{O}=(-1,1)\times (-1,1)$ and the diffusion $(X_t, Y_t)$ on $\RR^2$ given by 
\begin{align}
\label{eqn:BRproc}
dX_t &= Y_t^2 \, dt \\
\nonumber dY_t& = \, dB_t
\end{align}   
where $B_t$ is a standard, one-dimensional Brownian motion.  Then it is easy to check that $(X_t, Y_t)$ is hypoelliptic (i.e. has a smooth transition density on all of $(0, \infty) \times \RR^2 \times \RR^2$) using H\"{o}rmander's theorem~\cite{Hor_67}, but any point $(x_*, y_*)\in  \{-1\}\times [-1/2, 1/2]$ is not regular for the process $\{(X_t, Y_t) \}_{t\geq 0}$ since, when started at $(x_*, y_*)\in\{-1\}\times [-1/2, 1/2]$ the process immediately enters the interior of the domain $\mathcal{O}$, spending a positive amount of time in $\mathcal{O}$ prior to exiting $\overline{\mathcal{O}}$.  \end{remark}

\begin{remark}
In order to obtain classical existence and uniqueness of the associated Dirichlet or Poisson problems on a bounded domain $\mathcal{O}\subset \RR^d$ with, for example, operator $L_\epsilon$ as in~\eqref{eqn:Ldef}, it is critical to know that $\mathcal{O}$ is boundary regular for $\{x_t^\epsilon\}_{t \geq 0}$.  Importantly, being boundary regular implies both uniqueness and that the solution is continuous up to the boundary $\partial \mathcal{O}$.  See, for example, \cite{FH_23, Ok_13}.   
\end{remark}

In addition to the process $\{x_t^\epsilon \}_{t\geq 0}$, we will also use the process which has its drift time-reversed. That is, we consider the solution $x_t^{-, \epsilon}$ of the following Stratonovich SDE 
\begin{align}
\begin{cases}
dx_t^{-, \epsilon} = [\epsilon Z(x_t^{-, \epsilon} )- Z_0(x_t^{-, \epsilon} )] \, dt + \sqrt{2\epsilon} \textstyle{ \sum_{j=1}^r Z_j(x_t^{-, \epsilon}) }\circ dB_t^j \\
x_0^{-, \epsilon}= x\in \mathcal{X}. 
\end{cases}
\end{align}
Note that $x_t^{-, \epsilon}$ is Markovian up until its time of explosion, which may be finite with positive probability even under our assumptions.  To deal with this, if 
\begin{align}
\tau_n^{-, \epsilon} =\inf\{ t\geq 0 \, : \, x_t^{-, \epsilon} \notin H_n \}\qquad \text{ and } \qquad \tau^{-, \epsilon} = \lim_{n\rightarrow \infty} \tau_n^{-, \epsilon},
\end{align}
we fix $\Delta \notin \mathcal{X}$ and set $x_{t}^{-, \epsilon}= \Delta $ for $t\geq \tau^{-, \epsilon}$.  Because being boundary regular is a local concept, it makes sense for the process $\{x_{t}^{-, \epsilon} \}_{t\geq 0}$ which is defined on $\mathcal{X} \cup \{ \Delta \}$ but path continuous up until $\tau^{-, \epsilon}$.    

\begin{assump}
\label{assump:BR}
Let $\{z_t^\epsilon\}_{t\geq 0} \in \{ \{x_t^\epsilon\}_{t\geq 0}, \{x_t^{-, \epsilon}\}_{t\geq 0} \}$.  For every $R>0$, there exists a bounded open set $\mathcal{O}_{R}$  with continuous boundary $\partial \mathcal{O}_{R}$ such that $$H_R\subset \mathcal{O}_{R}\subset \overline{\mathcal{O}_{R}} \subset \mathcal{X}$$ and such that $\mathcal{O}_{R}$ is boundary regular for $\{ z_t^\epsilon\}_{t\geq 0}$.  
\end{assump}

\begin{remark}
\label{rem:BR2}
Let us now revisit the example in Remark~\ref{rem:BR} to see how Assumption~\ref{assump:BR} is relevant.  Indeed, we saw that $\mathcal{O}=(-1,1) \times (-1,1)$ is not boundary regular for~\eqref{eqn:BRproc}.  However, take any ball $B$ containing $(-2,2) \times (-2,2)$ and select a sequence of points on the boundary $x_1,x_2, \ldots, x_k\in \partial B$ so that the closed convex hull of $x_1, x_2, \ldots, x_k$, denoted by $[x_1, x_2, \ldots, x_k]$, satisfies $\text{interior}([x_1, \ldots, x_k]) \supset \mathcal{O}$ and such that each face on $\partial[x_1, \ldots, x_k]$ is not vertical.  We will see below in Section~\ref{sec:examples} how to, more precisely, construct such domains.  Then it can be shown that $\text{interior}([x_1, \ldots, x_k])$ is boundary regular for \eqref{eqn:BRproc}~\cite[Corollary 7.10]{CFH_22}.  This is intuitively because the noise in the $Y$ direction forces the process to immediately exit the domain when started on the constructed boundary.   
\end{remark}

\begin{remark}
Assumption~\ref{assump:BR} is perhaps the most technical condition we employ.  However, we will see how to exhibit the domains $\mathcal{O}_{R}$ which are boundary regular for $\{z_t^\epsilon \}_{t\geq 0}\in \{ \{x_t^\epsilon\}_{t\geq 0}, \{x_t^{-, \epsilon}\}_{t\geq 0} \}$.  In the case of Langevin dynamics~\eqref{eqn:LD}, we will see that every level set $H_R$, for $R\gg 1$, is boundary regular for the process $\{z_t^\epsilon \}_{t\geq 0}$.  This turns out to be true under Assumption~\ref{assump:Ug} by following a slight generalization of the results in~\cite{FH_24}.  On the other hand, in the context of hypoelliptic SDE models in fluids, these domains can be constructed by taking the convex hull of a sufficient number of randomly selected points on the boundary of a ball similar to Remark~\ref{rem:BR2}. This, in turn, produces a polyhedral domain whose faces, with high probability, are transverse with the span of the noise.  More precisely, if $H_R$ is the ball of radius $\sqrt{R}>0$ centered at the origin in $\RR^d$ and $0<R<R'< \infty$, selecting a sufficient number of points $\xi_1, \xi_2, \ldots, \xi_k$ independently and randomly according to Hausdorff measure on $\partial H_{R'}$, we can ensure that, with high probability, 
\begin{align}
H_R \subset \text{interior}([\xi_1, \ldots, \xi_k]) 
\end{align}       
where $[\xi_1, \ldots, \xi_k]$ is the closed convex hull of $\xi_1, \ldots, \xi_k$.  Furthermore, with high probability, we can ensure that every face on $\partial[\xi_1, \ldots, \xi_k]$ is not parallel with a fixed constant direction, say $v\in \RR^d$.  See Section~\ref{sec:examples} for further details of these points.  If $v\in \RR^d$ is in $\text{span}\{ Z_1(x_*), \ldots, Z_r(x_*)\}$ for every $x_* \in \partial  [\xi_1, \ldots, \xi_k]$, then it follows by picking a convenient realization that $\text{interior}([\xi_1, \ldots, \xi_k])$ is boundary regular for $\{z_t^\epsilon\}_{t\geq 0}$~\cite[Corollary 7.10]{CFH_22}.    
\end{remark}

\begin{remark}
In order to get around the issues in the hypoelliptic setting with classical existence and uniqueness of various Poisson problems, an elliptic regularization is employed in~\cite{BL_21} with parameter $\delta >0$.  The dependence on various estimates with respect to $\delta >0$ has to be tracked and later removed.  Below, we do not employ an elliptic regularization.  Thus one can view Assumption~\ref{assump:BR} as an alternative approach such regularization.   
\end{remark}

Given the assumptions above, we can now state the main result to be proven in this paper.

 \begin{theorem}[Quantitative minorization]
\label{thm:minorization}
Suppose that Assumptions~\ref{assump:H}-\ref{assump:BR} are satisfied.  For every $R>0$ sufficiently large, there exists $t_0>0$ and $\lambda >0$ independent of $\epsilon \in (0,1)$ such that 
\begin{align}
\label{eqn:thmminor}
\inf_{\epsilon \in (0,1)}\inf_{x,y \in H_R} p_{t_0/\epsilon}^\epsilon(x,y) \geq \lambda.
\end{align}
\end{theorem}

As a corollary of Theorem~\ref{thm:minorization}, if the dynamics has a normalizable invariant measure, we can also discuss its rate of convergence to equilibrium quantitatively with respect to $\epsilon \in (0,1)$, provided we upgrade the basic Lyapunov structure ensured by Assumption~\ref{assump:V}.  Below, for simplicity, we use a Lyapunov function which ensures an exponential rate of convergence (on the $1/\epsilon$ timescale), but one can readily use Theorem~\ref{thm:minorization} coupled with slightly weaker Lyapunov bounds to ensure quantitative subexponential rates of convergence to equilibrium.  See, for example, \cite{DFG_09, Hai_10} and the references therein.   

Recall that a Borel measure $\nu_\epsilon$ is called \emph{invariant} for $\{ \mathcal{P}_t^\epsilon \}_{t\geq 0}$ if for all $A\in \mathcal{B}(\mathcal{X})$ 
\begin{align}
\nu_\epsilon \mathcal{P}_t^\epsilon (A):= \int_\mathcal{X} \nu_\epsilon(dx) \mathcal{P}_t^\epsilon(x, A) = \nu_\epsilon(A).
\end{align} 
We note that, in the above, $\epsilon \in (0,1)$ is fixed.  An invariant measure $\nu_\epsilon$ for $\{ \mathcal{P}_t^\epsilon \}_{t\geq 0}$ is called an \emph{invariant probability measure} if $\nu_\epsilon(\mathcal{X})=1$. 

\begin{assump}
\label{assump:Lp}
    There exists $V_\epsilon\in C^2(\mathcal{X}; [1,\infty))$, $\epsilon \in (0,1)$, such that:
    \begin{itemize}
        \item[(LF1)] For every $R>0$, there exist constants $v_R>0$ and $V_R$ independent of $\epsilon \in (0,1)$ such that $v_R\rightarrow \infty$ as $R\rightarrow \infty$ and 
        \begin{align*}
       v_R \leq  \inf_{x\in H_R^c} \, V_\epsilon(x)\leq  \sup_{x\in H_R} V_\epsilon(x) \leq V_R . 
        \end{align*}
        \item[(LF2)]  For some $d_1, d_2, R>0$ independent of $\epsilon \in (0,1)$, we have the global bound
        \begin{align*}
L_\epsilon V_\epsilon(x) \leq -\epsilon  d_1 V_\epsilon + d_2 \epsilon \mathbf{1}_{H_R}
        \end{align*}
     on $\mathcal{X}$. 
    \end{itemize}\end{assump}

\begin{theorem}
\label{thm:main}
   Suppose that Assumptions~\ref{assump:H}-\ref{assump:Lp} are satisfied.  Then there exists a unique invariant probability measure $\mu^\epsilon$ for $\{\mathcal{P}_t^\epsilon\}_{t\geq 0}$.  Furthermore, there exist constants $c, C>0$ independent of $\epsilon \in (0,1)$ and $x\in \mathcal{X}$ such that 
   \begin{align}
   \sup_{|\phi| \leq V_\epsilon} \bigg| \mathcal{P}_{t}^\epsilon \phi(x) -  \mu^\epsilon(\phi) \bigg| \leq C V_\epsilon(x) e^{- c \epsilon t }
   \end{align}   
   for all $t\geq 0$, $x\in \mathcal{X}$.  In the above, $V_\epsilon$ is as in Assumption~\ref{assump:Lp} and $\mu^\epsilon(\phi) = \int_\mathcal{X} \phi(x) \mu^\epsilon(dx)$.   
\end{theorem}

The proof of Theorem~\ref{thm:main} follows readily from the results in \cite{Hai_10, HM_11} when combined with Theorem~\ref{thm:minorization}.  Our focus in this paper is on the proof of Theorem~\ref{thm:minorization}.

 \section{Examples}
 \label{sec:examples}
In this section, we provide concrete examples of dynamics which satisfy Assumptions~\ref{assump:H}-\ref{assump:BR}.  For simplicity, all of our examples consider equations of the form~\eqref{eqn:SDEmain} with additive noise, i.e. the vector fields $Z_1, Z_2, \ldots, Z_r$ will all be constant below. This simplifying assumption allows us to easily connect with existing results in the literature.  Similar ideas used in the additive noise setting can be adapted to the state-depending noise setting to produce similar properties, in particular Assumption~\ref{assump:BR}, but we avoid these technicalities to keep the presentation more concise.    

\begin{example}[Langevin dynamics]
\label{ex:LD}
We return to the example presented in the introduction, namely Langevin dynamics~\eqref{eqn:LD} under Assumption~\ref{assump:Ug}.  Our goal here is to show Theorem~\ref{thm:LD} by validating the hypotheses of the general result, Theorem~\ref{thm:minorization}. 

First observe that \eqref{eqn:LD} is of the form~\eqref{eqn:SDEmain} with $d=2n$, $r=n$, $\gamma=\epsilon$ and 
\begin{align}
\label{eqn:vfLD}
Z = v \cdot \nabla_v, \qquad Z_0 =v\cdot \nabla_x -\nabla U(x) \cdot \nabla_v \quad \text{ and }\quad Z_j= \frac{\partial}{\partial v_j} , \,\,\, j=1,2,\ldots, n.
\end{align}

We will prove the following result which implies Theorem~\ref{thm:LD} by way of Theorem~\ref{thm:minorization}.
\begin{theorem}
\label{thm:LD1}
Suppose that Asssumption~\ref{assump:Ug} is satisfied.  Let $d=2n$, $r=n$ and suppose that the vector fields $Z, Z_0, Z_1, \ldots, Z_n$ are as defined in~\eqref{eqn:vfLD}.  Then Assumptions~\ref{assump:H}-\ref{assump:BR} are satisfied.   
\end{theorem} 

While most of the proof of Theorem~\ref{thm:LD1} is relatively straightforward, checking Assumption~\ref{assump:BR} requires a little more work.  However, much of this was established recently in \cite[Theorem 8.4]{FH_24}.   
  \begin{lemma}
  \label{lem:LD}
Suppose that Asssumption~\ref{assump:Ug} is satisfied.  Let $d=2n$, $r=n$ and suppose that the vector fields $Z, Z_0, Z_1, \ldots, Z_n$ are as in~\eqref{eqn:vfLD}.  Then Assumption~\ref{assump:BR} is satisfied. 
\end{lemma} 

\begin{proof}[Proof of Lemma~\ref{lem:LD}]
Let $R\gg 1$ and $(x_0, v_0 ) \in \partial H_R$.  It is enough to show that $H_R$ is boundary regular for each process $\{ (x_{t}^{\epsilon}, v_t^\epsilon) \}_{t\geq 0}$, $\{ (x_t^{-, \epsilon}, v_t^{-, \epsilon}) \}_{t\geq 0} $.  The only difference between this and \cite[Theorem 8.4]{FH_24} is that in the case when we consider the process $\{ (x_t^{-, \epsilon}, v_t^{-, \epsilon}) \}_{t\geq 0}$, we change the sign of the drift term.  The rest of the proof follows in a nearly identical way.  Note also that the definition of regular is opposite of the one in~\cite{FH_24}, i.e. \emph{entrances} to the domain versus \emph{exits from} the domain.  
\end{proof}

\begin{proof}[Proof of Theorem~\ref{thm:LD}]
By Lemma~\ref{lem:LD}, we have left to check Assumptions~\ref{assump:H}-\ref{assump:V}.  Note first that Assumption 2.1 (U1)-(U3)  and the definition of $H$ immediately imply Assumption~\ref{assump:H}. By construction, Assumption~\ref{assump:V} (V1)-(V3) is also immediate.  To see that Assumption~\ref{assump:V} (V4) is valid, recalling that $H(x,v) = |v|^2/2+ U(x)$, we note that for $\eta =1$
\begin{align*}
&\textstyle{\sum_{j=1}^n} e^{-H}( |Z_j H|^2 + |Z_j^2 H| + 1) = \textstyle{\sum_{j=1}^d} e^{-H}( |v_j|^2  + 2)\in L^1(\mathcal{X}, \, dx)
\end{align*}
and 
\begin{align*}
\sum_{j=1}^n (Z_j H)^2= |v|^2 \leq |v|^2+ n = ZH + \sum_{j=1}^r Z_j^2 H \leq ZH + n.
\end{align*}
To check Assumption~\ref{assump:V} (V5), similar to Example~\ref{ex:simple1} we compute the commutator 
 \begin{align*}
  [ Z_j, -\epsilon Z + Z_0]&= \frac{\partial}{\partial x_j} - \epsilon \frac{\partial}{\partial v_j}, \,\,\,\, j =1,2, \ldots, n. \end{align*}
 From this, it follows by Assumption~\ref{assump:Ug} that the ordered list of vectors 
 \begin{align}
 \label{eqn:vfLD2}
-\epsilon Z + Z_0; Z_1; \ldots ;Z_n,
 \end{align}
 satisfies the uniform parabolic H\"{o}rmander condition on every set $H_{<R}$ for all $R>0$.  Thus Assumption~\ref{assump:V} (V5) is now verified.

\begin{remark} \label{rem:heatbathes}
In the setting above, $\mathcal{N}^{-1} e^{-H(x,v)}$ is the unique invariant density (with respect to Lebesgue measure on $\mathcal{X}$) for the corresponding semigroup $\{ \mathcal{P}_t^\epsilon \}_{t\geq 0}$, where $\mathcal{N} = \int_\mathcal{X} e^{-H(x,v)} \, dx \, dv$.  However, it is worth noting that we need not hold the heat baths (modeled by the Brownian motions) at the same temperature to have the same result.  In particular, the conclusion of Theorem~\ref{thm:LD} still holds true if we instead define 
\begin{align}
Z_j = \sqrt{T_j} \frac{\partial}{\partial v_j}
\end{align}  
for some constants $T_1, \ldots, T_n >0$.  In this case, in general the form of the invariant measure is not known.  However, the only difference in the argument giving Theorem~\ref{thm:LD} is that one needs to pick $0< \eta < \infty$ small enough, but independent of $\epsilon \in (0,1)$ so that Assumption~\ref{assump:V} (V4) is met.  We will see a concrete example of this in the case of oscillators below.  
\end{remark}
\end{proof}

\begin{example}[Chains of oscillators]
\label{ex:osc}
In this example, we consider a similar setup to the previous example, but only place noise and damping on some degrees of freedom but not on others.  Specifically, we consider the case of a chain of oscillators with noise and damping only placed at each end of the chain.  Although we do not do this here, we believe these computations can be extended to graphs of oscillators as in~\cite{CEHR_18} provided a similar set of assumptions is employed so that energy is `effectively transferred' among the degrees of freedom in the equation. 

Consider equation~\eqref{eqn:SDEmain} with $d=2n$, $r=2$ and 
\begin{align}
\label{eqn:oscZZ_0}
Z= \gamma_1 v_1 \frac{\partial}{\partial v_1} + \gamma_n v_n \frac{\partial}{\partial v_n}, \quad Z_0 = v \cdot \nabla_x - \nabla U(x) \cdot \nabla_v, \quad \text{ and } 
\end{align}  
\begin{align}
\label{eqn:oscZ_1Z_2}
Z_1 = \sqrt{\gamma_1 T_1} \frac{\partial}{\partial v_1} , \quad Z_2 =  \sqrt{\gamma_n T_n} \frac{\partial}{\partial v_n}, \end{align}
where $x=(x_1, \ldots, x_n)$, $v=(v_1, \ldots, v_n)$, $\gamma_i >0$, $T_i >0$.  We furthermore assume that $U$ has the form
\begin{align}
\label{eqn:Uosc}
U(x)= \sum_{\ell=1}^n | x_\ell|^{2k} + \sum_{\ell=1}^{n-1} | x_\ell - x_{\ell+1}|^{2j}
\end{align}        
where $j\geq k\geq 1$ are integers, the first term on the righthand side of~\eqref{eqn:Uosc} is the `pinning' part of the potential $U$ and the second on the righthand side of~\eqref{eqn:Uosc} is the interaction part.  Letting
\begin{align}
\label{eqn:oscH}
H(x,v)= \frac{|v|^2}{2}+ U(x),
\end{align}  
we have the following:

\begin{theorem}
\label{thm:osc}
With $d=2n$, $r=2$, $Z, Z_0, Z_1, Z_2$ as above in~\eqref{eqn:oscZZ_0} and~\eqref{eqn:oscZ_1Z_2}, Assumptions~\ref{assump:H}-\ref{assump:BR} are satisfied for the resulting process as in~\eqref{eqn:SDEmain} with $H$ as in~\eqref{eqn:oscH}.  Thus the quantitative minorization condition in Theorem~\ref{thm:minorization} is met. 
\end{theorem}

In this and the remaining examples, we will leave the proof of Assumption~\ref{assump:BR} for the end of the section.  Note that this case is a little different than Lemma~\ref{lem:LD} since noise is only present on two directions.  However, the structure of the potential $U$ is simpler, and we can appeal to a general argument that works in more cases to verify Assumption~\ref{assump:BR}.  
\begin{lemma}
\label{lem:osc}
With $d=2n$, $r=2$, $Z, Z_0, Z_1, Z_2$ as above in~\eqref{eqn:oscZZ_0} and~\eqref{eqn:oscZ_1Z_2}, Assumption~\ref{assump:BR} is satisfied.  
\end{lemma}

\begin{proof}[Proof of Theorem~\ref{thm:osc}]
Applying Lemma~\ref{lem:osc}, we have left to check Assumptions~\ref{assump:H}-\ref{assump:V}.  By construction, Assumption~\ref{assump:H} is met and so is Assumption~\ref{assump:V} (V1)-(V3).  To see Assumption~\ref{assump:V} (V4), observe that for any $\eta >0$ independent of $\epsilon \in (0,1)$ we have  
\begin{align*}
&\sum_{j=1}^2 e^{-\eta H} ( |Z_j H|^2 +|Z_j^2 H|+1) = \sum_{j=1,n} e^{-\eta H}(\gamma_j T_j|v_j|^2+\gamma_j T_j +1) \in L^1(\RR^{d}),
\end{align*}  
and, choosing $0< \eta \leq 1/T_\text{max}$ where $T_\text{max} = \max \{ T_1, T_n \}$, we see
\begin{align*}
\eta \sum_{j=1}^2 (Z_j H)^2 = \eta (\gamma_1 T_1 |v_1|^2+ \gamma_2 T_2 |v_n|^2) \leq Z H \leq ZH + \sum_{j=1}^2 Z_j^2 H\leq Z H + \gamma_1 T_1 + \gamma_2 T_2.
\end{align*}
Thus Assumption~\ref{assump:V} (V4) is checked.

Checking Assumption~\ref{assump:V} (V5) uses an inductive argument.  Let $\tilde{Z}_0=Z_0-\epsilon Z$, $\tilde{Z}_i=Z_i$, $i=1,2$, and $\mathcal{I}$ denote the linear span of the list
\begin{align*}
&\tilde{Z}_1, \tilde{Z}_2 \\
&[\tilde{Z}_\ell,  \tilde{Z}_m], \quad \ell, m\in \{ 0,1,2\}\\
&[\tilde{Z}_{\ell_1}, [\tilde{Z}_{\ell_2}, \tilde{Z}_{\ell_3}]] ,\quad \ell_i \in \{0,1,2\}\\
&\vdots \qquad\qquad  \vdots\qquad \qquad \vdots
\end{align*} 
with coefficients restricted to $S_F(\RR^d)$.  It suffices to show that $\partial_{x_j}, \partial_{v_j} \in \mathcal{I}$ for all $j$.  First notice that if $Y_1= (1/\sqrt{\gamma_1 T_1}) \tilde{Z}_1$ then 
\begin{align*}
Y_1'&:=[Y_1, \tilde{Z}_0] = \partial_{x_1} - \epsilon \gamma_1 \partial_{v_1}\in \mathcal{I},\\
\partial_{x_1} &= Y_1' + \epsilon \gamma_1 Y_1 \in\mathcal{I}.
\end{align*}
Thus $\partial_{x_1}, \partial_{v_1} \in \mathcal{I}$.  Suppose now inductively that $\partial_{x_i}, \partial_{v_i}\in \mathcal{I}$ for all $1\leq i \leq i_0< n-1$.  Then  
\begin{align*}
Y_i':= (\text{ad}^{2j-1} \partial_{x_i})(\tilde{Z}_0) = - [(2k)! \mathbf{1}_{\{j=k\}} + (2j)!]\partial_{v_i} + (2j)! \mathbf{1}_{\{i\geq 2\}} \partial_{v_{i-1}} + (2j)! \partial_{v_{i+1}},
\end{align*}  
where we recall that for smooth vector fields $X$ and $Y$, $\text{ad}^j X(Y)$, $j\geq 0$, is defined inductively by 
\begin{align}
\text{ad}^0 X(Y):=Y \qquad  \text{ and } \qquad \text{ad}^jX(Y):= \text{ad}^{j-1}X([X,Y]),\,\,\, j\geq 1.
\end{align} 
Thus the computation above shows that $\partial_{v_{i+1}} \in \mathcal{I}$, in which case 
\begin{align}
[\partial_{v_{i+1}}, \tilde{Z}_0] = \partial_{x_{i+1}} - \epsilon \gamma_n \mathbf{1}_{\{i+1= n\}} \partial_{v_{i+1}}.  
\end{align}  
It thus follows also that $\partial_{x_{i+1}} \in \mathcal{I}$, and this validates Assumption~\ref{assump:V} (V5) by induction.

\end{proof}

\end{example}

\begin{example}[Lorenz '96]
\label{ex:Lor}
We next consider the stochastic Lorenz '96 system for $x_t^\epsilon=(x_{1}^\epsilon,\ldots, x_d^\epsilon)\in \RR^d$, $d\geq 4$, defined by
\begin{align}
\label{Lor_96}
\begin{cases}
  dx_{i}^\epsilon=-\epsilon\lambda_i x_{i}^\epsilon \, dt+(x_{i+1}^\epsilon-x_{i-2}^\epsilon) x_{i-1}^\epsilon dt +\sqrt{2\epsilon} \sigma_i dB_t^i,\\
 x_i^\epsilon(0)=x_i \in \RR, 
\end{cases}
\end{align}
where the dynamics lives on the discrete circle so that $x_{k}\equiv x_{k+d}$.  Below, we assume that $\lambda_1>0, \lambda_3 >0$, $\sigma_1 >0$, $\sigma_3 >0$, and $\lambda_i \geq 0$ for all $i$, while $\sigma_2=\sigma_4=\sigma_5=\cdots = \sigma_d=0$.  Observe that equation~\eqref{Lor_96} is of the form~\eqref{eqn:SDEmain} with 
\begin{align}
\label{eqn:vfLor}
Z = \sum_{i=1}^d \lambda_i  x_i\frac{\partial}{\partial x_i}, \quad Z_0 = \sum_{i=1}^d (x_{i+1}-x_{i-2}) x_{i-1} \frac{\partial}{\partial x_{i}}, \quad \text{ and } \quad   Z_j = \sigma_j \frac{\partial}{\partial x_j},\,\,\, j=1,2,\ldots, d. 
\end{align}   
Below, we set 
\begin{align}
\label{eqn:HLor}
H(x)= |x|^2.
\end{align}
Our goal here is to show the following:
\begin{theorem}
\label{thm:Lor}
Let $r=d$, $Z, Z_0, Z_1, \ldots, Z_d$ be as in~\eqref{eqn:vfLor} and $H$ as in~\eqref{eqn:HLor}.  Then the resulting dynamics in~\eqref{eqn:SDEmain} satisfies Assumptions~\ref{assump:H}-\ref{assump:BR}.  In particular, the quantitative minorization condition of Theorem~\ref{thm:minorization} is met.    

\end{theorem} 
 
 As in the previous example, we will only check Assumptions~\ref{assump:H}-\ref{assump:V} here, leaving Assumption~\ref{assump:BR} for the end of this section.    
 \begin{lemma}
  \label{lem:Lor}
Let $r=d$, $Z, Z_0, Z_1, \ldots, Z_d$ be as in~\eqref{eqn:vfLor} and $H$ as in~\eqref{eqn:HLor}.  Then Assumption~\ref{assump:BR} is satisfied. 
\end{lemma} 

\begin{proof}[Proof of Theorem~\ref{thm:Lor}]
By Lemma~\ref{lem:Lor}, we have left to check Assumptions~\ref{assump:H}-\ref{assump:V}.  Note first that Assumption~\ref{assump:H} is immediate since $H(x)=|x|^2$.  By construction, Assumption~\ref{assump:V} (V1)-(V3) is also immediate.  To see that Assumption~\ref{assump:V} (V4) is valid, we note that for $\eta >0$ with $\eta \leq \min\{ \lambda_1/(2\sigma_1^2), \lambda_3/(2\sigma_3)^2\}$
\begin{align*}
\textstyle{\sum_{j=1,3}} e^{-\eta H}( |Z_j H|^2 + |Z_j^2 H| + 1)\in L^1(\RR^d, \, dx)
\end{align*}
and 
\begin{align*}
\eta \sum_{j=1}^d (Z_j H)^2  = 4\eta (\sigma_1^2 |x_1|^2 + \sigma_3^2 |x_3|^2) &\leq 2\lambda_1 |x_1|^2+ 2\lambda_2 | x_3|^2  \\
&\leq ZH  +\sum_{j=1}^d Z_j^2 H = ZH + 2\sum_{j=1}^d \sigma_j^2.
\end{align*}
To check Assumption~\ref{assump:V} (V5), we compute commutators.  That is, first observe that  
\begin{align}
    &[ Z_1, -\epsilon Z+Z_0]=-\epsilon \lambda_1\sigma_1\frac{\partial}{\partial x_1} +\sigma_1(x_3-x_d) \frac{\partial}{\partial x_2} -\sigma_1 x_2 \frac{\partial}{\partial x_3} +\sigma_1 x_{d-1} \frac{\partial}{\partial x_d} \\
& [ Z_3,[ Z_1, -\epsilon Z+Z_0]]=  \sigma_1 \sigma_3 \frac{\partial}{\partial x_2}.
    \end{align}
 Similarly, 
    \begin{align}
      [Z_3, [\partial_{x_2}, -\epsilon Z+Z_0]]= - \sigma_3 \frac{\partial}{\partial x_4}. 
\end{align}
Inductively, we find that 
\begin{align}
[\partial_{x_k}, [\partial_{x_{k-1}}, -\epsilon Z+Z_0]]= -\frac{\partial}{\partial x_{k+1}}.
\end{align}
 From this, it follows that the list of vectors 
 \begin{align*}
 -\epsilon Z+Z_0; Z_1;  Z_3
 \end{align*}
 satisfies the uniform parabolic H\"{o}rmander condition with respect to $\epsilon \in (0,1)$ on every set $H_{<R}$, $R>0$.

\end{proof}
\end{example}

\end{example}
\begin{example}
\label{ex:NSE}
For general finite-dimensional example from fluids (see, e.g., random perturbations of the Euler equations and the shell models in~\cite{BL_21}), suppose $B\in T_d(\RR^d)$ is a polynomial vector field which conserves length, i.e. $B(x)\cdot x=0$ for all $x\in \RR^d$, and such that $\text{div}(B)=0$.  In this example, we consider a general equation on $\RR^d$ of the form 
\begin{align}
\label{eqn:gfluid}
\begin{cases}
dx_t^\epsilon= - \epsilon \Lambda x_t^\epsilon \, dt + B(x_t^\epsilon) \, dt + \sqrt{2\epsilon \Lambda} \, dW_t,\\
x_0^\epsilon =x \in \RR^d
\end{cases}
\end{align}
where $\Lambda =\text{diag}(\lambda_1, \ldots, \lambda_d)$ is a $d\times d$ diagonal, non-negative definite matrix with $\text{Trace}(\Lambda)>0$, and $W_t$ is a standard, $d$-dimensional Brownian motion.  Note also that the Lorenz '96 model of Example~\ref{ex:Lor} falls within the class of equations~\eqref{eqn:gfluid}, but we included it to see the specific quantitative bracket computations.  

Here again, we let 
\begin{align*}
    H(x)=|x|^2. 
\end{align*}
First observe that relation~\eqref{eqn:gfluid} is of the prescribed form~\eqref{eqn:SDEmain} with 
\begin{align}
\label{eqn:fluidZdef}
Z= \Lambda x \cdot \nabla_x, \quad Z_0 = B(x) \cdot \nabla_x, \quad \text{ and } \quad Z_j = \sqrt{\lambda_j} \frac{\partial}{\partial x_j}, \,\, j=1,2,\ldots, d.
\end{align}

We will prove:
\begin{theorem}
\label{thm:gfluids}
Suppose that $\text{\emph{Trace}}(\Lambda)>0$ and $-\epsilon Z+Z_0; Z_1; \ldots; Z_d$ as in~\eqref{eqn:fluidZdef} satisfies the uniform parabolic H\"{o}rmander condition on any $H_R>0$.  Then Assumptions~\ref{assump:H}-\ref{assump:BR} are satisfied for the corresponding dynamics~\eqref{eqn:SDEmain}.  Hence, the quantitative minorization condition of Theorem~\ref{thm:minorization} is met.   
\end{theorem}

Again, we need the following lemma, which will be proved at the end of the section. 
\begin{lemma}
\label{lem:NSE}
Consider $Z_0, Z, Z_1, \ldots, Z_d$ as in~\eqref{eqn:fluidZdef} and suppose that $\text{\emph{Trace}}(\Lambda)>0$.  Then Assumption~\ref{assump:BR} is met.    
\end{lemma} 

\begin{proof}[Proof of Theorem~\ref{thm:gfluids}]
By construction, Assumption~\ref{assump:H} is met.  One can also use a similar line of reasoning to the previous example of Lorenz to show that Assumption~\ref{assump:V} (V4) is satisfied.\end{proof}

\begin{remark}
Note that under our framework, we can also damp all modes but leave the additive noise to be degenerate, so long as the uniform parabolic H\"{o}rmander condition is met on any $H_R$, $R>0$. In the setting of Galerkin truncations of the stochastic Navier--Stokes equations on the periodic box, minimal conditions on the noise for the uniform parabolic H\"{o}rmander condition to hold were first obtained in \cite{EMattingly}. For the bracket computations in the case of shell models, see \cite[Section 1.2]{BL_21}.
\end{remark}

\end{example}

\subsection{Regular points}
\label{sec:BR}
In this section, we use a general argument to validate Assumption~\ref{assump:BR} in Examples~\ref{ex:osc}-\ref{ex:NSE}.

\begin{proof}[Proof of Lemma~\ref{lem:osc}, Lemma~\ref{lem:Lor}, Lemma~\ref{lem:NSE}]
Fix $R>0$.  Given the definition of $H_R$ in each example, there exists $S>R$ such that $B_S(0)\supset \overline{H_R}$.  Let $\xi_1, \ldots, \xi_N, \ldots $ be i.i.d. random variables distributed on $\partial B_{S}(0)$ according to Hausdorff measure.  Also, let $v \in \RR^d_{\neq 0}$ be fixed.  Then following the proof of~\cite[Theorem 7.7]{CFH_22}, for every $N>0$ large enough, with probability one, every face on 
\begin{align}
\partial (\text{interior}([\xi_1, \ldots, \xi_N])),
\end{align} 
where we recall that $[\xi_1, \ldots, \xi_N]$ is the closed convex hull of $\{ \xi_1, \ldots, \xi_N\}$, is not parallel to $v$.  Furthermore, by~\cite[Proposition 1]{PSSW_24}, for any $\zeta >0$, we can pick $N>0$ large enough so that, with high probability,  
\begin{align}
\delta_H( \text{interior}([\xi_1, \ldots, \xi_N]), B_{S}(0)) < \zeta 
\end{align}
where $\delta_H$ denotes Hausdorff measure. 
In particular, we can pick $N>0$ large enough so that, with high probability,
\begin{align}
H_R\subset \text{interior}([\xi_1, \ldots, \xi_N])
 \end{align}
 and each face $\partial (\text{interior}([\xi_1, \ldots, \xi_N]))$ is not parallel to $v$.  Since, by assumption, there is a fixed direction $v \in \text{span}\{ Z_1, \ldots, Z_r\}$ in each example, applying~\cite[Corollary~7.10]{CFH_22} finishes validating Assumption~\ref{assump:BR} for Examples~\ref{ex:Lor}-\ref{ex:NSE}.
\end{proof}

\section{Outline of the Proof of Theorem~\ref{thm:minorization}}
\label{sec:outline}

In this section, we outline the proof of Theorem~\ref{thm:minorization} in order to break apart the argument into manageable pieces, and to explain the overarching details.  Much of the argument is an adaptation of the classical elliptic theories of De Giorgi~\cite{De_57} and Moser~\cite{Mos_61} (see also the expository presentation in~\cite{Stin_24}) to the hypollelliptic setting, but with a careful eye towards quantitative estimates as they depend on $\epsilon \in (0,1)$.  In particular, the basic mantra employed throughout the paper is similar to the one arising from the work of De Giorgi and Moser, except the additional stipulation that the bounds be quantitative; that is, ``quantitative upper bounds imply quantitative lower bounds".

To setup the statements of the results below, for $t>0, (x,y) \in \mathcal{X}^2$ and $\epsilon \in (0,1)$, let
\begin{align}
q_t^\epsilon(x,y) = p_{t/\epsilon}^\epsilon(x,y)
\end{align}       
be the time re-scaled density of the process $\{x_t^\epsilon\}_{t\geq 0}$ solving~\eqref{eqn:SDEmain}.  Also, letting $\chi \in C^\infty_0(\RR; [0,1])$ with $\chi(x)=1$ for $x\leq 3/2$ and $\chi(x)=0$ for $x\geq 2$, define $\chi_R:\mathcal{X}^2\rightarrow [0,1]$ by
\begin{align}
\label{eq:cutoff}
\chi_R(x,y) = \chi( H(x)/R) \chi(H(y)/R). 
\end{align} 
\begin{theorem}[Quantitative upper bounds]
\label{thm:upperboundo}
Suppose that Assumptions~\ref{assump:H}-\ref{assump:BR} are satisfied.  Then for every $R\geq 1,T\geq 1$ there exists a constant $C_{R,T}>0$ and a parameter $s>0$ independent of $\epsilon \in (0,1)$ so that 
\begin{align}
\label{eqn:genupperb}
\sup_{1\leq t \leq T} \| q_t^\epsilon \|_{L^\infty(H_R\times H_R)} + \int_1^T \| \chi_R q_t^\epsilon \|^2_{H^s(\RR^d \times \RR^d) } \, dt \leq C_{R,T}  
\end{align}
for all $\epsilon \in (0,1)$.  
\end{theorem}

The proof of Theorem~\ref{thm:upperboundo} will be carried out in Section~\ref{sec:unifhor} and Section~\ref{sec:Moser}.  In Section~\ref{sec:unifhor},  we discuss uniform H\"{o}rmander-type smoothing estimates which gives the $H^s$ part of the bound in~\eqref{eqn:genupperb}, and in Section~\ref{sec:Moser} we use the uniform $H^s$ type bound, coupled with a quantitative Moser iteration scheme, to give the $L^\infty$ bound in~\eqref{eqn:genupperb}.

Given the quantitative upper bounds on the time-changed density $q_t^\epsilon$ in Theorem~\ref{thm:upperboundo}, we can use the structure of the equations satisfied by $q_t^\epsilon$ to produce a quantitative lower bound, but only on its time-averages over a large enough window which is independent of $\epsilon \in (0,1)$.  This is precisely the mantra of the paper and is central to the classical theories of De Giorgi and Moser.  That is, given `nice' uniform upper bounds of the time-changed density $q_t^\epsilon$ as in~\eqref{eqn:genupperb}, one can try taking functionals of $q_t^\epsilon$, e.g. 
\begin{align}
\label{eqn:typesoff}
(q_t^\epsilon)^{-\delta} \text{ for some } \delta >0 \,\,\text{ or }\,\, - \log(q_t^\epsilon), 
\end{align}
and hope to produce uniform upper bounds of the modified function.  These upper bounds naturally translate to uniform lower bounds on the original function $q_t^\epsilon$.  

However, there is one notable technicality we cannot currently get around, and it is precisely the reason why we consider time-avegrages of the density, instead of the density itself.  This is because we do not have enough quantitative regularity in time to make the arguments, in particular those of De Giorgi~\cite{De_57}, work.  That is, if we had also uniform $H^s$ smoothing in time as well as space, then it can be made to work.  But, it is not clear to us that such regularity is even true due to the structure of the dynamics, as it moves `quickly' in time on the $1/\epsilon$ time-scale when $\epsilon \in (0,1)$ is small.  Thus, we will prove the following:       
\begin{theorem}
\label{thm:lowerbtaint}
Suppose that Assumptions~\ref{assump:H}-\ref{assump:BR} are satisfied and let $R\geq 1$ be such that $H_R\neq \emptyset$.  Then there exists a time $T>1$ and a constant $c_R>0$, both independent of $\epsilon \in (0,1)$, such that
\begin{align}
\inf_{(x,y) \in H_R \times H_R}\int_1^T q_t^\epsilon (x,y) \, dt \geq c_R.  
\end{align}   
\end{theorem}
Theorem~\ref{thm:lowerbtaint} will be proven in Section~\ref{sec:lowerb} using a special time-averaged version of the density, denoted by $h^\epsilon(x,y)$ below in Section~\ref{sec:lowerb}, which behaves well with the De Giorgi argument.  See Theorem~\ref{thm:lowerbta} for a precise statement.  Theorem~\ref{thm:lowerbtaint} will then follow easily from Theorem~\ref{thm:lowerbta}. 

The final, crucial ingredient in the proof of Theorem~\ref{thm:minorization} is to see that, abstractly, the properties obtained in Theorem~\ref{thm:upperboundo} and Theorem~\ref{thm:lowerbtaint}  together imply Theorem~\ref{thm:minorization} using a quantified version of a small set argument as in~\cite{MT_12}.  Here, `abstractly' means that the result holds in the more general context of a Markov process $\{x_t\}_{t\geq 0}$ on the state space $\mathcal{X}$ with $H$ satisfying Assumption~\ref{assump:H}, and with a probability density function $p_t(x,y)$ satisfying:
\begin{assump}
\label{assump:gen}
For every $R\geq 1$, there exists $T>1$, $s>0$, $C_R, c_R>0$ such that the following conditions are met:
\begin{itemize}
\item[(P1)] $\displaystyle{\sup_{1\leq t \leq T} \| p_t \|_{L^\infty(H_R \times H_R)} + \int_1^T \| \chi_R p_t\|^2_{H^s(\RR^d\times \RR^d)} \, dt \leq C_R}, $
\item[(P2)] $\displaystyle{\inf_{H_R \times H_R} \int_1^T p_t(x,y) \, dt \geq c_R}$. 
\end{itemize}
\end{assump}

\begin{theorem}
\label{thm:generalminorization}
Suppose that Assumption~\ref{assump:H} is satisfied and that $\{ x_t\}_{t\geq 0}$ is a Markov process on the state space $\mathcal{X}$ with transition measures $\{\mathcal{P}_t(x, \, \cdot\,)\}_{t >0, x\in \mathcal{X}}$ which are absolutely continuous with respect to Lebesgue measure on $\mathcal{X}$, i.e., for every $x\in \mathcal{X}$ and $t>0$,
\begin{align}
\mathcal{P}_t(x,dy) = p_t(x,y) \, dy
\end{align}
for some $y \mapsto p_t(x,y)\in L^1(\mathcal{X})$. 
Furthermore, suppose that the transition density $p_t(x,y)$ satisfies Assumption~\ref{assump:gen}.  Then, for every $R\geq 2$, there exists $t_0>0$ and $\lambda >0$, both depending only on $R, d$ and $c_R, C_R, T$ and $s$ in Assumption~\ref{assump:gen}, such that 
\begin{align}
\label{eq:generalminorization}
\sup_{x,y \in H_R} p_{t_0}(x,y) \geq \lambda.\end{align}  
\end{theorem}

\begin{remark}
	In the language of Markov process theory, Assumption~\ref{assump:gen} (P2) says that $H_R$ is a \textit{petite set} for the semigroup $\{\mathcal{P}_t\}_{t\ge 0}$. The existence of a petite set is a standard assumption in the study of minorization properties of Markov semigroups; see e.g. \cite{MT_12}.
\end{remark}

\begin{remark}
	It is well known that a Markov process satisfies a minorization condition on any petite set under a suitable aperiodicity assumption \cite{MT_12}. Following the general philosophy of these results, it can be shown that Assumption~\ref{assump:gen} (P2) (i.e., that $H_R$ is a petite set) and \textit{qualitative} smoothness of $(t,x,y) \mapsto p_t(x,y)$ (which in continuous time gives a soft form of aperiodicity) together imply that \eqref{eq:generalminorization} holds for some $t_0, \lambda >0$. The key novelty of Theorem~\ref{thm:generalminorization} is that $\lambda$ and $t_0$ depend quantitatively on the regularity of $p_t$ and a time-averaged lower bound only through the constants in Assumption~\ref{assump:gen}.
\end{remark}

Given Theorem~\ref{thm:generalminorization}, we can immediately conclude Theorem~\ref{thm:minorization} by combining Theorem~\ref{thm:upperboundo} with Theorem~\ref{thm:lowerbtaint}.  Theorem~\ref{thm:generalminorization} will be established in Section~\ref{sec:smallset}.

\section{Uniform minorization at a fixed time $t>0$}
\label{sec:smallset}

In this section, we prove Theorem~\ref{thm:generalminorization}, which we recall is in the more general setting of a Markov semigroup  $\{\mathcal{P}_t\}_{t \ge 0}$ on the state space $\mathcal{X}$ satisfying the hypotheses in the statement.  In particular, $\mathcal{X}$ is induced through an $H$ satisfying Assumption~\ref{assump:H} and the transitions $\{ \mathcal{P}_t(x, \, \cdot \,) \}_{x\in \mathcal{X}, t > 0 }$ are absolutely continuous with respect to Lebesgue measure on $\mathcal{X}$ with transition densities $\{ p_t(x,y) \}_{t >0, x,y \in \mathcal{X}}$ satisfying Assumption~\ref{assump:gen}.  

Throughout this section, we assume that the assumptions of Theorem~\ref{thm:generalminorization} are met, and we let $R \ge 2$ and $T$, $s$, $c_R$, and $C_R$ denote the associated constants guaranteed by Assumption~\ref{assump:gen}. 

\subsection{Proof of Theorem~\ref{thm:generalminorization}}

We now outline the main steps in the proof of Theorem~\ref{thm:generalminorization} and complete the argument while taking key intermediate results, to be proven in subsequent sections, as given. 

The first step in the proof is to construct a \emph{small set} $E \subseteq H_{R}$. Precisely, we establish the following. 

\begin{proposition} \label{prop:smallset}
	There exist $\delta > 0$ (depending only on $R$, $d$, $c_{R}$, $C_{R}$, $T$, and  $s$), $t_* \le 3T$, and $E \subseteq H_{R}$ such that 
	\begin{equation} \label{eq:smallset}
		|E| \ge \delta \qquad \text{and} \qquad p_{t_*}(x,y) \ge \delta \quad \text{for all } \,\, x,y\in E.  
	\end{equation}
\end{proposition}

The proof of Proposition~\ref{prop:smallset} is an adaptation of the arguments in \cite{MT_12} to our setting, and will be carried out in Section~\ref{sec:quantsmall}. The primary challenge is to use use Assumption~\ref{assump:gen} (P1) to control the size and quality of minorization in $E$. 

To prove Theorem~\ref{thm:generalminorization}, we will employ Proposition~\ref{prop:smallset} in the form of the corollary stated below.

\begin{corollary} \label{lem:manytimes}
	There exists $\delta_1 > 0$, depending only on $R$, $T$, $d$, $c_{R}$, $C_{R}$, and $\delta$ from Proposition~\ref{prop:smallset}, for which the following hold:
	\begin{itemize}
		\item[(a)] For every $x \in H_R$ there exists $t_x \le 4T$ such that 
		\begin{equation} \label{eq:firsthit}
			\inf_{y \in E} p_{t_x}(x,y) \ge \delta_1.
		\end{equation}
		\item[(b)] There is a set $S \subseteq [0,7T]$ with
		\begin{equation} \label{eq:manytimes}
			|S| \ge \delta_1 \quad \text{and} \quad \inf_{x,y \in E}p_t(x,y) \ge \delta_1  \quad \forall t \in S.
		\end{equation}
	\end{itemize}
\end{corollary}

Corollary~\ref{lem:manytimes} follows readily from combining Assumption~\ref{assump:gen} (P2) with Proposition~\ref{prop:smallset}. We will carry out the argument below. First, however, we record a simple auxiliary lemma that formalizes the general principle that a lower bound on an integral, combined with a uniform upper bound on the integrand, yields a quantitative pointwise lower bound on a subset of positive measure. Although elementary, this fact will be invoked repeatedly throughout the remainder of Section~\ref{sec:smallset} to obtain pointwise lower bounds from Assumption~\ref{assump:gen}, and so we include a precise statement for the sake of completeness.

\begin{lemma} \label{lem:noconcentration}
	Fix $n \in \N$ and let $D \subseteq \RR^n$ be a bounded, measurable set. Let $f:D \to [0,\infty)$ be a measurable function such that $\|f\|_{L^\infty(D)} \le C$ and $\int_{D} f \ge c$ for some constants $c,C > 0$. Then, for every $K \ge 2$ we have 
	$$ \left|\left\{x \in D: f(x) \ge \frac{c}{K|D|}\right\}\right| \ge \frac{c(1-K^{-1})}{C - c(K|D|)^{-1}} > 0.$$
\end{lemma}

\begin{proof}
	First, note that $C - c(K|D|)^{-1} > C - c|D|^{-1} > 0$ since $c \le \int_D f \le C|D|$. Let $\eta = c(K|D|)^{-1}$. We have 
	\begin{align*}
		c & \le \int_D f = \int_{\{f \ge \eta\}} f + \int_{\{f < \eta\}} f \\
		& \le C |\{f \ge \eta\}| + \eta(|D| - |\{f \ge \eta\}|) \\ 
		& = |\{f \ge \eta\}|(C-\eta) + \eta |D|.
	\end{align*}
	Thus, 
	$$ |\{f \ge \eta\}| \ge \frac{c-\eta|D|}{C-\eta} $$
	and putting in the choice of $\eta$ completes the proof.
\end{proof}

\begin{proof}[Proof of Corollary~\ref{lem:manytimes}]
	We begin with the proof of part (b).  Let $E \subseteq H_{R}$, $\delta > 0$, and $t_* \le 3T$ be as in the statetment of Proposition~\ref{prop:smallset}. By Assumption~\ref{assump:gen} (P2) \and \eqref{eq:smallset}, we have 
	\begin{equation} \label{eq:aperiodic1}
		\int_1^T \left(\int_{E\times E} p_t(x,y) dx dy \right) dt \ge c_R |E|^2 \ge c_R \delta^2. 
	\end{equation}
	Since $\int_{E\times E} p_t(x,y) dx dy \le |E| \le |H_{R}|$ for $t>0$ as $\int_\mathcal{X} p_t(x,y) \, dy =1$, it follows from \eqref{eq:aperiodic1} and Lemma~\ref{lem:noconcentration} that there exist $U \subseteq [1,T]$ and $\delta_2 > 0$ depending only $R$, $d$, $c_R$, $C_R$, and $T$ such that 
	\begin{equation}\label{eq:aperiodic2}
		|U| > \delta_2 \quad \text{and} \quad \int_{E\times E} p_t(x,y)dx dy\ge \delta_2 \quad \forall t \in U.
	\end{equation}
	By Proposition~\ref{prop:smallset} and \eqref{eq:aperiodic2}, for any $t \in U$ and $(x,z) \in E\times E$  we have 
	\begin{equation} \label{eq:aperiodic3}
		\begin{aligned} 
			p_{t+2t_*}(x,z) &= \int_{\mathcal{X} \times \mathcal{X}} p_{t_*}(y_2,z)p_t(y_1,y_2)p_{t_*}(x,y_1) d y_1 d y_2 \\ 
			& \ge \delta^2 \int_{E \times E}p_t(y_1,y_2) d y_1 d y_2 \ge \delta^2 \delta_2.
		\end{aligned}
	\end{equation}
	This proves \eqref{eq:manytimes} with $S = U + 2t_*\subseteq [0,7T]$.  For part (a), reasoning as in \eqref{eq:aperiodic1} and \eqref{eq:aperiodic2} shows that for any $x \in H_R$ there exists $T_x \le T$ such that
	\begin{equation}
		\int_E p_{T_x}(x,y) dy \ge \frac{c_R \delta}{T-1}.
	\end{equation} 
	Setting $t_x = T_x + t_* \le 4T$, the proof is concluded similarly to (a) using the Chapman-Kolmogorov equation and Proposition~\ref{prop:smallset}.
\end{proof}

\begin{remark}\label{rem:setgeometry}
	Although the lower bound on $|S|$ is quantitative in terms of the constants in Assumption~\ref{assump:gen}, without control of $\partial_t p_t$, the proof above yields essentially no information about the structure of $S$. 
\end{remark}

Corollary~\ref{lem:manytimes}(a) ensures an $x$-dependent time of strict positivity within the small set $E$, while part (b) asserts that this strict positivity recurs on a set of times with positive measure. The latter can be interpreted as a form of aperiodicity. Indeed, as a corollary of the Steinhaus theorem, $S+S$ contains some open interval $(a,a+\rho)$. Combining this with \eqref{eq:firsthit} and \eqref{eq:manytimes} it can be shown that there exists $T_* > 0$ such that for every $x \in H_R$ and $t \ge T_*$ there is some $c>0$ for which 
$$\inf_{y \in H_R} p_{T_*}(x,y) \ge c.$$
The constants $c$ and $T_0$, however, depend on $\rho$, which in turn depends on the structure of $S$. To overcome this issue, we establish a quantitative Steinhaus-type theorem. Its proof, which requires some results from additive combinatorics, is deferred to Section~\ref{sec:aperiodic}. 

\begin{lemma}\label{lem:aperiodic}
	Fix $\eta > 0$ and $L \in \N$. Let $A \subseteq [0,L]$ be a measurable set with $|A| \ge \eta$. There exists $n \in \N$, depending only on $\eta$ and $L$, and a closed interval $I$ with $|I| = 1$ such that
	\begin{equation} \label{eq:Steinhaus}
		I \subseteq nA:= \underbrace{(A+A+A+\ldots+A)}_{n\text{ } \mathrm{ times}}. 
	\end{equation}
\end{lemma}

\begin{remark}
	It is immediate from the Steinhaus theorem that \eqref{eq:Steinhaus} holds for some $n \in \N$ depending possibly in a complicated way on $A$. The key feature of Lemma~\ref{lem:aperiodic} is that $n$ depends only on $\eta$ and $L$. 
\end{remark}

With Corollary~\ref{lem:manytimes} and Lemma~\ref{lem:aperiodic} in hand, we can now conclude the proof of Theorem~\ref{thm:generalminorization}.

\begin{proof}[Proof of Theorem~\ref{thm:generalminorization}]
	Let $S \subseteq [0,7T]$ and $\delta_1 > 0$ be as in Corollary~\ref{lem:manytimes}. Choose $m \in \N$ with $m > T$, and let $n \in \N$ be the number guaranteed by Lemma~\ref{lem:aperiodic} applied with $A = S$, $\eta = \delta_1$ and $L = 7m$. By Lemma~\ref{lem:aperiodic} and $nS \subseteq [0,7mn]$, there exists $T_0 \le 7nm$ such that
	\begin{equation}\label{eq:intervalT_0}
		T_0 + [0,1] \subseteq nS.
	\end{equation}
	 We set 
	\begin{equation} \label{eq:t0def}
		t_0 = 5T+5mT_0.
	\end{equation}
Observe that $t_0 \le 5m(1+7nm)$ and that by Lemma~\ref{lem:aperiodic} this upper bound depends only on $T$ and $\delta_1$. 

Fix $(x,y) \in H_R^2$. By Corollary~\ref{lem:manytimes}(a), there exists $t_x \le 4T$ such that 
\begin{equation} \label{eq:hitsmall}
	\inf_{z \in E} p_{t_x}(x,z) \ge \delta_1.
\end{equation} 
Moreover, from Assumption~\ref{assump:gen} (P2) and Proposition~\ref{prop:smallset}, for some $t_y \le T$ we have 
\begin{equation}\label{eq:leavesmall}
	\int_E p_{t_y}(z,y) dz \ge \frac{c_R \delta}{T-1}.
\end{equation} 
Since $0<t_x + t_y \le 5T$, it follows from \eqref{eq:t0def} that we can write 
$$t_0 - (t_x+t_y) = 5m(T_0 + T_1) $$  
for some $T_1 \in [0,1)$. By \eqref{eq:intervalT_0}, we have $T_0 + T_1 \in nS$. Thus, $5m(T_0+T_1) \in 5mn S$, and hence by Corollary~\ref{lem:manytimes}(b) and the Chapman-Kolmogorov equation we have 
\begin{equation} \label{eq:nSsmall}
	\inf_{z_1, z_2 \in E} p_{5m(T_0+T_1)}(z_1,z_2) \ge \delta_1^{5nm}.
\end{equation} 
Writing $t_0 = t_x + t_y + 5m(T_0+T_1)$ and combining \eqref{eq:hitsmall}, \eqref{eq:leavesmall}, and \eqref{eq:nSsmall}, we obtain 
\begin{align*} 
	p_{t_0}(x,y) &= \int_{\mathcal{X}^2} p_{t_x}(x,z_1)p_{5m(T_0+T_1)}(z_1,z_2)p_{t_y}(z_2,y) dz_1 dz_2\\
	& \ge \int_{E^2} p_{t_x}(x,z_1)p_{5m(T_0+T_1)}(z_1,z_2)p_{t_y}(z_2,y) dz_1 dz_2 \\ 
	& \ge  \delta_1^{5mn+1} |E| \int_E p_{t_y}(z_2,y) dz_2 \\ 
	& \ge \frac{\delta_1^{5mn+1} \delta^2 c_R}{T-1}.
	\end{align*}
Since $(x,y) \in H_R^2$ was arbitrary, this completes the proof.
\end{proof}

\subsection{Constructing a small set}
%One small set for a man, one giant leap for a Markov chain} 
\label{sec:quantsmall}

Our goal now is to prove Proposition~\ref{prop:smallset} on the existence of a small set $E \subseteq H_{R}$, where recall that $R \ge 2$ is fixed. Throughout this section, it is essential that all estimates remain quantitative in terms of the constants appearing in Assumption~\ref{assump:gen}. For notational convenience, we will consistently use $C_i$ and $\delta_i$ to denote constants that depend only on $R$, $d$, $c_{R}$, $C_{R}$, $T$, and  $s$, even if this dependence is not stated explicitly in each instance.

The main step in the proof of Proposition~\ref{prop:smallset} is establishing a lower bound on transitions between a pair of sets $E_1,E_2 \subseteq H_{R}$ at some fixed time. Precisely, we will prove the following lemma.

\begin{lemma} \label{lem:smallsetkey}
There exist constants $\delta_1, \delta_2 > 0$, a time $t_1 \le 2T$, and sets $E_1,E_2 \subseteq H_{R}$ such that $|E_1| \wedge |E_2| \ge \delta_1$ and  
\begin{equation}
	p_{t_1}(x,z) \ge \delta_2 \quad \forall (x,z) \in E_1\times E_2.
\end{equation}
\end{lemma}

Taking Lemma~\ref{lem:smallsetkey} for granted, Proposition~\ref{prop:smallset} follows readily from Assumption~\ref{assump:gen} (P2) and the Chapman-Kolmogorov equation. 

\begin{proof}[Proof of Proposition~\ref{prop:smallset}]
Let $\delta_1$, $\delta_2$, $E_1$, $E_2$, and $t_1$ be as in the statement of Lemma~\ref{lem:smallsetkey}. By Assumption~\ref{assump:gen} (P2), for any $x \in H_R$ we have
$$
\int_1^T \mathcal{P}_t(x,E_1) \,  dt = \int_1^T \int_{E_1} p_t(x,y)\, dy dt \ge c_{R} |E_1|.
$$
Integrating this inequality over $x \in E_2  \subseteq H_R$ yields
$$ \int_1^T \left(\int_{E_2}\mathcal{P}_t(x,E_1) \, dx\right) ds \ge c_R |E_1||E_2| \ge c_{R} \delta_1^2, $$
and hence for some $t_2 \in [1,T]$ we have 
$$ \int_{E_2} \mathcal{P}_{t_2}(x,E_1)d x \ge \frac{c_{R} \delta_1^2}{T-1}. $$
Therefore, by Lemma~\ref{lem:noconcentration}, there exist $\delta_3 > 0$ and $E \subseteq E_2$ such that 
\begin{equation} \label{eq:16}
	|E| \ge \delta_3 \qquad \text{and} \qquad \mathcal{P}_{t_2}(x,E_1) \ge \delta_3 \quad \forall x \in E. 
\end{equation}
We conclude from Lemma~\ref{lem:smallsetkey} and \eqref{eq:16} that for any $(x,z) \in E \times E$ we have
$$ p_{t_1 + t_2}(x,z) \ge \int_{E_1} p_{t_2}(x,y) p_{t_1}(y,z) dy \ge \delta_1 \mathcal{P}_{t_2}(x,E_1) \ge \delta_1 \delta_3.$$
This completes the proof, since $t_1 + t_2 \le 2T + T = 3T$ and $|E| \ge \delta_3$.
\end{proof}

\subsubsection{Proof of Lemma~\ref{lem:smallsetkey}}
The remainder of Section~\ref{sec:quantsmall} is dedicated to the proof of Lemma~\ref{lem:smallsetkey}. The argument is somewhat tedious and technical, so we begin here by motivating the general strategy. Notice first that a \textit{qualitative} version of Lemma~\ref{lem:smallsetkey} would hold trivially if Assumption~\ref{assump:gen} (P1) were replaced simply with the hypothesis that $p_t \in C^\infty(\RR^d \times \RR^d)$. Indeed, by Assumption~\ref{assump:gen}, there exists $t \in [1,T]$ and $(u,w) \in H_{R} \times H_{R}$ such that $p_t(u,w) > 0$.  Smoothness would then imply that for $\delta > 0$ sufficiently small there holds
\begin{equation} \label{eq:motivation1}
C_\delta(u) \times C_\delta(v) \subseteq \{(x,z) \in H_{R}^2 \, : \,p_t(x,z) > 0\}, 
\end{equation}
where $C_\delta(u)$ denotes the $\delta$-neighborhood centered at $u$. One could then simply take $E_1 = C_\delta(u)$ and $E_2 = C_\delta(v)$. It is not possible to use such an argument if one only has $p_t(x,\cdot) \in L^1$, nor is it possible to make it quantitative with control solely of $\|p_t\|_{H^s(\RR^d \times \RR^d)}$ for $s < d$. To circumvent this issue, a standard approach for $L^1$ densities is to introduce an intermediate state and bound $p_t(x,z)$ from below using the Chapman-Kolmogorov equation \cite{MT_12}. More specifically, one identifies a pair of times $(t,\tau)$ and triple $(u,v,w)$ such that $p_t(u,v)p_\tau(v,w) > 0$, and then uses the Lebesgue density theorem to construct $E_1, E_2$ such that
\begin{equation} \label{eq:smallsetgoal}
(x,z) \in E_1 \times E_2 \implies |\{y \in \mathcal{X}: p_t(x,y) \wedge p_\tau(y,z) > 0\}| > 0.
\end{equation}
Our basic strategy is to show that Assumption~\ref{assump:gen} is sufficient to make the lower bounds in \eqref{eq:smallsetgoal} quantitative for $|E_1|\wedge |E_2| \ge \delta_1$. Once this is accomplished, Lemma~\ref{lem:smallsetkey} follows with $t_1 = t + \tau$ immediately from the Chapman-Kolmogorov equation.

In view of \eqref{eq:smallsetgoal} and the discussion above, it is natural to begin by studying the family of sets $A_{t,\tau}^\delta$, defined for $(t,\tau) \in [1,T]\times [1,T]$ and $\delta > 0$ by
\begin{equation} \label{eq:Attaudef}
A_{t,\tau}^\delta = \{(x,y,z) \in H_{R_1}^3 \, :\, (x,y) \in A_{t}^\delta, (y,z) \in A_{\tau}^\delta\},
\end{equation}
where 
\begin{align}
A_t^\delta = \{(x,y) \in H_{R_1}^2 \,: \,p_t(x,y) \ge \delta\}
\end{align}
 and $R_1 = R/2 \ge 1$. We first show that there exist $\delta > 0$ and a pair of times $(t,\tau)$ such that $|A_{t,\tau}^\delta|$ is bounded below while both $p_t$ and $p_\tau$ are controlled in $H^s(\RR^d\times \RR^d)$. That is, recalling the definition of $\chi_R$ from \eqref{eq:cutoff}, we will now show:

\begin{lemma}\label{lem:smallsetstep1}
There exist $\delta_4, C_1 > 0$ and $(t,\tau) \in [1,T] \times [1,T]$ such that 
\begin{equation} \label{eq:1smallsetstep1}
	\|\chi_{R} p_t\|_{H^s(\RR^d\times \RR^d)} \vee \|\chi_{R} p_\tau\|_{H^s(\RR^d\times \RR^d)} \le C_1
\end{equation}
and
\begin{equation} \label{eq:2smallsetstep1}
	|A_{t,\tau}^{\delta_4}| \ge \delta_4.
\end{equation}
\end{lemma}

\begin{proof}
By Assumption~\ref{assump:gen} (P2), for any $(x,y,z) \in H_{R_1}^3$ we have 
$$
\left(\int_1^T p_t(x,y) dt\right) \left(\int_1^T p_\tau(y,z) d\tau\right) \ge c_{R}^2.
$$
Integrating this lower bound over $H_{R_1}^3$ gives 
\begin{equation} \label{eq:minorization1}
	\int_1^T \int_1^T \left(\iiint_{H_{R_1}^3} p_t(x,y)p_\tau(y,z)d x dy dz\right)d t d\tau \ge c_{R}^2 |H_{R_1}|^3.
\end{equation}
From  Assumption~\ref{assump:gen} (P1) we have 
$$
\left|\iiint_{H_{R_1}^3} p_t(x,y)p_\tau(y,z)d x dy dz\right| \le C_{R}^2 |H_{R_1}|^3,
$$
and therefore by \eqref{eq:minorization1} and Lemma~\ref{lem:noconcentration} it follows that 
\begin{equation} \label{eq:minorization2}
	\left|\left\{(t,\tau)\in [1,T]^2: \iiint_{H_{R_1}^3} p_t(x,y)p_\tau(y,z)dx dy dz \ge \delta_5\right\}\right|\ge \delta_5
\end{equation}
for some $\delta_5 > 0$. Now, from Assumption~\ref{assump:gen} (P1) and Chebyshev's inequality, for $C_1 = 2\sqrt{C_{R}(T-1)/\delta_5}$ we have 
\begin{equation}
	|\{t \in [1,T]: \|\chi_{R} p_t\|_{H^s(\RR^d \times \RR^d)} \ge C_1\}| \le \frac{\delta_5}{4 (T-1)},
\end{equation}
which implies that 
\begin{equation} \label{eq:minorization3}
	|\{(t,\tau) \in [1,T]^2: \|\chi_{R} p_t\|_{H^s(\RR^d \times \RR^d)} \vee  \|\chi_{R} p_\tau\|_{H^s(\RR^d \times \RR^d)}\ge C_1 \}|\le \frac{\delta_5}{2}.
\end{equation}
Combining \eqref{eq:minorization2} and \eqref{eq:minorization3}, it follows that there exists $(t, \tau) \in [1,T]\times [1,T]$ such that 
\begin{equation}\label{eq:minorization4}
	\begin{aligned}
		& \iiint_{H_{R_1}^3} p_t(x,y)p_\tau(y,z) d x d y  d z \ge \delta_5 \quad \text{and} \\
		&\quad \|\chi_{R} p_t\|_{H^s(\RR^d\times \RR^d)} \vee \|\chi_{R} p_\tau\|_{H^s(\RR^d \times \RR^d)} \le C_1,
	\end{aligned}
\end{equation}
which proves \eqref{eq:1smallsetstep1}. To establish \eqref{eq:2smallsetstep1} for $(t,\tau)$ satisfying \eqref{eq:minorization4}, observe that Assumption~\ref{assump:gen} (P1) imply that $A_{t,\tau}^\delta \subseteq \{(x,y,z) \in H_{R_1}^3: p_t(x,y)p_\tau(y,z) \ge C_{R_1} \delta\}$. Therefore, \eqref{eq:2smallsetstep1} follows for $\delta_4$ sufficiently small by Lemma~\ref{lem:noconcentration} and the lower bound in \eqref{eq:minorization4}.
\end{proof}

For the remainder of this section, $(t,
\tau) \in [1,T]\times [1,T]$ and $\delta_4$ denote the fixed choices from  Lemma~\ref{lem:smallsetstep1}. To simplify notation, we now drop the superscript $\delta$ from the definitions in \eqref{eq:Attaudef} and write $A_{t,\tau}$, $A_t$ and $A_\tau$ for the choice $\delta = \delta_4$. 

The next step in the proof of Lemma~\ref{lem:smallsetkey} is a quantitative Lebesgue density type argument. Our goal is to use the regularity of $p_t$ and $p_\tau$ provided by \eqref{eq:1smallsetstep1} to show that there exists $(u,v,w) \in A_{t,\tau}$ for which the density of points $(x,y)$ near $(u,v)$ and $(y,z)$ near $(v,w)$ such that $p_t(x,y) \wedge p_\tau(y,z)$ remains bounded below is high. Precisely, fix $\delta_* > 0$ small enough so that $u\in H_{R_1}$ implies 
\begin{equation} \label{eq:deltastar}
 \{x \in \mathcal{X}: |x-u| \le \delta_*\} \subseteq H_{R}.
\end{equation}
Such a $\delta_*$ exists because the sublevel sets $H_R$ are open and precompact. Then, for $\delta \in (0,\delta_*)$ and $(u,v) \in H_{R_1}^2$ we define $$C_\delta(u,v) = \{(x,y) \in H_{R}^2: |x-u| \le \delta, |y-v| \le \delta\}$$
and
$$A_t(u,v;\delta) = \left\{(x,y) \in C_\delta(u,v):p_t(x,y) \ge \frac{\delta_4}{2}\right\},$$  
where $\delta_4>0$ is as in the statement of Lemma~\ref{lem:smallsetstep1}.
\begin{lemma}\label{lem:smallsetstep2}
There exists $C_2 > 0$ such that for all $\delta \in (0,\delta_*)$ there is some $(u,v,w) \in A_{t,\tau}$ for which
\begin{equation}\label{eq:smallsetstep2}
	\frac{|A_t(u,v;\delta)|}{|C_\delta(u,v)|} \ge 1-C_2 \delta^{2s} \quad \text{and} \quad \frac{|A_\tau(v,w;\delta)|}{|C_\delta(v,w)|} \ge 1-C_2 \delta^{2s}.
\end{equation}
\end{lemma}

Before proving Lemma~\ref{lem:smallsetstep2}, we record a simple auxilliary lemma concerning fractional Sobolev regularity.

\begin{lemma}\label{lem:Besov}
Let $f \in H^s(\RR^d \times \RR^d)$ and for $\delta > 0$ define
$$ I_\delta(x,y) = \frac{1}{\delta^{2d}}\int_{|x-x'| \le \delta, |y-y'|\le \delta} |f(x,y) -f(x',y')|^2 dx' d y'. $$
Then there exists a constant $C > 0$ depending only on $s$ and $d$ such that for any $K \ge 1$ and $\delta > 0$ we have   
$$ |\{(x,y) \in H_R^2 \, :\, I_\delta(x,y) > K \delta^{2s}\}| \le \frac{C|H_R|^2}{K}\|f\|^2_{H^s(\RR^d\times\RR^d)}.$$  	
\end{lemma}

\begin{proof}
From the characterization of the homogeneous Besov spaces in terms of finite differences (see e.g. \cite{Bahouri2011}), there is $C > 0$ depending only on $s$ and $d$ such that  
\begin{equation}\label{eq:finitedifference}
	\int_{\RR^d \times \RR^d}\left(\int_{\RR^d\times \RR^d} \frac{|f(x,y) - f(x',y')|^2}{|(x,y)-(x',y')|^{2d+2s}} dx'  dy'\right)  dx dy  \le C\|f\|^2_{H^s(\RR^d \times \RR^d)}
\end{equation}
for every $f \in H^s(\RR^d \times \RR^d)$. 
Since $|(x,y)-(x',y')| \le \sqrt{2} \delta$ when $|x-x'| \le \delta$ and $|y-y'| \le \delta$, it follows from \eqref{eq:finitedifference} that
\begin{align*}
	\int_{H_R\times H_R} \left(\int_{|x-x'|\le \delta, |y-y'|\le \delta} |f(x,y) - f(x',y')|^2  dx'  dy'\right) dx dy \le 2^{d+s} \delta^{2d+2s}C\|f\|^2_{H^s(\RR^d \times \RR^d)}.
\end{align*}
In other words, 
$$  \int_{H_R^2} I_\delta(x,y) dx dy \le 2^{d+s} \delta^{2s} C \|f\|_{H^s(\RR^d \times \RR^d)}^2,$$
and so the lemma follows from Chebyshev's inequality. 
\end{proof}

\begin{proof}[Proof of Lemma~\ref{lem:smallsetstep2}]

For $\delta \in (0,\delta_*)$ and $(u,v) \in H_{R_1}^2$, let 
$$ I_{t,\delta}(u,v) = \frac{1}{\delta^{2d}} \int_{|x-u| \le \delta,|y-v| \le \delta}|p_t(x,y) - p_t(u,v)|^2 dx dy. $$
We first claim for all $C_3 > 0$ sufficiently large and $\delta \in (0,\delta_*)$ we have
\begin{equation} \label{eq:claim1}
	A_{t,\tau} \cap \{(u,v,w) \in H_{R_1}^3\, :\, I_{t, \delta}(u,v) \vee I_{\tau, \delta}(v,w) \le C_3 \delta^{2s}\} \neq \emptyset.
\end{equation}
To prove \eqref{eq:claim1}, we begin by defining
$$  B_t = \{(u,v) \in A_t: I_{t,\delta}(u,v) \le C_3 \delta^{2s}\}.$$
By \eqref{eq:1smallsetstep1}, Lemma~\ref{lem:Besov}, and the fact that $I_{t,\delta}(u,v)$ remains unchanged if $p_t$ is replaced by $\chi_{R} p_t$ (by definition of $\delta_*$), there exists $C_4 > 0$ that does not depend on $C_3$ such that
$$ |\{(u,v) \in H_{R_1}^2\, :\, I_{t,\delta}(u,v) > C_3 \delta^{2s}\}| \le \frac{C_4 C_1^2}{C_3}|H_{R_1}|^2.$$
Therefore,
\begin{equation} \label{eq:7}
	|A_t \setminus B_t| \le |H_{R_1}|^2 \frac{C_4 C_1^2}{C_3}
\end{equation}
and by \eqref{eq:1smallsetstep1} the same bound also holds with $t$ replaced by $\tau$. 
Observe now that by the definitions of $A_{t,\tau}$, $B_t$, and $B_\tau$ we have
\begin{align*}
	A_{t,\tau}&=\{(u,v,w) \in H_{R_1}^3: (u,v) \in A_t, (v,w) \in A_\tau\} \\ 
	& \subseteq \{(u,v,w) \in H_{R_1}^3: (u,v) \in B_t, (v,w) \in B_\tau\} \\
	& \quad \cup \{(u,v,w) \in H_{R_1}^3: (u,v) \in A_t \setminus B_t\} \cup \{(u,v,w) \in H_{R_1}^3: (v,w) \in A_\tau \setminus B_\tau\},
\end{align*}
which combined with \eqref{eq:2smallsetstep1} and \eqref{eq:7} yields
$$ \delta_4 \le |\{(u,v,w) \in H_{R_1}^3: (u,v) \in B_t, (v,w) \in B_\tau\}| + 2|H_{R_1}|^3 \frac{C_4 C_1^2}{C_3}.$$
Thus, $\{(u,v,w) \in H_{R_1}^3: (u,v) \in B_t, (v,w) \in B_\tau\}$ is nonempty for $C_3$ sufficiently large, which implies the claim \eqref{eq:claim1}.

From \eqref{eq:claim1}, we may select some $(u,v,w) \in A_{t,\tau}$ such that 
\begin{equation} \label{eq:Icontrol}
	I_{t,\delta}(u,v) \vee I_{\tau,\delta}(v,w) \le C_3 \delta^{2s}.
\end{equation}
To complete the proof of Lemma~\ref{lem:smallsetstep2}, it now suffices to show that \eqref{eq:smallsetstep2} holds for this particular choice of $(u,v,w)$. By \eqref{eq:Icontrol}, Chebyshev's inequality, and $|C_\delta(u,v)| \approx \delta^{2d}$, for any $K \ge 1$ we have 
$$ \frac{|\{(x,y) \in C_{\delta}(u,v): |p_t(x,y) - p_t(u,v)| \ge K^{-1/2}\}|}{|C_\delta(u,v)|} \le C_3 K \delta^{2s}. $$
Now, $(u,v) \in A_t$, so $p_t(u,v) \ge \delta_4$. Thus, choosing $K = 4\delta_4^{-2}$, we get that 
\begin{equation} \label{eq:1smallsetstep2}
	\frac{|\{(x,y) \in C_\delta(u,v): p_t(x,y) \le \frac{\delta_4}{2}\}|}{|C_\delta(u,v)|} \le \frac{4C_3 \delta^{2s}}{\delta_4^2}. 
\end{equation}
This implies the first inequality in \eqref{eq:smallsetstep2}. The second follows similarly, since the same argument as above gives \eqref{eq:1smallsetstep2} with $(u,v)$ and $t$ replaced by $(v,w)$ and $\tau$, respectively.
\end{proof}

With Lemma~\ref{lem:smallsetstep2} in hand, we are finally ready to construct $E_1, E_2 \subseteq H_{R}$ in the spirit of \eqref{eq:smallsetgoal} and finish the proof of Lemma~\ref{lem:smallsetkey}.

\begin{proof}[Proof of Lemma~\ref{lem:smallsetkey}]
Let $C_2>0$ be as in \eqref{eq:smallsetstep2} and fix $\delta > 0$ small enough so that $1-C_2 \delta^{2s} \ge 7/8$. Let $(u,v,w)$ be the associated triple guaranteed by Lemma~\ref{lem:smallsetstep2}. We define the sections 
$$ A_t(u,v;\delta)_x = \{y \in C_\delta(v):(x,y) \in A_t(u,v;\delta)\}$$ 
and 
$$A_\tau(v,w;\delta)_z = \{y \in C_\delta(v):(y,z) \in A_\tau(v,w;\delta)\}, $$
where $A_t(u,v;\delta)$ and $A_\tau(u,v;\delta)$ are as defined just before the statement of Lemma~\ref{lem:smallsetstep2}.
By Lemma~\ref{lem:smallsetstep2}, the choice of $\delta$, and Fubini's theorem we have 
\begin{equation} \label{eq:11}
	\int_{C_\delta(u)} |A_t(u,v;\delta)_x \cap C_\delta(v)| d x = |A_t(u,v;\delta)| \ge \frac{7}{8}|C_\delta(u,v)| = \frac{7}{8}|C_\delta(u)||C_\delta(v)|.
\end{equation}
Let
$$ E_1 = \left\{x \in C_\delta(u): |A_t(u,v;\delta)_x \cap C_\delta(v)| \ge \frac{3}{4}|C_\delta(v)|\right\}. $$
Then, by \eqref{eq:11}, we have
\begin{align*} \frac{7}{8}|C_\delta(u)||C_\delta(v)| &\le 	\int_{C_\delta(u) \cap E_1} |A_t(u,v;\delta)_x \cap C_\delta(v)| d x  + 	\int_{C_\delta(u)\setminus E_1} |A_t(u,v;\delta)_x \cap C_\delta(v)| d x \\ 
	& \le |C_\delta(v)||E_1| + \frac{3}{4}|C_\delta(v)||C_\delta(u)|,
	\end{align*}
which implies
\begin{equation} \label{eq:12}
	|E_1| \ge \frac{1}{8}|C_\delta(u)|.
\end{equation}
By the same argument, defining $E_2 = \{z \in C_\delta(w): |A_\tau(v,w;\delta)_z \cap C_\delta(v)| \ge \frac{3}{4}|C_\delta(v)|\}$, we have
\begin{equation} \label{eq:13}
	|E_2| \ge  \frac{1}{8}|C_\delta(w)|.
\end{equation}
Moreover, it follows by the definitions of $E_1$ and $E_2$ that if $(x,z) \in E_1 \times E_2$, then 
\begin{align} \nonumber
	|A_t(u,v;\delta)_x \cap A_\tau(v,w;\delta)_z| &\ge |C_\delta(v)| - |C_\delta(v) \setminus A_t(u,v;\delta)_x| - |C_\delta(v) \setminus A_\tau(v,w;\delta)_z| \\ 
	& \ge |C_\delta(v)| - \frac{1}{4}|C_\delta(v)| - \frac{1}{4}|C_\delta(v)| = \frac{1}{2}|C_\delta(v)|.\label{eq:14}
\end{align}
Since $y \in A_t(u,v;\delta)_x \cap A_\tau(v,w;\delta)_z$ implies $p_t(x,y) \ge \delta_4/2$ and $p_\tau(y,z) \ge \delta_4/2$, we have by \eqref{eq:14} that for $(x,z) \in E_1 \times E_2$ there holds
\begin{equation}\label{eq:15}
	p_{t+\tau}(x,z) \ge \int_{A_t(u,v;\delta)_x \cap A_\tau(v,w;\delta)_z} p_t(x,y)p_\tau(y,z) dy \ge \frac{\delta_4^2 |C_\delta(v)|}{8}.
\end{equation}
As the choice of $\delta$ depends only on $C_4$, the bounds \eqref{eq:12}, \eqref{eq:13}, and \eqref{eq:15} together complete the proof. 
\end{proof}

\subsection{Proof of quantitative Steinhaus-type theorem} \label{sec:aperiodic}

In this section, we prove Lemma~\ref{lem:aperiodic}. Throughout, $A$, $\eta$, and $L$ are as in the statement of the lemma, and we use the notation 
$$ n B = \underbrace{(B+B+\ldots+B)}_{n\text{ } \mathrm{ times}}, \quad B \subseteq \RR,\ n \in \N.$$

Before proceeding to the main argument, we show that it is sufficient to consider the case where $A \subseteq [0,L]$ is open. By the Steinhaus theorem, there exists $\rho \in (0,1)$ and $x \in \RR$ and such that $(x,x+\rho) \subseteq A + A$.
Let 
$$ A_1 = A+A+A$$
and
$$A_2 = \bigcup_{a \in A} (x+a, x+a+\rho).$$
Then, $A_2$ is an open set satisfying 
\begin{equation}  \label{eq:A2properties}
A_2 \subseteq A_1 \cap [0,3L] \quad \text{and} \quad |A_2| \ge |A| \ge \eta.
\end{equation}
From the definition of $A_1$ and the fact that $A_2 \subseteq A_1$, we have
\begin{equation} \label{eq:4nA}
3nA = nA_1 \supseteq nA_2.
\end{equation}
Therefore, by \eqref{eq:A2properties} and \eqref{eq:4nA}, it is sufficient to prove Lemma~\ref{lem:aperiodic} with $A$ and $L$ replaced, respectively, by $A_2$ and $3L$. In other words, we may assume without loss of generality that $A$ is open. 

Supposing $A$ is open will allow us to discretize Lemma~\ref{lem:aperiodic} and make use of ideas from additive combinatorics. The key ingredient we need is the following result from sumset theory, due to Lev~\cite[Corollary 1]{Lev2010}.

\begin{theorem} \label{thm:sumset}
Let $n, \ell \ge 1$ and $M \ge 3$ be integers. Let $B \subseteq \mathbb{Z} \cap [0,\ell]$ be a set containing at least $M$ elements that is not contained in an arithmetic progression with difference greater than one. If $n \ge 2\ceil{(\ell-1)/(M-2)}$, then $nB$ contains a block of consecutive integers of length $n(M-1)$.
\end{theorem}

\begin{proof}[Proof of Lemma~\ref{lem:aperiodic}]

Let $A \subseteq [0,L]$ satisfy $|A| \ge \eta$. As discussed above at the start of this section, we may assume that $A$ is open. Therefore, there is a finite collection of disjoint open intervals $\{U_i\}_{i=1}^m$ such that 
\begin{equation}\label{eq:eta/2} A':=\bigcup_{i=1}^m U_i \subseteq A \quad \text{and} \quad \sum_{i=1}^m |U_i| \ge \frac{\eta}{2}. 
\end{equation}
Let $N \in \N$ large enough so that $N^{-1} \le |U_i|/20$ for every $1 \le i \le m$. We partition $[0,L)$ into the intervals
$$ I_j = \left[\frac{j}{N}, \quad \frac{j+1}{N}\right), \quad 0 \le j \le LN-1$$
and define $B \subseteq \Z \cap [0,LN-1]$ by  $B = \{j: I_{j} \subseteq A' \}.$ Observe that the relationship between $A'$ and $B$ is such that 
\begin{equation}\label{eq:A'B}
	nA'  \supseteq \bigcup_{j\in nB} I_j,
\end{equation}
where here the definition of $I_j$ is extended to all $j \in \N \cup \{0\}.$  We now check the hypotheses of Theorem~\ref{thm:sumset} with $B$ as defined above, $\ell = LN-1$, and $M = \# B$. We first note that the choice of $N$ ensures that each $U_i$ contributes at least $\lfloor 18|U_i| N/20 \rfloor$ elements to $B$. Thus, 
\begin{equation} \label{eq:Bcardprelim}
	\# B \ge \sum_{i=1}^m \left\lfloor \frac{9 |U_i| N}{10}\right \rfloor \ge \sum_{i=1}^m \left(\frac{9|U_i|N}{10} -1\right) \ge \sum_{i=1}^m \frac{4|U_i| N}{5} \ge \frac{2 N \eta}{5},
\end{equation}
where in the third inequality we used that $N|U_i| \ge 20$ implies that $|U_i|N/10 - 1 \ge 0$ and in the final inequality we recalled \eqref{eq:eta/2}. Combining \eqref{eq:Bcardprelim} with $N\eta \ge 20$, we have
\begin{equation}\label{eq:Bcard}
\# B \ge \frac{2N \eta}{5} \vee 8,
\end{equation}
and hence
\begin{equation} \label{eq:nceilbound}
	\left\lceil\frac{LN-2}{\#B - 2}\right\rceil \le \frac{2LN-4}{\#B} + 1\le \frac{10L}{\eta} + 1
\end{equation}
where in the first inequality we used that $\# B - 2 \ge \# B/2$ since $\# B \ge 8$.
Regarding the arithmetic progression assumption, the fact that $N^{-1} \le |U_i|/10$ also implies that if $j \in B$, then either $j-1 \in B$ or $j+1 \in B$. Thus, $B$ cannot be contained in a nontrivial arithmetic progression. It then follows from \eqref{eq:Bcard}, \eqref{eq:nceilbound}, and Theorem~\ref{thm:sumset} that for 
\begin{equation} \label{eq:ncase1}
	n = \frac{20L}{\eta} +2
\end{equation} 
the $n$-fold sumset of $B$ contains a block of consecutive integers of length 
$$ n (\#B - 1) \ge \frac{n \ \#B}{2} \ge \frac{10 L \, 
\#B}{\eta} \ge 4NL \ge 4N. $$
Therefore, by \eqref{eq:A'B}, for $n$ given by \eqref{eq:ncase1} there exists an interval $I \subseteq \RR$ with $|I| \ge 4$ such that $I\subseteq nA'$. Since this choice of $n$ depends only on $\eta$ and $L$, this completes the proof.
\end{proof}

\section{Uniform $H^s$ Bounds}
\label{sec:unifhor}

The goal of this section is to state a quantitative version of H\"{o}rmander's fractional smoothing estimates for parameter-dependent, second-order hypoelliptic operators~\cite{Hor_67}.  In the context of the original dynamics~\eqref{eqn:SDEmain}, these results imply uniform-in-$\epsilon$ fractional smoothing for the associated transitions.  It is important to remark that such quantitative estimates were done previously in~\cite{BL_21}.  In this section, we review these results and state them as needed in our context.

It should be mentioned that in the process of writing this paper, we discovered that Kohn's shorter argument of H\"{o}rmander's theorem as presented in~\cite{Hor_07} can be followed to prove part of the main regularity estimate below.  However, the sheer strength of H\"{o}rmander's estimates in~\cite{Hor_67} were difficult to reproduce, especially because the argument need not refer to a specific operator until bounding the sum of the \emph{hypoelliptic norm} and its dual 
\begin{align*}
\vertiii{u}_\delta+\vertiii{u}_{*,\delta}. 
\end{align*}
See their definitions below in the righthand side of~\eqref{eqn:Hssmooth2}.  On the other hand, Kohn's argument~\cite{Hor_07} uses the specific form of the second-order operator throughout the regularity estimates, which makes it less malleable to generalizations than H\"{o}rmander's proof.

Below, we will work in a slightly more general setting than in the rest of the paper.  
That is, let
\begin{align}
K\subset \mathcal{O}\subset  \RR^d
\end{align}
where $K$ is compact and $\mathcal{O}$ is bounded and open.  We use the notation $g\lesssim h$ if there exists a constant $C>0$ independent of $\epsilon \in (0,1)$  such that 
\begin{align}
g \leq C h.  
\end{align}
For brevity, we let 
\begin{align}
\| \cdot \|= \| \cdot \|_{L^2(\RR^d)} \,\,\, \text{ and } \,\,\, \langle\, \cdot \,, \, \cdot\, \rangle = \langle\, \cdot \,, \cdot\, \rangle_{L^2(\RR^d)}
\end{align}
and set $T(\mathcal{O})=T_d(\mathcal{O})$, $S(\mathcal{O})=S_d(\mathcal{O})$ where $T_d(\mathcal{O})$ and $S_d(\mathcal{O})$ are as introduced in Section~\ref{sec:mainresults}.

 If $X\in T(\mathcal{O})$ and $x\in K$, then $e^{tX}x$ denotes the local solution, including both positive and negative times, of the initial value problem
 \begin{align}
 \label{eqn:ODEX}
 \begin{cases}\frac{d}{dt} ( e^{tX} x) = X(e^{tX} x)\\
e^{tX} x\big|_{t=0}=x .
 \end{cases}
 \end{align}  
 Since $X\in T(\mathcal{O})$ and $K\subset \mathcal{O}$ is compact, there exists $t_0>0$ small enough so that $(x,t) \mapsto e^{tX} x\in C^\infty(K \times (-t_0, t_0); \mathcal{O})$.  Furthermore, the group property
 \begin{align*}
 e^{tX} (e^{sX} x)= e^{(t+s) X} x
 \end{align*}  
 is satisfied 
 whenever $x\in K$ and $ |s|+|t| <t_0$.  Observe that if $X\in S(\mathcal{O})$, then we can choose $t_0>0$ small enough and independent of $\epsilon \in (0,1)$ so that $(x,t) \mapsto e^{tX} x \in C^\infty(K \times (-t_0, t_0) ; \mathcal{O})$.  For functions $u:\mathcal{O}\rightarrow \RR$ and $x\in K$, define
 \begin{align}
 e^{tX} u(x):= u(e^{tX}x).
 \end{align}

 Central to the statements in this section are Besov norms which measure fractional regularity along a given vector field.  That is, for all $X\in T(\mathcal{O})$, $t_0>0$ sufficiently small, and $s\in (0,1]$ define
 \begin{align}
 \label{eqn:Fbes}
 |u|_{X,s}^{t_0}:= \sup_{0<|t| \leq t_0} \frac{\| e^{tX} u - u \|}{|t|^s}, \qquad u \in C_0^\infty(K),
 \end{align}
 where $C_0^\infty(K)$ denotes the class of $C^\infty(\mathcal{O})$ functions supported in $K$.  
We additionally make use of the following norm which measures fractional smoothing along every direction:
 \begin{align}
 \label{eqn:homognorm}
 |u|^{t_0}_s = \sup_{0<|h| \leq t_0} \frac{ \| u(\cdot + h) - u(\cdot) \|}{|h|^s}, \qquad u \in C_0^\infty(K), \end{align}  
 which is defined also for $t_0>0$ sufficiently small.  Unless we need to emphasize it, below we will drop the use of $t_0$ explicitly, recalling that for $X\in S(\mathcal{O})$, $t_0>0$ can be chosen independent of $\epsilon \in (0,1)$.

Fix vector fields 
\begin{align}
X_0, X_1, \ldots, X_r \in S(\mathcal{O}).  
\end{align}
Note that, possibly, each of $X_0, X_1, \ldots, X_r$ could depend on $\epsilon \in (0,1)$.  For any multi-index $I=(i_1, \ldots, i_k)$, $0 \leq i_j \leq r$, let
 \begin{align*}
|I|:=k \qquad \text{ and } \qquad X_I := \text{ad}\, X_{i_k} \text{ad} \, X_{i_{k-1}} \cdots \, \text{ad}\, X_{i_2}\, X_{i_1},
 \end{align*}
 where 
 \begin{align}
 \text{ad}X(Y):=[X,Y]
 \end{align}
  denotes the commutator of the vector fields $X,Y \in T(\mathcal{O})$.  By definition of the space $S(\mathcal{O})$, note that each $X_I \in S(\mathcal{O})$.  Associated to the vector fields $X_0, X_1, \ldots, X_r$ are the \emph{smoothing parameters} 
  \begin{align}
  \label{eqn:sj}
  s_0, s_1, \ldots, s_r\in (0, 1],
  \end{align}
  which we assume are independent of $\epsilon \in (0,1)$. 
Intuitively, they should be thought of as the presumed level of regularity expected along the respective directions $X_0, X_1, \ldots, X_r$.   
 These parameters define a stratified space of commutators of vector fields.  That is, define $S^s(\mathcal{O})$ to be the $S_F(\mathcal{O})$ submodule of $S(\mathcal{O})$ generated by all $X_I$ such that $s(I) \geq s$, where $s(I)$ is defined by the relation
\begin{align}
\label{eqn:sDEF}
\frac{1}{s(I)}= \sum_{j=1}^k \frac{1}{s_{i_j}} 
\end{align}
whenever $I= (i_1, \ldots, i_k)$. 

\begin{example}
For our purposes, the most important example is when $s_0=1/2$ and $s_1=s_2=\cdots =s_r =1$.  In this case, any vector field $X$ of the form 
\begin{align*}
X= \varphi_1 X_1 + \cdots + \varphi_r X_r 
\end{align*}
belongs to $S^1(\mathcal{O})$ provided $\varphi_i \in S_F(\mathcal{O})$.  Also, for example, any vector field of the form 
\begin{align*}
X= \varphi_0 X_0 +  \varphi_1 X_1 + \cdots + \varphi_r X_r  + \varphi_{0,1}[X_0, X_1]
\end{align*}  
for $\varphi_i , \varphi_{0,1}\in S_F(\mathcal{O})$ belongs to $S^{1/3}(\mathcal{O})$ since each vector field above in the sum belongs to, respectively, $S^{1/2}(\mathcal{O}), S^1(\mathcal{O}), \ldots, S^1(\mathcal{O}), S^{1/3}(\mathcal{O})$ and the containment 
\begin{align}
S^s(\mathcal{O})\subset S^t(\mathcal{O}) \qquad \text{ for }s\geq t
\end{align}
holds. 
\end{example}

\
The proof of H\"{o}rmander's theorem~\cite{Hor_67} essentially splits into two key parts.  The first part assumes a given fractional regularity (in the Besov sense in~\eqref{eqn:Fbes}) along each $X_j$, i.e. the $s_j$ above in~\eqref{eqn:sj}, and shows that there exists fractional smoothing of order $s$ along any member of $S^s(\mathcal{O})$.  Assuming $S^s(\mathcal{O})=S(\mathcal{O})$ for some $s\in (0,1]$, this in turn implies fractional smoothing $s\in (0,1]$ among all directions.

\begin{theorem}
\label{thm:Hssmooth1}
Suppose that $X_0, X_1, \ldots, X_r \in S(\mathcal{O})$, and fix $s_0, s_1, \ldots, s_r \in (0,1]$.  Let $\sigma >0$. Then for all $X\in S^s(\mathcal{O})$ 
\begin{align}
\label{eqn:Hssmooth1}
|u|_{X,s} \lesssim \sum_{j=0}^r |u|_{X_j, s_j} + |u|_\sigma
\end{align} 
for all $u\in C^\infty_0(K)$.  In particular, if $s\in (0,1]$ is independent of $\epsilon \in (0,1)$ and is such that $S^s(\mathcal{O})=S(\mathcal{O})$, then (by choosing $\sigma < s$ in~\eqref{eqn:Hssmooth1}) 
\begin{align}
\label{eqn:Hssmooth23}
|u|_s \lesssim \sum_{j=0}^r |u|_{X_j, s_j} + \|u\|
\end{align} 
for all $u\in C^\infty_0(K)$.  
\end{theorem}

\begin{remark}
Note by definition of the norm $|\cdot |_s$, if $\sigma < s$, then by choosing $t_0$ small enough
\begin{align*}
|u|_\sigma \leq \delta |u|_s +C_\delta \|u\|.  
\end{align*}
Relation~\eqref{eqn:Hssmooth23} follows from~\eqref{eqn:Hssmooth1} from this line of reasoning.
\end{remark}
The proof of Theorem~\ref{thm:Hssmooth1} follows the corresponding argument in H\"{o}rmander's original paper~\cite{Hor_67} almost identically, except that each inequality $\leq$ is replaced by $\lesssim$.  This proof uses repeated applications of the Baker-Cambell-Hausdorff-Dynkin formula to deduce representations of flows along brackets in terms of $X_0, X_1, \ldots, X_r$ and a remainder term.  (See~\cite{BL_21}). 

The condition that $S^s(\mathcal{O})=S(\mathcal{O})$ for some $s\in (0,1]$ independent of $\epsilon \in (0,1)$ is crucial in this paper.  Effectively, it is a uniform-in-$\epsilon$ spanning condition on the Lie algebra generated by $X_0, X_1, \ldots, X_r \in S(\mathcal{O})$.      
\begin{definition}
\label{def:gunifhor}
We say that $X_0, X_1, \ldots, X_r\in S(\mathcal{O})$ satisfies the \emph{uniform H\"{o}rmander condition on} $\mathcal{O}$ \emph{with respect to} $\epsilon \in (0,1)$ if there exists $s\in (0,1]$ independent of $\epsilon \in (0,1)$ such that \begin{align*}
S^s(\mathcal{O})=S(\mathcal{O}).
\end{align*}  
\end{definition}

\begin{remark}
The uniform H\"{o}rmander condition on $\mathcal{O}$ with respect to $\epsilon \in (0,1)$ is nearly identical to the uniform parabolic H\"{o}rmander condition on $\mathcal{O}$ with respect to $\epsilon \in (0,1)$ as in Definition~\ref{def:unifhor}, except that the list of vector fields in Definition~\ref{def:unifhor} is slightly reduced to deal with both the parabolic and elliptic settings simultaneously.  
\end{remark}

\begin{assump}
\label{assump:unifhor}
The list $X_0, X_1, \ldots, X_r \in S(\mathcal{O})$ satisfies the uniform H\"{o}rmander's condition on $\mathcal{O}$ with respect to $\epsilon \in (0,1)$.   
\end{assump}
As an immediate corollary of Theorem~\ref{thm:Hssmooth1} we obtain:
\begin{corollary}
Fix $s_0, s_1, \ldots, s_r \in (0,1]$ independent of $\epsilon \in (0,1)$, and suppose that $X_0, X_1, \ldots, X_r \in S(\mathcal{O})$ satisfy Assumption~\ref{assump:unifhor}.  Then there exists $s\in (0,1]$ independent of $\epsilon \in (0,1)$ such that 
\begin{align}
|u|_s \lesssim \sum_{j=0}^r |u|_{X_j, s_j}  + \|u\|
\end{align}
for all $u\in C_0^\infty(K)$. 
\end{corollary}

The second key, and much more subtle, part of the proof of H\"{o}rmander's theorem~\cite{Hor_67} shows that if $s_0=1/2, \, s_1=s_2=\cdots = s_r=1$, then the righthand side of~\eqref{eqn:Hssmooth23} can be controlled by the \emph{hypoelliptic} norm $\vertiii{u}$ given by
\begin{align}
\vertiii{u}^2= \vertiii{u}^2_\mathcal{O}=\|u\|_{L^2(\mathcal{O})}^2+ \sum_{1}^r \| X_j u \|^2_{L^2(\mathcal{O})}
\end{align} 
and its associated dual norm 
\begin{align}
\vertiii{u}_*= \vertiii{u}_{*, \mathcal{O}}= \sup_{\substack{ \varphi \in C^\infty_0(\mathcal{O})\\ \vertiii{\varphi} \leq 1}}\int \varphi u \, dx.\end{align}
In the De Giorgi iteration scheme used below in Section~\ref{sec:lowerb}, it will be helpful to have a slight modification of these norms.  That is, for $u\in C_0^\infty(K)$ and a parameter $\delta \in [0,1]$ independent of $\epsilon \in (0,1)$, define\begin{align}
 \vertiii{u}_{\delta}^2 =\vertiii{u}^2+ \delta\|u \|_{L^\infty(\mathcal{O})}^2 \qquad \text{ and } \qquad \vertiii{u}_{\delta,*} =  \sup_{\substack{ \varphi \in C^\infty_0(\mathcal{O})\\ \vertiii{\varphi}_\delta \leq 1}}\int \varphi u \, dx. \end{align}
 \begin{theorem}
 \label{thm:Hssmooth2}
Let $\delta \in [0,1]$ and suppose that $s_0=1/2, \, s_1=s_2=\cdots = s_r=1$ and $X_0, X_1, \ldots, X_r \in S(\mathcal{O})$ satisfy Assumption~\ref{assump:unifhor}.  Then
\begin{align}
\label{eqn:Hssmooth3}
\sum_{j=0}^r |u|_{X_j, s_j} \lesssim \vertiii{u}_\delta + \vertiii{X_0 u}_{\delta, *} 
\end{align}  
for all $u\in C^\infty_0(K)$. 
\end{theorem}
\begin{remark}
Observe that we have yet to make reference to an underlying second-order differential operator defined by the vector fields $X_0, X_1, \ldots, X_r$.  As mentioned in the introduction of this section, this gives H\"{o}rmander's argument in~\cite{Hor_67} its strength over Kohn's argument presented in~\cite{Hor_07}.  Indeed, in Kohn's argument, the form of the operator is employed throughout to control various Sobolev norms of brackets.  In particular, it was not obvious to us that we could use this shorter argument to prove the results of with a general $\delta \in (0,1]$.  On the other hand in H\"{o}rmander's proof~\cite{Hor_67}, the form of the differential operator is only employed on the last step to conclude fractional regularity by estimating the hypoelliptic norm and its dual using basic apriori bounds satisfied by such operators.  This allows for the generalization stated above.    

 \end{remark}  
 
 We remark again that Theorem~\ref{thm:Hssmooth2} follows again from the same details in H\"{o}rmander's paper~\cite{Hor_67} (see~\cite{BL_21}) but with $\leq$ replaced by $\lesssim$.  
 
 Combining Theorem~\ref{thm:Hssmooth1} and Theorem~\ref{thm:Hssmooth2} produces the following result which we use below in our context~\eqref{eqn:SDEmain}.
 \begin{theorem}
 \label{thm:Hssmooth3}
 Let $\delta \in [0,1]$ and suppose that $s_0=1/2, \, s_1=s_2=\cdots = s_r=1$ and $X_0, X_1, \ldots, X_r \in S(U)$ satisfy Assumption~\ref{assump:unifhor}.  Then there exists an $s\in (0,1]$ independent of $\epsilon \in (0,1)$ such that
\begin{align}
\label{eqn:Hssmooth2}
 |u|_{s} \lesssim \vertiii{u}_\delta + \vertiii{X_0 u}_{\delta,*} 
\end{align}  
for all $u\in C^\infty_0(K)$. 
 \end{theorem}
\begin{remark}
Given the conclusion of Theorem~\ref{thm:Hssmooth3}, it follows by comparison of norms (see~\cite{Hor_07}) that for any $\sigma< s$ we have
\begin{align}
\| u \|_{H^{\sigma}(\RR^n)} \lesssim \vertiii{u}_\delta + \vertiii{X_0 u}_{\delta,*}
\end{align}
for all $u \in C^\infty_0(K)$, where $\|u\|_{H^\sigma(\RR^n)}$ is the usual Sobolev norm, defined for $u\in C^\infty_0(\RR^n)$ by
\begin{align*}
\|u\|_{H^\sigma(\RR^n)}^2:= \int_{\RR^n} |\hat{u}(\xi)|^2 (1+ 4\pi^2 |\xi|^2)^{\sigma} \, d\xi, 
\end{align*} 
where
\begin{align*}
\hat{u}(\xi) = \int_{\RR^n} e^{-2\pi i \xi \cdot x} u(x) \, dx
\end{align*}
is the Fourier transform of $u$.  Here, note that we naturally extend the definition of $u\in C^\infty_0(K)$ to all of $\RR^n$ by taking $u\equiv 0$ outside of $K$. 
\end{remark}

We conclude this section by stating a result that will be useful in the parabolic setting.  Although we do not quite get as much fractional smoothing in time, it will still be necessary. To state it, we let $S^{s,0}(\mathcal{O})$ denote the $S_F(\mathcal{O})$ submodule generated by all $X_I \in S(\mathcal{O})$ with $X_I \neq X_0$ and $s(I) \geq s$.  Note here we are allowed to include $X_0$, so long as it first occurs in a bracket, e.g. $[X_1, X_0]\in S^{1/3,0}(\mathcal{O})$.  Observe that the condition $S^{s,0}(\mathcal{O})=S(\mathcal{O})$ for some $s>0$ independent of $\epsilon \in (0,1)$ is precisely the uniform parabolic H\"{o}rmander condition for $\{ X_0 ; X_1; \ldots; X_r\}$ on $\mathcal{O}$ with respect to $\epsilon \in (0,1)$ as in Definition~\ref{def:unifhor}.         
\begin{theorem}[Time-dependent H\"{o}rmander]
\label{thm:tdhor}
Suppose $s_0=1/2, s_1=s_2=\cdots = s_r=1$ and fix $a,b\in \RR$ independent of $\epsilon \in (0,1)$ with $0<a<b$.  Let $\eta \in (0, (b-a)/4]$ and suppose that there exists $s\in (0,1)$ independent of $\epsilon \in (0,1)$ such that $S^{s, 0}(\mathcal{O}) =S(\mathcal{O})$.  Then, for $\sigma<s$,
\begin{align}
\label{eqn:tdhor1}
\int_a^b \| u(t) \|_{H^\sigma(\RR^n)}^2 \,d t  \lesssim \int_a^b \vertiii{u(t)}^2 \, d t + \bigg(\sup_{\substack{\phi \in C^\infty_0((a,b) \times \mathcal{O})\\\int_a^b \vertiii{\phi}^2 \, dt \leq 1}} \int_{(a,b) \times \mathcal{O}} \phi (\epsilon \partial_t +X_0)u \, dt \, dx\bigg)^2 
\end{align} 
for all $u\in C^\infty_0([a+ \eta, b-\eta] \times K)$. For any fixed $\eta_0 \le (b-a)/4$, the implicit constant in the inequality is uniformly bounded with respect to $\eta$ for $\eta \in [\eta_0, (b-a)/4]$.
\end{theorem}

\begin{remark}
Note that Theorem~\ref{thm:tdhor} does not follow immediately from the statement of Theorem~\ref{thm:Hssmooth3} since, even when the vector fields $X_1, \ldots, X_r$ are extended to vector fields on $(a,b) \times \mathcal{O}$, the Lie algebra generated by $\epsilon \partial_t + X_0, X_1, \ldots, X_r$ does not uniformly span the tangent space.  If it did, then we would be able to replace the lefthand side of~\eqref{eqn:tdhor1} by the $H^\sigma$ norm on $\RR^{n+1}$.  However, as noticed in~\cite{BL_21}, one still obtains the claimed control of the integrated $H^\sigma$ norm in time. 
\end{remark}

\section{Uniform upper bounds}
\label{sec:Moser}

The goal of this section is to prove Theorem~\ref{thm:upperboundo}.  The proof of this result naturally splits into two separate lemmas, which we now state.  Recall that $q_t^\epsilon = p_{t/\epsilon}^\epsilon$ denotes the time-changed density and $\chi_R$ is the smooth cutoff defined in \eqref{eq:cutoff}.  Below, we will use the notation $\lesssim_{R, \alpha}$ to indicate that the bound is independent of $\epsilon \in (0,1)$ but depends on the parameters $R, \alpha$.

\begin{lemma}
\label{lem:upper1}
Suppose that Assumptions~\ref{assump:H}-\ref{assump:BR} hold, and let $R>0$ and $\alpha >\tfrac{1}{2}\|\mathrm{div}(Z)\|_{L^\infty(\mathcal{X})} $.  Then, 
  \begin{align}
  \label{eqn:unifupper}
\sup_{(t,x,y)\in [1, \infty) \times H_R \times H_R} (e^{-\alpha t} q_t^\epsilon(x,y) )  \lesssim_{R, \alpha} 1.
\end{align}\end{lemma}

\begin{lemma}
\label{lem:upper2}
Suppose that Assumptions~\ref{assump:H}-\ref{assump:BR} hold, and let $1 <T< \infty$ and $R>0$. Then,
\begin{align}
\int_{1}^T \| \chi_R q_t^\epsilon \|^2_{H^s( \RR^d \times \RR^d)} \, dt \lesssim_{R, T} 1. 
\end{align} 
\end{lemma}

Observe that Theorem~\ref{thm:upperboundo} follows immediately after combining Lemma \ref{lem:upper1} with Lemma \ref{lem:upper2}.  Lemma~\ref{lem:upper1} will be established after proving a quantitative $L^2$-apriori estimate with respect to some `natural' reference measure (Proposition~\ref{prop:apriorigreen} below) and a quantitative Moser iteration scheme (Theorem~\ref{thm:moser2} below), which in particular upgrades the $L^2$ bound to an $L^\infty$ estimate.  Lemma~\ref{lem:upper2} will follow by using a combination of said results with the quantitative H\"{o}rmander result (Theorem~\ref{thm:tdhor}).

To set up the statement of the apriori bound, fix $\alpha > \tfrac{1}{2}\|\text{div}(Z)\|_{L^\infty(\mathcal{X})}$ and for $f\in  C^\infty_0(\mathcal{X})$ define 
\begin{align}
\label{eqn:gdef}
g_\alpha = g_\alpha(f,t, x):= e^{-\alpha t} \mathcal{Q}_t^\epsilon f(x),  
\end{align}
where 
\begin{align}
\mathcal{Q}_t^\epsilon := \mathcal{P}_{t/\epsilon}^\epsilon, 
\end{align}
i.e., $\mathcal{Q}_t^\epsilon$ is the time-changed Markov semigroup.  Observe that the basic estimate 
\begin{align} \label{eq:galphaupper}
|g_\alpha(f,t,x)| \leq  e^{-\alpha t}\| f\|_{L^\infty(\mathcal{X})} 
\end{align}
 is satisfied for $(t,x) \in (0, \infty) \times \mathcal{X}$.  Moreover, under Assumptions~\ref{assump:H} and~\ref{assump:V}, $(t,x) \mapsto g_\alpha(f, t,x)\in C^\infty((0, \infty) \times \mathcal{X})$ and $g_\alpha$ satisfies the equation 
\begin{align}
\label{eqn:geq}
(\mathcal{L}_\epsilon - \alpha - \partial_t) g_\alpha = 0 \quad \text{ on } \quad  (0, \infty) \times\mathcal{X}
\end{align}
in the classical sense, 
where 
\begin{align}
\mathcal{L}_\epsilon := \epsilon^{-1} L_\epsilon = -Z + \frac{1}{\epsilon}Z_0 + \sum_{j=1}^r Z_j^2
\end{align}
 is the time-changed generator.

\begin{proposition}[Apriori estimate]
\label{prop:apriorigreen}
Suppose that Assumptions~\ref{assump:H} and~\ref{assump:V} are satisfied, and let $f\in C^\infty_0(\mathcal{X})$.  Then for any $\alpha >\tfrac{1}{2} \| \text{\emph{div}}(Z) \|_{L^\infty(\mathcal{X})}$ we have 
\begin{align}
\label{eqn:apriori}
\int_{\mathcal{X}\times(0, \infty)} (\alpha - \tfrac{1}{2}\text{\emph{div}}(Z)) g_\alpha^2 e^{-\eta H} \, dx dt + \sum_{j=1}^r \int_{\mathcal{X}\times (0,\infty)} (Z_j g_\alpha)^2 e^{-\eta H} \, dx dt \leq  \int_{\mathcal{X}} f^2 e^{-\eta H} \, dx 
\end{align}
where $g_\alpha= g_\alpha(f,t, x)$ is as in~\eqref{eqn:gdef} and $\eta >0$ is as in Assumption~\ref{assump:V} (V4). 
\end{proposition}

\begin{remark}
Proposition~\ref{prop:apriorigreen} is a both a global quantitative smoothing estimate for $g_\alpha$ (i.e. the $(Z_j g_\alpha)^2$ terms in~\eqref{eqn:apriori}) and a global quantitative apriori estimate on the $L^2$-norm of $g_\alpha$, with respect to the measure $e^{-\eta H} \, dx \, dt$.  Equivalently, we can view these bounds as a quantitative estimate on the $L^2$-norm of $\mathcal{Q}^\epsilon_t f$ and $Z_j \mathcal{Q}^\epsilon_t f$ with respect to the finite measure $e^{-2\alpha t} e^{-\eta H} \, dx \, dt$.  The weights $e^{-2\alpha t}$ and $e^{-\eta H}$ are useful in producing global bounds instead of local bounds (i.e. over bounded sets of time and space). While local bounds alone would be sufficient for our purposes, using a reference measure depending on $H$ is quite natural here. 
\end{remark}
\begin{proof}[Proof of Proposition~\ref{prop:apriorigreen}]
Let $\varphi \in C^\infty_0(\RR; [0,1])$ be such that 
\begin{align*}
\varphi(x) = \begin{cases}
1 & \text{ if } x \leq 1\\
0 & \text{ if } x >2
\end{cases}
\end{align*}
and $\varphi'(x) \le 0$, and define $\varphi_R: \mathcal{X}\rightarrow [0,1]$ by 
\begin{align}
\varphi_R(x) = \varphi(H(x) R^{-1}) e^{-\eta H(x)},
\end{align}
where we again recall that $\eta >0$ is as in Assumption~\ref{assump:V} (V4). 
Let $ \phi_S \in C^\infty([0,\infty); [0,1])$ be such that $ \phi_S(t) = 1 $ for $t \leq S $, $ \phi_S(t) = 0 $ for $t \geq S+1$, and $\|\phi'_S\|_{L^\infty([0,\infty))} \lesssim 1$ uniformly in $S$. Then, integrating by parts and using Assumption~\ref{assump:V} (V2) and (V3) produces
\begin{align*}
0=&\int_{\mathcal{X}\times (0,\infty)} g_\alpha(\mathcal{L}_\epsilon-\alpha-\partial_t)g_\alpha \phi_S\varphi_R \, dx dt 
\\
&=\int_{\mathcal{X}\times (0,\infty)} (-\alpha+\tfrac{1}{2}\text{div}(Z)) g_\alpha^2 \phi_S\varphi_R \, dx dt+\tfrac{1}{2}\int_{\mathcal{X}} f^2 \varphi_R \, dx \\
& \qquad+ \tfrac{1}{2}\sum_{j=1}^r \int_{\mathcal{X}\times (0,\infty)} g_\alpha^2 \phi_SZ_j^2 \varphi_R \, dx dt+\tfrac{1}{2} \int_{\mathcal{X}\times (0,\infty)} g_\alpha^2 \phi'_S\varphi_R \, dxdt\\&\qquad - \sum_{j=1}^r \int_{\mathcal{X}\times (0,\infty)} (Z_j g_\alpha)^2\phi_S \varphi_R \, dx dt+ \frac{1}{2}\int_{\mathcal{X}\times (0,\infty)} g_\alpha^2 \phi_SZ\varphi_R \, dx dt.
 \end{align*}
 Rearranging the above then implies 
 \begin{align*}
 &\int_{\mathcal{X}\times (0,\infty)} ( \alpha - \tfrac{1}{2}\text{div}(Z))g_\alpha^2 \phi_S\varphi_R \, dxdt + \sum_{j=1}^r \int_{\mathcal{X}\times (0,\infty)} (Z_j g_\alpha)^2 \phi_S\varphi_R \, dx dt \\
 &= \tfrac{1}{2}\int_{\mathcal{X}} f^2 \varphi_R \, dx+\tfrac{1}{2} \int_{\mathcal{X}\times (0,\infty)} g_\alpha^2 \phi'_S\varphi_R \, dxdt + \tfrac{1}{2}\int_{\mathcal{X}\times (0,\infty)} g_\alpha^2 \phi_SZ\varphi_R \, dx dt  \\
 &\qquad + \tfrac{1}{2}\sum_{j=1}^r \int_{\mathcal{X}\times (0,\infty)} g_\alpha^2 \phi_S Z_j^2 \varphi_R \, dxdt.\end{align*}
First, by \eqref{eq:galphaupper} and $e^{-\eta H} \in L^1(\mathcal{X})$, we have
\begin{align*}
 \bigg|\int_{\mathcal{X}\times (0,\infty)} g_\alpha^2 \phi'_S\varphi_R \, dxdt\bigg| \lesssim \|f\|_{L^\infty(\mathcal{X})}^2 \int_{\mathcal{X} \times (S, S+1)} e^{-2\alpha t} e^{-\eta H(x)} \, dx\rightarrow 0 
\end{align*}
as $S\rightarrow \infty$. Next, observe that the inequalities in Assumption~\ref{assump:V} (V4) imply that 
$$(Z + \textstyle{\sum}_{j=1}^r Z_j^2) e^{-\eta H} \le 0 \quad \text{and} \quad ZH \ge -d_*/2.$$  
Thus, by direct computation we have 
\begin{align*}
(Z+ \textstyle{\sum}_{j=1}^r Z_j^2) \varphi_R &= \varphi(H R^{-1}) (Z+ \textstyle{\sum}_{j=1}^r Z_j^2)(e^{-\eta H}) + \frac{\varphi'(H R^{-1})}{R} ZH e^{-\eta H} \\
&\qquad+\textstyle{ \sum_{j=1}^r} R^{-1} \varphi'(HR^{-1}) Z_j^2 H e^{-\eta H}\\
&\qquad +\textstyle \sum_{j=1}^r \big\{\frac{\varphi''(HR^{-1})}{R^2}-2\eta \frac{\varphi'(H R^{-1})}{R}\big\} (Z_j H)^2 e^{-\eta H}\\
& \leq \textstyle{ \sum_{j=1}^r} R^{-1} \varphi'(HR^{-1}) Z_j^2 H e^{-\eta H} + \frac{|\varphi'(HR^{-1})|}{R} \frac{d_*}{2} e^{-\eta H} \\
&\qquad +\textstyle \sum_{j=1}^r \big\{\frac{\varphi''(HR^{-1})}{R^2}-2\eta \frac{\varphi'(H R^{-1})}{R}\big\} (Z_j H)^2 e^{-\eta H} \\
&=: T_R\end{align*}
and 
\begin{align*}
|T_R|\lesssim \frac{1}{R} \sum_{j=1}^r( |Z_j^2 H| + (Z_j H)^2)e^{-\eta H} .
\end{align*}
The claimed bound now follows after letting $S, R\rightarrow \infty$ and employing the integrability part of Assumption~\ref{assump:V} (V4).
\end{proof}

Next, letting more generally $\alpha \in \RR$, we turn to the proof of a quantitative Moser iteration for $u\in C^\infty((0, \infty) \times \mathcal{X}); [0, \infty))$ satisfying 
\begin{align}
\label{eqn:Green}
(\mathcal{M}_\epsilon - \alpha - \partial_t) u \geq 0  \,\,\,\text{ on }\,\,\, (0, \infty ) \times \mathcal{X},
\end{align}
where $\mathcal{M}_\epsilon \in \{ \mathcal{L}_\epsilon, \mathcal{L}_\epsilon^{-} \}$, and $\mathcal{L}_\epsilon^{-}$ is the generator of the process $\{ x_{t/\epsilon}^{-, \epsilon} \}_{t\geq 0}$, i.e.,
\begin{align}
\label{eqn:gentr}
    \epsilon\mathcal{L}_\epsilon^-:=-Z_0+\epsilon Z+\epsilon\sum_{j=1}^rZ_j^2.
\end{align}
Proving the result for both $\mathcal{L}_\epsilon$ and $\mathcal{L}_\epsilon^{-}$ will be convenient when dealing with the formal adjoint $\mathcal{L}_\epsilon^*$ below.  

\begin{theorem}
\label{thm:moser2}
Suppose that Assumptions~\ref{assump:H}-\ref{assump:BR} hold, fix $1/2\le t_1 < \infty$, $t_1+1/2\leq t_2 \leq t_1+1$, and let $\alpha \in \RR, \, m= (t_2-t_1)/4$.  Then for every $u\in C^\infty((0, \infty) \times \mathcal{X})$ satisfying \eqref{eqn:Green} with $u\geq 0$ on $(t_1, t_2) \times H_{R}$, we have
\begin{align}
\label{eqn:moserest1}
\| u\|_{L^\infty((t_1+m, t_2-m)\times H_{R/2})}\lesssim_{R,\alpha} \| u\|_{L^2((t_1, t_2) \times H_{R})}.
\end{align}
\end{theorem}

In order to prove Theorem~\ref{thm:moser2}, it will be helpful to first establish two technical estimates, collected in the following lemma, which are used in its proof and the rest of this section.  \begin{lemma}
\label{lem:apriori2}
Suppose that Assumptions~\ref{assump:H}-\ref{assump:V} are satisfied and fix $R>0$ and $a,b\in \RR$ with $a<b$ independent of $\epsilon \in (0,1)$. 
\begin{itemize}
\item[(i)]  Let $\beta > 1$ and $w:(a,b)\times H_R^2\rightarrow (0, \infty)$ be measurable.  Then 
\begin{align*}
\int_a^b \int_{H_R^2} w^{2\beta-2} \, dx \, dy \, dt \leq 2 \max \Big\{ \frac{\beta}{\beta-1} \| w\|^{2\beta}_{L^{2\beta}((a,b) \times H_R^2)} , \beta |(a,b) \times H_R^2|\Big\}.  
\end{align*}    
\item[(ii)]  Suppose that $w\in C^\infty_0((a,b) \times H_R; [0, \infty))$ and $$(\mathcal{M}_\epsilon -\alpha - \partial_t)w\geq f$$
for some $f\in L^\infty((a,b) \times H_R)$ and $\alpha \in \RR $ independent of $\epsilon \in (0,1)$.  Then 
\begin{align*}
\sum_{j=1}^r  \| Z_j w\|^2_{L^2((a,b) \times H_R)} \lesssim_{R, a,b} \| w\|^2_{L^2((a,b) \times H_R)} + \| f\|^2_{L^2((a,b)\times H_R)}. 
\end{align*} 
\end{itemize}
\end{lemma}
\begin{proof}
For part (i), use H\"{o}lder's inequality and Young's inequality to see that   
\begin{align*}
\int_a^b \int_{H_R^2} w^{2\beta-2} \, dx \, dy \, dt & \leq \bigg(\int_a^b \int_{H_R^2} w^{2\beta} \, dx \, dy \, dt \bigg)^{(\beta-1)/\beta} |(a,b) \times H_R^2|^{1/\beta}\\
& \leq  \beta | (a,b) \times H_R^2| + \frac{\beta}{\beta-1} \| w\|_{L^{2\beta}((a,b) \times H_R^2)}^{2\beta}\\
& \leq 2 \max \Big\{ \frac{\beta}{\beta-1} \| w\|^{2\beta}_{L^{2\beta}((a,b) \times H_R^2)} , \beta |(a,b) \times H_R^{2}|\Big\}.  
\end{align*}

For part (ii), multiply $(\mathcal{M}_\epsilon - \alpha -\partial_t) w \geq f$ by $w$ and integrate over $(a,b) \times H_R$ to see that 
\begin{align*}
\int_{(a,b) \times H_R} f w \, dx \, dt&\leq \int_{(a,b) \times H_R} w(\mathcal{M}_\epsilon - \alpha -\partial_t) w\, dx \, dt\\
&= - \sum_{j=1}^r \int_{(a,b) \times H_R} |Z_j w|^2 \, dx \, dt + \int_{(a,b) \times H_R} (-\alpha \pm \text{div}(Z)) w^2 \, dx \, dt,
\end{align*}
where the $\pm$ depends on whether $\mathcal{M}_\epsilon = \mathcal{L}_\epsilon$ or $\mathcal{M}_\epsilon = \mathcal{L}_{\epsilon}^{-}$.  Rearranging the above inequality and using the fact that $\alpha, \text{div}(Z)$ are bounded uniformly in $\epsilon$ finishes the proof. 

\end{proof}

Below in the proof of Theorem~\ref{thm:moser2}, we fix 
\begin{align}
\{z_t^\epsilon\}_{t\geq 0} \in \{ \{x_{t/\epsilon}^{\epsilon}\}_{t\geq 0}, \{x_{t/\epsilon}^{-,\epsilon}\}_{t\geq 0}\}
\end{align}  
so that $\{z_t^\epsilon\}_{t\geq 0}$ corresponds to the choice of operator $\mathcal{M}_\epsilon$ as in~\eqref{eqn:Green}.  

\begin{proof}[Proof of Theorem~\ref{thm:moser2}]
Let $\gamma >0$ be a constant independent of $\epsilon \in (0,1)$ and consider $v=\frac{u}{\gamma} + 1$.  We will prove that there exists a constant $C>0$ depending only on $R>0$ and $|t_2-t_1|$ such that 
 \begin{align}
 \| v\|_{L^\infty((t_1+m, t_2-m)\times H_{R/2})}\leq C(\| v\|_{L^2((t_1, t_2)\times H_R )} + 1)
 \end{align}
  Multiplying the above inequality through by $\gamma$ and then taking $\gamma \downarrow 0$ implies the claimed result.  
  
  Let  
\begin{align}
w_k= v^{\beta^k}
\end{align}
 for some $\beta \in(1,2)$ to be determined.  Routine calculations show that, on $(0, \infty)\times \mathcal{X}$,  
\begin{align*}
(\mathcal{M}_\epsilon - \alpha\beta^k-\partial_t)w_k \geq -\alpha   \beta^{k} v^{\beta^k -1} .  
\end{align*}
For $k \ge 1$, let $R_k=\tfrac{R}{2}(1+2^{-k+1})$ and $s_k= m(1-2^{-k+1})$.  Let $I_k= (t_1+s_k, t_2-s_k)$ and suppose $\varphi_k \in C^\infty_0(\RR; [0,1])$, $\psi_k \in C^\infty_0(I_k; [0,1])$ are cutoff functions satisfying
\begin{align*}
\varphi_k(x) = \begin{cases}
1 &  x  \leq R_{k+1} \\
0 &  x\geq R_k
\end{cases}
\end{align*}  
with $|\varphi'_k|+ |\varphi''_k| \lesssim_{R} 2^{2k},$ and 
\begin{align*}
\psi_k(t) = \begin{cases}
1 &  t\in I_{k+1}  \\
0 &  t\in I_k^c
\end{cases}
\end{align*}  
with $|\psi'_k|+ |\psi''_k|\lesssim 2^{2k}$.  Abusing notation slightly, set 
\begin{align*}
\chi_k(t,x) = \varphi_k(H(x)) \psi_k(t).
\end{align*}
Define $v_k = w_k \chi_k$ and note, for any $\ell>0$,
\begin{align*}
(\mathcal{M}_\epsilon - \ell-\partial_t)v_k &= [(\mathcal{M}_\epsilon - \ell-\partial_t)w_k ]\chi_k + w_k( \mathcal{M}_\epsilon-\partial_t )\chi_k +2 \textstyle{\sum_{j=1}^r} Z_j w_k Z_j \chi_k\\
&\geq(\alpha \beta^k-\ell)v_k- \alpha  \beta^k v^{\beta^k -1} \chi_k+ w_k( \mathcal{M}_\epsilon-\partial_t )\chi_k +2 \textstyle{\sum_{j=1}^r} Z_j w_k Z_j \chi_k=:   S_k.
\end{align*}

We now apply Assumption~\ref{assump:BR} to find a bounded open set $\mathcal{O}_{R}$ with continuous boundary $\partial \mathcal{O}_{R}$ such that 
\begin{align*}
H_{R+1}\subset \mathcal{O}_{R} \subset \overline{\mathcal{O}_{R}} \subset \mathcal{X} 
\end{align*}
and such that $\mathcal{O}_R$ is boundary regular for $\{z_t^\epsilon \}_{t\geq 0}$.  For $(s, x) \in \RR \times \mathcal{X}$, define the time-extended process $\tilde{z}_t(s,x)=(s-t, z_t^\epsilon(x))$, where $z_t^\epsilon(x)$ is as before with $z_0(x)=x$. Observe that this process has generator $$\widetilde{\mathcal{M}}_\epsilon:=\mathcal{M}_\epsilon -\partial_t $$ 
and that the parabolic domain $(t_1/2, 2t_2)\times \mathcal{O}_R$ is boundary regular for $\tilde{z}_t(s,x)$ on the parts of the boundary $[t_1/2,2t_2]\times \partial \mathcal{O}_R$ and $\{t_1/2 \} \times \mathcal{O}_R$, but not on $\{2 t_2 \} \times \mathcal{O}_R$ (because the motion of the time process flows backwards). Below, we will drop the use of the initial condition $(s, x)$ explicitly in $\tilde{z}_t(s,x)$ and write $\tilde{z}_t$.  Define \begin{align}
\tilde{\tau}=  \inf\{ t > 0 \, : \tilde{z}_t\, \notin [t_1/2,2t_2]\times\overline{\mathcal{O}_R}\}. 
\end{align}  
and consider 
\begin{align*}
\tilde{g}(s,x) =- \EE_{(s,x)} \int_{0}^{\tilde{\tau}} e^{-\ell t} S_k(\tilde{z}_t) \, dt.  
\end{align*}

\emph{Claim}.   $v_k \leq \tilde{g}$ on $(t_1,t_2)\times \mathcal{O}_R$.  

\emph{Proof of Claim}.  
By Dynkin's formula, for $(s,x)\in (t_1,t_2)\times \mathcal{O}_R$ and any $t\geq 0$, we have
\begin{align*}
\EE_{(s,x)} [ e^{-\ell{t\wedge \tilde{\tau}}} v_k(\tilde{z}_{t\wedge \tilde{\tau}})] &= v_k(s,x) +\EE_{(s,x)} \int_0^{t\wedge \tilde{\tau}}  e^{-\ell u} (\tilde{\mathcal{M}}_\epsilon-\ell) v_k(\tilde{z}_u) \, du. 
\end{align*}  
However, $v_k(t,x) \equiv 0$ for $(t, x) \in \partial ((t_1/2, 2t_2) \times \mathcal{O}_R)$.  Thus taking $t\rightarrow \infty$ using the fact that $\tilde{\tau}$ is bounded a.s. produces 
\begin{align*}
0& = v_k(s,x) + \EE_{(s,x)} \int_0^{\tilde{\tau}} e^{-\ell u} (\tilde{\mathcal{M}}_\epsilon-\ell) v_k(\tilde{z}_u) \, du \geq v_k(s,x) + \EE_x \int_0^{\tilde{\tau}}  e^{-\ell u } S_k(\tilde{z}_u) \, du
\end{align*}
so that 
the claim follows.

Next, we note that~\cite[Theorem 5.4 and Theorem 6.8]{FH_23} together with the boundary regularity properties of $\tilde{z}_t$ imply that $\tilde{g} \in C^\infty((t_1/2,2t_2)\times\mathcal{O}_R)$, $\tilde{g}\in C([t_1/2,2t_2)\times\overline{\mathcal{O}_R})$, $\tilde{g}|_{t=t_1/2}=0$, $\tilde{g}|_{\partial \mathcal{O}_R}=0$, and 
\begin{align}\label{eq:tildegPDE}
(\mathcal{M}_\epsilon - \ell-\partial_t) \tilde{g} =  S_k
\end{align}
on $(t_1/2, 2t_2) \times \mathcal{O}_R$ in the classical sense.  Fix $t\in (t_1/2,2t_2)$, set $J_t=(t_1/2, t)$ and observe that 
by pairing $(\mathcal{M}_\epsilon - \ell-\partial_t )\tilde{g} $ with $\tilde{g}$ we obtain 
\begin{align*}
\int_{J_t} \int_{\mathcal{O}_R} S_k \tilde{g} \, dx \, ds&=\int_{J_t} \int_{\mathcal{O}_R} \tilde{g}(\mathcal{M}_\epsilon -\ell - \partial_t) \tilde{g} \, dx \,ds\\
&= - \int_{\mathcal{O}_R} \frac{\tilde{g}(t,x)^2}{2}\, dx + \int_{J_t} \int_{\mathcal{O}_R}\Big(-\ell \pm \frac{\text{div}(Z)}{2}\Big) \tilde{g}(s,x)^2 \, dx \, ds \\
&\qquad - \sum_{j=1}^r \int_{J_t} \int_{\mathcal{O}_R} (Z_j \tilde{g}(s,x))^2 \, dx \, ds,
\end{align*}
where the $\pm$ above is $+$ when $\mathcal{M}_\epsilon = \mathcal{L}_\epsilon$ and $-$ when $\mathcal{M}_\epsilon = \mathcal{L}^-_{\epsilon}$.  Picking $\ell > \tfrac{1}{2}\|\text{div}(Z)\|_{L^\infty(\mathcal{X})}$, we arrive at the bound
 \begin{align}
 &\sup_{t\in (t_1/2, 2t_2)} \int_{\mathcal{O}_R} \tilde{g}(t,x)^2 \, dx  + \int_{t_1/2}^{2 t_2} \int_{\mathcal{O}_R} \tilde{g}(s,x)^2 \, \, dx \, ds + \sum_{j=1}^r \int_{t_1/2}^{2t_2} \int_{\mathcal{O}_R} (Z_j \tilde{g}(s,x))^2 \, dx \, ds \nonumber \\
 &\qquad\leq C \| S_k \|_{L^2((t_1/2, 2t_2)\times \mathcal{O}_R)}^2  \lesssim_{R} \alpha^2 \beta^{2k}2^{4k}  \max\{ \|w_k\|_{L^2(I_k \times H_{R_k})}^2, 1 \}, \label{eq:tildegSkbound}
 \end{align}
 where we applied Lemma~\ref{lem:apriori2} in the final inequality above. Now fix a smooth cutoff function $\chi \in C^\infty_0((3t_1/4, 3 t_2/2) \times H_{R+1})$ that depends only on the Hamiltonian in space and is $\equiv 1$ in a neighborhood of $(t_1, t_2)\times H_R $. Let $Z_{0,\epsilon} = \epsilon Z - Z_0$ if $\mathcal{M}_\epsilon = \mathcal{L}_\epsilon$ and $Z_{0,\epsilon} = -\epsilon Z + Z_0$ if $\mathcal{M}_\epsilon = \mathcal{L}_\epsilon^-$. From \eqref{eq:tildegPDE}, for any $\phi \in C_0^\infty((t_1/2,2t_2)\times \mathcal{O}_R)$ we have
\begin{align*}
	\int_{t_1/2}^{2t_2} \int_{\mathcal{O}_R} \phi (\epsilon \partial_t + Z_{0,\epsilon})(\chi \tilde{g}) \, dx \, dt & = \epsilon \sum_{j=1}^r \int_{t_1/2}^{2t_2} \int_{\mathcal{O}_R} \phi \chi Z_j^2 \tilde{g} \, dx \, dt\\ 
	& \quad + \int_{t_1/2}^{2t_2} \int_{\mathcal{O}_R} \phi \tilde{g} (\epsilon \partial_t + Z_{0,\epsilon})\chi \, dx \, dt. 
\end{align*}
Integrating by parts and using \eqref{eq:tildegSkbound} then gives
\begin{align} \label{eq:tildegdual}
&\sup_{\substack{\phi \in C^\infty_0((t_1/2,2t_2) \times \mathcal{O}_R)\\\int_{t_1/2}^{2t_2} \vertiii{\phi}_{\mathcal{O}_R}^2 \, dt \leq 1}} \left|\int_{t_1/2}^{2t_2} \int_{\mathcal{O}_R} \phi (\epsilon \partial_t + Z_{0,\epsilon})(\chi \tilde{g}) \, dx \, dt\right|^2\\
\nonumber &\qquad \lesssim_{R, |t_2-t_1|} \alpha^2 \beta^{2k} 2^{4k}\max\{ \|w_k\|_{L^2(I_k \times H_{R_k})}^2, 1 \}.
\end{align}
Noting that \eqref{eq:tildegSkbound} holds with $\tilde{g}$ replaced by $\chi \tilde{g}$, it follows from Sobolev interpolation, Assumption~\ref{assump:V} (V5), \eqref{eq:tildegdual},  and Theorem~\ref{thm:tdhor} that for some $s > 0$ independent of $\epsilon \in (0,1)$ we have 
\begin{align*}
\| \chi \tilde{g}\|_{L^{2\beta}((t_1/2, 2t_2) \times \mathcal{O}_R) } &\lesssim \| \chi \tilde{g} \|_{L^\infty_t L^2_x} + \| \chi \tilde{g}\|_{L^2_t H_x^s}\lesssim_R 2^{k} \max\{ \|w_k\|_{L^2(I_k \times H_{R_k})}, 1 \},\end{align*}
where $\beta \in (1,2)$ is now chosen depending on the value of $s > 0$. Using the fact that $v_k\leq \chi \tilde{g}$ on $(t_1, t_2) \times H_R$, we obtain 
\begin{align}
\| w_k \|_{L^{2\beta}(I_{k+1} \times H_{R_{k+1}})} \leq C^k \max\{ \|w_k\|_{L^2(I_k \times H_{R_k})}, 1 \},
\end{align}
which implies
\begin{align}
\| v\|_{L^{2\beta^{k+1}}(I_{k+1} \times H_{R_{k+1}}) } \leq C^{k \beta^{-k}} \max\{ \|v \|_{L^{2\beta^k}(I_k\times H_{R_k})}, 1 \}.
\end{align}
Iterating the estimate above finishes the proof. 

\end{proof}

We now combine the apriori estimate (Proposition~\ref{prop:apriorigreen}) with Theorem~\ref{thm:moser2} to conclude Lemma~\ref{lem:upper1}.

\begin{proof}[Proof of Lemma~\ref{lem:upper1}]
Fix $t_0 \ge 1$ and $R > 0$. Let $I_0 = [t_0-1/2,t_0+1/2]$, $I_1 = [t_0-1/4,t_0+1/4]$, and $I_2 = [t_0 - 1/8, t_0+1/8]$. Let $f\in C_0^\infty(\mathcal{X}; [0, \infty))$ satisfy $f\geq 0$ and fix $\alpha > \tfrac{1}{2}\| \text{div}(Z)\|_{L^\infty(\mathcal{X})}$.   
Observe that for $t>0$ and $x\in\mathcal{X}$, we have
\begin{align*}
    (\mathcal{L}_\epsilon - \alpha -\partial_t)(e^{-\alpha t} \mathcal{Q}_t^\epsilon  f)=0.
\end{align*}
By Assumption~\ref{assump:V} and H\"{o}rmander's theorem~\cite{Hor_67}, $(t, x) \mapsto e^{-\alpha t} \mathcal{Q}_t^\epsilon f(x) \in C^\infty((0, \infty) \times \mathcal{X})$.  Now, by first applying Theorem~\ref{thm:moser2} and then applying Proposition~\ref{prop:apriorigreen}, it follows that 
\begin{align*}
    \|e^{-\alpha t}\mathcal{Q}_t^\epsilon f\|_{L^\infty(I_1\times H_{2R})}&\lesssim_{R}     \|e^{-\alpha t} \mathcal{Q}_t^\epsilon f\|_{L^2(I_0\times H_{4R})}\\
    &\lesssim_{R} \|f\|_{L^2(\mathcal{X},e^{-\eta H}dx)}.
\end{align*}
Using the definition of $\mathcal{Q}^\epsilon_t f$, we obtain
\begin{align}
\label{eqn:7.1b1}
    \sup_{(t,x)\in I_1 \times H_{2R}}\int e^{-\alpha t} q_t^\epsilon(x,y) f(y)dy\lesssim_R \|f\|_{L^2(\mathcal{X},e^{-\eta H}dx)},
\end{align}
and hence by duality we have 
\begin{align}
\label{eqn:L^2comp}
\sup_{x \in H_{2R}}\|e^{-\alpha \cdot} q_{\cdot}^\epsilon(x,\cdot)\|_{L^2(I_1\times H_{2R})} \lesssim_R 1.
\end{align}

Next, note that $ e^{-\alpha t} q_t^\epsilon(x,y)$, which is $C^\infty$ on $(0, \infty) \times \mathcal{X}$ for fixed $x$ by Assumption~\ref{assump:V}, satisfies
\begin{align*}
    (\mathcal{L}^*_{\epsilon,y}- \alpha -\partial_t)(e^{-\alpha t} q_t^\epsilon(x,y))=(\mathcal{L}^-_{\epsilon,y}- \alpha+\text{div}(Z)-\partial_t)(e^{-\alpha t} q_t^\epsilon(x,y))=0.
\end{align*}
Letting $\ell =  \| \text{div}(Z) \|_{L^\infty(\mathcal{X})}$, the above then implies  
\begin{align*}
   (\mathcal{L}^-_{\epsilon,y}- \alpha+ \ell -\partial_t)(e^{-\alpha t} q_t^\epsilon(x,y))\geq 0. 
\end{align*}
Therefore, first using Theorem~\ref{thm:moser2} again and then applying \eqref{eqn:L^2comp}, we conclude that
\begin{align*}
    \sup_{(t,x,y) \in I_2\times H_{R}\times H_{R}} (e^{-\alpha t}q_t^\epsilon(x,y))&\lesssim_{R}  \sup_{x\in H_R} \| e^{-\alpha \cdot}q_\cdot^\epsilon(x,\cdot)\|_{L^2(I_1 \times H_{2R})}\\
    &\lesssim_{R} 1.
\end{align*}
Since $t_0 \ge 1$ was arbitrary and the implicit constants above are all independent of $t_0$, we have shown that
\begin{align} \label{eq:7.1extended}
\sup_{(t,x,y)\in [7/8, \infty) \times H_R \times H_R} (e^{-\alpha t} q_t^\epsilon(x,y) ) \lesssim_R 1. 
\end{align}
This finishes the proof. 
\end{proof}

Now, given Lemma~\ref{lem:upper1} as well as Theorem~\ref{thm:tdhor}, we can conclude Lemma~\ref{lem:upper2}.
\begin{proof}[Proof of Lemma~\ref{lem:upper2}]
First observe that on $(0, \infty) \times \mathcal{X}^2$ we have in the classical sense
\begin{align}
\label{eqn:q_tequ}
( L_{\epsilon, x}+ L_{\epsilon, y}^* - 2\epsilon \partial_t) q_t^\epsilon(x,y) = 0.
\end{align} 
Note that we can write  
\begin{align*}
( L_{\epsilon, x}+ L_{\epsilon, y}^* -2 \epsilon \partial_t) &=   -\epsilon Z_x + Z_{0, x} + \epsilon Z_y - Z_{0, y}+ \text{div}(Z_y) +\sum_{j=1}^r Z_{j,x}^2 + Z_{j, y}^2  -2 \epsilon \partial_t \\
&=: Z_{0, x,y, \epsilon} + \text{div}(Z_y)+ \sum_{j=1}^r Z_{j,x}^2 + Z_{j, y}^2  - 2 \epsilon \partial_t.  
\end{align*}
Also, it can be checked that the list $\{ Z_{0, x,y, \epsilon}; Z_{1, x}; Z_{1,y}; \ldots; Z_{r, x}, Z_{r, y} \}$ satisfies the uniform parabolic H\"{o}rmander condition on $H_R \times H_R$ for any $R>0$ with respect to $\epsilon \in (0,1)$.  This is because derivatives in $x$ and $y$ commute.  

Let $\chi_R$ be as in~\eqref{eq:cutoff} and for $a \ge 1$ let $I_a = [a-1/16,a+1/16]$. To complete the proof, it is sufficient to show that there is some $s > 0$ independent of $\epsilon \in (0,1)$ such that for every $a \ge 1$ we have 
\begin{equation}\label{eq:lemma7.2goal}
	\int_{I_a} \|\chi_R q_t^\epsilon\|_{H^s}^2 dt \lesssim_{a,R} 1.
\end{equation}
Define $\mathring{I}_a = [a-1/8,a+1/8]$ and let $\psi \in C^\infty_0(\RR; [0,1])$ satisfy $\psi\equiv 1$ on $I_a$ and $\psi \equiv 0$ on $\RR \setminus \mathring{I}_a$. Applying Theorem~\ref{thm:tdhor} to $u=\psi \chi_R q_t^\epsilon$, we find that there exists an $s>0$ independent of $\epsilon \in (0,1)$ such that  
\begin{align*}
&\int_{\mathring{I}_a} \| u(t) \|^2_{H^s(\RR^d \times \RR^d)} \, dt\\
& \lesssim_R  \int_{\mathring{I}_a} \vertiii{u(t)}_{H_{2R}\times H_{2R}}^2 \, d t + \bigg(\sup_{\substack{\phi \in C^\infty_0(\mathring{I}_a \times H_{2R}^2)\\\int_{\mathring{I}_a} \vertiii{\phi}^2 \, dt \leq 1}} \int_{\mathring{I}_a \times \RR^d \times \RR^d} \phi (-2\epsilon \partial_t+ Z_{0, x,y, \epsilon})u \, dt \, dx\bigg)^2,
\end{align*}
where 
\begin{align*}
\vertiii{u(t)}^2_{H_{2R} \times H_{2R}}= \| u(t) \|^2_{L^2(H_{2R}^2)}+ \sum_{j=1}^r    \|Z_{j, x} u(t) \|^2_{L^2( H_{2R}^2) } +\sum_{j=1}^r \|Z_{j, y} u(t) \|^2_{L^2(H_{2R}^2)}. 
\end{align*}
Noting that $I_a \subseteq [7/8, \infty)$, it then follows from \eqref{eqn:q_tequ} and \eqref{eq:7.1extended} that
\begin{align*}
&\int_{\mathring{I}_a}  \|u\|^2_{H^s(\RR^d \times \RR^d)} \, dt\\
& \lesssim_{R} \| q_t^\epsilon \|_{L^2(\mathring{I}_a\times H_{2R}^2)}^2 + \sum_{j=1}^r  \|Z_{j, x} u \|^2_{L^2(\mathring{I}_a \times H_{2R}^2) }+ \sum_{j=1}^r\|Z_{j, y} u \|^2_{L^2(\mathring{I}_a \times H_{2R}^2)}  \\
& \lesssim_{a,R} 1+ \sum_{j=1}^r  \|Z_{j, x} u \|^2_{L^2(\mathring{I}_a \times H_{2R}^2) } +\sum_{j=1}^r \|Z_{j, y} u\|^2_{L^2(\mathring{I}_a \times H_{2R}^2)} .
\end{align*}
To estimate the remaining  terms on the right hand side above, multiplying~\eqref{eqn:q_tequ} by $\psi^2\chi_R^2 q_t^\epsilon$, integrating over $\mathring{I}_a \times H_{2R}^2$, and integrating by parts produces the inequality
\begin{align*}
\sum_{j=1}^r   \| Z_{j, x}u \|^2_{L^2(\mathring{I}_a \times H_{2R}^2) } + \sum_{j=1}^r \|Z_{j, y} u \|^2_{L^2(\mathring{I}_a \times H_{2R}^2)} \lesssim_R \| q_t^\epsilon \|^2_{L^2(\mathring{I}_a \times H_{2R}^2)} \lesssim_{a,R} 1. 
\end{align*}
We have thus shown 
$$\int_{I_a} \|\chi_R q_t^\epsilon\|^2_{H^s(\RR^d \times \RR^d)} dt \le \int_{\mathring{I}_a} \|u\|^2_{H^s(\RR^d \times \RR^d)}dt \lesssim_{a,R} 1,$$
which gives \eqref{eq:lemma7.2goal} and hence completes the proof.
\end{proof}

\section{Uniform lower bounds}
\label{sec:lowerb}
In light of the mantra `quantitative upper bounds become quantitative lower bounds', we next turn to the proof of Theorem~\ref{thm:lowerbtaint}.   As explained briefly in Section~\ref{sec:outline}, the idea behind this mantra is relatively simple.  To add more detail to that explanation, suppose for the moment that $u$ is a strictly positive and smooth function satisfying 
\begin{align}
(\mathcal{M}_\epsilon  - \partial_t) u = 0 \,\,\, \text{ on } \, \, \, (0, \infty) \times \mathcal{X}
\end{align}
in the classical sense, where $\mathcal{M}_\epsilon \in \{ \mathcal{L}_\epsilon, \mathcal{L}_\epsilon^{-}\}$ and $\mathcal{L}_{\epsilon}^{-}$ is as in~\eqref{eqn:gentr}.  Then, a standard computation shows that $u^{-\delta}$ for $\delta >0$ small satisfies a similar relation, namely 
\begin{align}
(\mathcal{M}_\epsilon - \partial_t) (u^{-\delta}) \geq 0 \,\,\, \text{ on } \,\,\, (0, \infty)\times \mathcal{X}. 
\end{align}
Applying the quantitative Moser theorem (Theorem~\ref{thm:moser2}), we can control the $L^\infty$-norm of $u^{-\delta}$ uniformly by its $L^2$-norm on a bounded set.  However, what remains to control is the $L^2$-norm of $u^{-\delta}$ uniformly, and this is the central difficulty that arises from this idea.  Note also that we assumed $u$ is strictly positive, which is another detail that would need to be addressed, but one can nevertheless follow similar argumentation to that in the proof of Theorem~\ref{thm:moser2} to get around this issue.

For us, applying this approach directly on the time-changed density $$u=q_t^\epsilon(x,y)=p_{t/\epsilon}^\epsilon(x,y)$$ does not seem to yield the desired result due to a lack of quantitative information about its time regularity when $\epsilon \approx 0$.  However, we do indeed have the needed regularity when we average out the time component.  Then, we can obtain quantitative lower bounds on the averaged density from quantitative upper bounds.     

\begin{remark}
Below, instead of adopting the approach of Moser~\cite{Mos_61} outlined above via $u^{-\delta}$, we follow the approach of De Giorgi~\cite{De_57}, using functions like $w=(1-Cu)_+$ and attempting to choose the constant $C>0$ large enough so that the $L^2$-norm of $w$ is small.  Note that such an upper bound on $w$ would indeed produce the correct positive lower bound on $u$ using the quantitative Moser iteration (Theorem~\ref{thm:moser2}). 
\end{remark}

\subsection{The time-averaged density}
To produce a uniform lower bound on compact sets for the time-averaged density, it seems critical in the arguments to average in a relatively precise manner.  That is, fix $t_0 \geq 1$ independent of $\epsilon \in (0,1)$ and using Assumption~\ref{assump:V} (V3) set 
\begin{align}
\alpha :=\inf_{x\in \mathcal{X}} \text{div}(Z)(x)>\tfrac{1}{2}\| \text{div}(Z) \|_{L^\infty(\mathcal{X})}.  
\end{align} 
Introduce the \emph{time-averaged density} $h^\epsilon$ given by
\begin{align}
\label{def:tadens}
    h^\epsilon=h^\epsilon(x,y)= \alpha \int_{t_0}^\infty  e^{-\alpha (t-t_0) } q_t^\epsilon(x,y) dt.
\end{align}
Observe that under Assumptions~\ref{assump:H}-\ref{assump:BR}, for any $R>0$, $(x,y) \in H_R \times H_R$, and $\alpha '$ satisfying $\frac{1}{2}\| \text{div}(Z) \|_{L^\infty(\mathcal{X})}<\alpha' < \alpha$:
\begin{align}
\label{eqn:alpha23}
 h^\epsilon(x, y) =\alpha  \int_{t_0}^{\infty}  e^{-\alpha( t-t_0)} q_t^\epsilon(x,y) \, dt&\lesssim_{t_0}  \| e^{-\alpha' t} q_t^\epsilon(x,y) \|_{L^\infty([t_0, \infty) \times H_R^2)} \int_{t_0}^\infty e^{-(\alpha-\alpha') t} \, dt \\
 \nonumber&\lesssim_{R, t_0} 1,
 \end{align}
 where we used Lemma~\ref{lem:upper1} on the last line above.  In particular, for any $R>0$, we have the uniform upper bound
 \begin{align}
 \label{eqn:unifh}
 \|h^\epsilon \|_{L^\infty( H_R^2)} \lesssim_{R, t_0} 1. 
 \end{align}

   The goal of this section is to turn the upper bound~\eqref{eqn:unifh} into a lower bound.  That is, we will show:
\begin{theorem}
\label{thm:lowerbta}
Suppose that Assumptions~\ref{assump:H}-\ref{assump:BR} are satisfied.  Then for every $R>0$, there exists $d_R>0$ such that 
\begin{align}
\inf_{(x,y) \in H_R\times H_R} h^\epsilon (x,y) \geq d_R
\end{align}
for all $\epsilon \in (0,1)$. 
\end{theorem}

As an immediate corollary of Theorem~\ref{thm:lowerbta}, we obtain Theorem~\ref{thm:lowerbtaint}.

\begin{proof}[Proof of Theorem~\ref{thm:lowerbtaint}]
Fix $R>0$, $t_0=1$ and $\alpha = \inf_{x\in \mathcal{X}} \text{div}(Z)(x)$.  Then, Theorem~\ref{thm:lowerbta} implies that there exists $d_R>0$ such that for all $(x,y) \in H_R\times H_R$, $T> 1$ and $\epsilon \in (0,1)$, we have 
\begin{align*}
d_R \leq h^\epsilon(x,y) &= \alpha e^{\alpha} \int_{1}^T e^{-\alpha t}  q_t^\epsilon(x,y) \, dt + \alpha e^{\alpha } \int_T^\infty  e^{-\alpha t} q_t^\epsilon(x,y) \, dt \\
& \leq \alpha e^{\alpha } \int_{1}^T e^{-\alpha t}  q_t^\epsilon(x,y) \, dt \\
&\qquad + \alpha e^{\alpha }\| e^{-\alpha' t} q_t^\epsilon(x,y) \|_{L^\infty([1, \infty) \times H_R^2) } \int_{T}^\infty e^{-(\alpha-\alpha') t} \,dt ,
\end{align*} 
where $\alpha' \in (\tfrac{1}{2}\| \text{div}(Z)\|_{L^\infty(\mathcal{X})}, \alpha)$.  
By Lemma~\ref{lem:upper1}, we can pick $T>1$ independent of $\epsilon \in (0,1)$ such that 
\begin{align*}
\alpha e^{\alpha} \| e^{-\alpha' t} q_t^\epsilon(x,y) \|_{L^\infty([1, \infty) \times H_R^2) } \int_{T}^\infty e^{-(\alpha-\alpha') t} \,dt \leq \frac{d_R}{2}, 
\end{align*}
and thus
\begin{align*}
\frac{d_R}{2}\leq \alpha  \int_{1}^T  q_t^\epsilon(x,y) \, dt.  
\end{align*}
This finishes the proof.  

\end{proof}

The proof of Theorem~\ref{thm:lowerbta} will be carried out in the rest of this section after showing a series of lemmata.

\begin{remark}
One can interpret the time-averaged density $h^\epsilon(x,y)$ stochastically. Indeed, $h^\epsilon(x,y)$ is the one-step transition density of a Markov kernel $\mathcal{Q}^{\epsilon, \text{av}}$ defined by 
\begin{align}
\mathcal{Q}^{\epsilon, \text{av}}f(x) = \EE[\mathcal{Q}_\sigma^\epsilon f(x)], \,\,\, f\in \mathcal{B}(\mathcal{X}),
\end{align}
where we recall that $\mathcal{Q}_t^{\epsilon} = \mathcal{P}_{t/\epsilon}^\epsilon$ is the time-changed kernel and $\sigma$ is a random variable distributed on $[0, \infty)$ according to the shifted exponential distribution 
\begin{align}
\PP[\sigma \in (a,b)] =\alpha  \int_a^b e^{-\alpha( t-t_0)}\, dt, \,\, (a,b) \subset (0, \infty).
\end{align}
Ultimately, if $\sigma_1, \sigma_2, \ldots $ are i.i.d. random variables on $(\Omega, \mathcal{F}, \PP)$ with $\sigma_i \sim \sigma$, the corresponding $n$-step transition for $\mathcal{Q}^{\epsilon, \text{av}}$ is given by 
\begin{align}
\mathcal{Q}^{\epsilon, \text{av}}_n f(x) = \EE[\mathcal{Q}_{\sigma_1+\cdots + \sigma_n}^\epsilon f(x)], \, \, \, f\in \mathcal{B}(\mathcal{X}).
\end{align}
Thus one can think of this kernel as sampling the original dynamics~\eqref{eqn:SDEmain} at the time $\sum_{i=1}^n\sigma_i$, and then averaging as above.  
\end{remark}

The first lemma shows that upper bounds become lower bounds in measure, with some control over the original dynamics~\eqref{eqn:SDEmain} through the Lyapunov structure guaranteed by Assumption~\ref{assump:V}.

\begin{lemma}
\label{lem:posmeas}
Suppose that Assumptions~\ref{assump:H}-\ref{assump:BR} are satisfied.  Fix $R>0$. There exist $S>R$, $c_{R,S}>0$ and $\delta_{R,S} >0$ independent of $\epsilon \in (0,1)$ such that  
\begin{align}
\inf_{x\in H_R}| \{ y \in H_{S} \, : \, h^\epsilon (x,y) \geq \delta_{R,S} \}| \geq c_{R,S}
\end{align}  
for all $\epsilon \in (0,1)$. 
\end{lemma}

\begin{proof}
Fix $x\in H_R$.  Below, we write $h^\epsilon(x,y)$ more succinctly as $h(x,y)$.  Recall from Section~\ref{sec:mainresults} that Assumption~\ref{assump:V} implies that
$$
	\mathcal{L}_\epsilon V \le d_*
$$
for $V = H+1$. The above estimate translates to the following estimate on the time-changed semigroup $\mathcal{Q}_t^\epsilon$:
\begin{align}
	\mathcal{Q}_t^\epsilon V(x)  \leq V(x) + d_* t , \,\,\, t \geq 0, \,\,\epsilon \in (0,1).
\end{align}
Thus,
\begin{align} \label{eq:hV}
	\int_\mathcal{X}h(x,y) V (y)dy =\alpha \int_{t_0}^\infty  e^{-\alpha (t-t_0)} \mathcal{Q}_t^\epsilon V (x) dt  \leq C_R <\infty
\end{align}
for some constant $C_{R}>0$ independent of $\epsilon \in (0,1)$. Since $\int_{\mathcal{X}} h(x,y) dy = 1$ by construction and $V=H+1\geq 1+S$ on $H_S$, it follows from \eqref{eq:hV} that for any $S > R$ we have
\begin{align*}
	1 = \int_{\mathcal{X}} h(x,y) dy &= \int_{H_{S}} h(x,y) dy  + \int_{H_{S}^c} h(x,y) dy \\ 
	& \le \int_{H_{S}} h(x,y) dy+ \frac{1}{1+S} \int_{H_{S}^c}  h(x,y) V(y) dy \\ 
	& \le \int_{H_{S}} h(x,y) dy + \frac{C_R}{1+S}.
\end{align*}
Choosing $S = (2C_R)^{-1}$ then yields 
\begin{align*} 
	 \frac{1}{2} \le \int_{H_{S}} h(x,y)dy & = \int_{H_{S}\cap \{y \,: \, h(x,y)\geq \delta\}}h(x,y)dy+\int_{H_{S}\cap \{y \,:\,h(x,y)< \delta\}}h(x,y)dy \\ 
	 & \le \|h\|_{L^\infty(H_R \times H_{S})} |H_{S} \cap\{y \,: \, h(x,y)\geq \delta\}| + \delta |H_{S}|. 
\end{align*}
Thus for some $\delta_{R,S}>0$ small enough we have 
\begin{align*}
    \frac{1}{4}\leq \|h\|_{L^\infty(H_R\times H_{S})}|H_{S}\cap \{y\,: \,h(x,y)\geq \delta_{R,S}\}|.
\end{align*}
By way of~\eqref{eqn:unifh}, this in turn implies the claimed result. 
\end{proof}

Next, we construct a version of De Giorgi's iteration functions~\cite{De_57} as in~\cite{BL_21}, which are effectively smoothed out versions of
\begin{align}
(1-C h^\epsilon)_+
\end{align}
where $C>0$ is a `big' constant. Such functions will ultimately satisfy a sufficiently small uniform upper bound, which then translates to a uniform lower bound on $h^\epsilon$.

To this end, for $\theta \in(0,\frac{1}{2})$ and $k\in \mathbf{N}$, define
\begin{align}
\label{eqn:wk}
    \widetilde{w}_k(x,y)=1-\left(\frac{4}{\theta}\right)^k \frac{h^\epsilon(x,y)}{\delta_{R,S}} \qquad \text{ and } \qquad w_k(x,y) = \phi_\epsilon(\widetilde{w}_k(x,y)),
\end{align}
where $\delta_{R,S} >0$ is as in the conclusion of Lemma~\ref{lem:posmeas} and we assume that the cutoff function $\phi_\epsilon : \RR\rightarrow \RR$ satisfies the following conditions:
\begin{itemize}
\item $\phi_\epsilon \in C^\infty(\RR)$ and $\| \phi''_\epsilon\|_{L^\infty(\RR)} \lesssim \epsilon^{-1/4}$;
\item $\phi_\epsilon'' \geq 0$;
\item $\phi_\epsilon(x) = x$ for $x\geq \epsilon^{1/4}$;
\item $\phi_\epsilon(x)=0$ if $x\leq -\epsilon^{1/4}$;
\item $\phi_\epsilon$ is nondecreasing, $\| \phi'_\epsilon \|_{L^\infty(\RR)}\lesssim 1$ and $\phi_\epsilon >0$ for $x>- \epsilon^{1/4}$. 
\end{itemize}

After a short exercise using the properties of $\phi_\epsilon$, it follows that the conditions below are met for $k,\ell \in \N$, $\epsilon \in (0,1/16)$, $c_{R,S}>0$, $\delta_{R,S} >0$, $x\in H_R$ and $S\gg R$ as in the conclusion of Lemma~\ref{lem:posmeas}:
\begin{itemize}
\item $| \{ y\in H_S: w_k(x,y) =0 \}| \geq c_{R,S}>0$;
\item $\{y: w_{k+1}(x,y) >0 \} \subset \{ y: w_{k}(x, y) \geq 1-\theta\}$;
\item $\{y: 0<w_k(x, y) <1-\theta\} \cap \{y: 0< w_\ell(x, y) < 1-\theta\} = \emptyset$ if $k\neq \ell$. 
\end{itemize}

Moreover, we also have:
\begin{lemma}
\label{lem:L-}
Suppose that Assumptions~\ref{assump:H}-\ref{assump:BR} are satisfied. Let $\theta_* \in (0,1/2)$, $k_*\in \mathbf{N}$, and $S>R$ be as in the conclusion of Lemma~\ref{lem:posmeas}.  Then, there exists $\epsilon_*(k_*, \theta_*, S)>0$ such that, whenever $\epsilon \in (0, \epsilon_*)$ and $k=1,2,\ldots, k_*$, $w_k$ as in~\eqref{eqn:wk} satisfies 
\begin{align}
\label{ineq_on_L-}
    0 \leq \mathcal{L}^-_{\epsilon,y}w_k \lesssim_{S} \epsilon^{-1/2} \Big(1+ \sum_{j=1}^r(Z_{j,y} h^\epsilon (x,y))^2 \Big)
\end{align}
on $H_{2S}\times H_{2S}$. 
\end{lemma}  

\begin{proof}
First observe that, for $x\in \mathcal{X}$ fixed, $ h^\epsilon(x, \cdot)$ satisfies, in the sense of distributions on $\mathcal{X}$,  
\begin{align} \label{eq:Lstarh}
 \mathcal{L}_{\epsilon, y}^*  h^\epsilon (x,y)=  - \alpha q_{t_0}^\epsilon(x,y)  +  \alpha h^\epsilon(x,y)
  &\leq \alpha  h^\epsilon (x,y)
\end{align}
on $\mathcal{X}$.  Since $\mathcal{L}_{\epsilon, y}^*$ is hypoelliptic by Assumption~\ref{assump:V}, $h^\epsilon(x,\cdot) \in C^\infty(\mathcal{X})$ and the relations above are satisfied in the classical sense on $\mathcal{X}$ by~\cite{Hor_67}.  After noting that $\mathcal{L}_{\epsilon,y}^*= \mathcal{L}_{\epsilon, y}^{-} + \text{div}(Z)$, by Assumption~\ref{assump:V} (V3) and the choice of $\alpha = \inf_{x\in \mathcal{X}} \text{div}(Z)(x)$, we find that for $(x,y) \in H_{2S} \times H_{2S}$ we have
\begin{align*}
\mathcal{L}^-_{\epsilon,y} w_k (x,y)&=\phi'_\epsilon(\tilde{w}_k)\frac{\left(\frac{4}{\theta}\right)^k}{\delta_{R,S}}\left( (\alpha - \mathcal{L}_{\epsilon, y}^*)h^\epsilon(x,y)+ (\text{div}(Z)- \alpha) h^\epsilon(x,y)\right)\\
   &\,\,+\phi''_\epsilon(\tilde{w}_k)\frac{\left(\frac{4}{\theta}\right)^{2k}}{\delta_{R,S}^2}\sum_{j=1}^r(Z_{j,y} h^\epsilon (x,y))^2 \ge 0.
\end{align*}
On the other hand, starting from the equality above, we also have
\begin{align*}
   \mathcal{L}^-_{\epsilon,y} w_k(x,y) &= \phi'_\epsilon(\tilde{w}_k)\frac{\left(\frac{4}{\theta}\right)^k}{\delta_{R,S}}\left( \alpha q_{t_0}^\epsilon(x,y)+ (\text{div}(Z)- \alpha) h^\epsilon(x,y)\right) \\ 
    & \qquad + \phi''_\epsilon(\tilde{w}_k)\frac{\left(\frac{4}{\theta}\right)^{2k}}{\delta_{R,S}^2}\sum_{j=1}^r(Z_{j,y} h^\epsilon (x,y))^2\\
        & \lesssim_S \phi'_\epsilon(\tilde{w}_k) \left(\frac{4}{\theta}\right)^k +\phi''_\epsilon(\tilde{w}_k)\frac{\left(\frac{4}{\theta}\right)^{2k}}{\delta_{R,S}^2}\sum_{j=1}^r(Z_{j,y} h^\epsilon(x,y))^2,
            \end{align*}
where we used \eqref{eqn:unifh} and Lemma~\ref{lem:upper1} to bound the first term in the last inequality. Choosing $\epsilon_*^{1/4}\leq (\theta_*/4)^{2k_*}\delta_{R,S}^2$ results in the claimed bound.  
  \end{proof}

The next result eliminates the possibility that the class of functions $w$ satisfying the inequality~\eqref{ineq_on_L-} includes step functions in the small-$\epsilon$ limit.

\begin{lemma}[Intermediate Value Lemma]
\label{IVL}
Suppose that Assumptions~\ref{assump:H}-\ref{assump:BR} are satisfied.  Let $S\geq R>0$ and $\lambda_1, \lambda_2 >0$ be independent of $\epsilon \in (0,1)$. Then, there exist $\epsilon_0>0, \beta >0$ and $\theta \in (0,1/2)$ independent of $\epsilon \in (0,1)$ such that if $\epsilon \leq \epsilon_0$ and $w\in C^\infty(H_{2S}\times H_{2S})$ with $0\leq w \leq 1$ satisfies
\begin{align*}
   0 \leq \mathcal{L}^-_{\epsilon,y} w \le C_S\epsilon^{-1/2}\left(1 + \textstyle{ \sum_{j=1}^r}(Z_{j,y}h^\epsilon(x,y))^2\right)
\end{align*}
on $H_{2S}\times H_{2S}$, then the inequalities 
\begin{align}
\label{eqn:lemmcond}
\inf_{x\in H_R} | \{ y \in H_S\, : \, w(x,y)= 0\} | \geq \lambda_1 \qquad \text{ and } \qquad \inf_{x\in H_R} | \{y\in H_S \, : \,  w(x,y) \geq 1-\theta\} | \geq \lambda_2
\end{align}
together imply
\begin{align}
\inf_{x\in H_R} | \{ y \in H_S\, : \, 0< w<1-\theta\}  | \geq \beta. 
\end{align}
\end{lemma}
\begin{proof}
Suppose to the contrary that the lemma is not true. Then, there exist sequences $ \{\epsilon_n\} \subset (0,1)$ with $\epsilon_n\rightarrow 0$, $\{ x_n \} \subset H_R$ and $\{w_n\} \subseteq C^\infty(H_S^2; [0,1])$ satisfying the following conditions:
\begin{itemize}
\item[(i)] $| \{y \in H_S \, : \, w_n(x_n, y) =0 \} | \geq \lambda_1$ and $| \{y\in H_S: w_n(x_n, y) \geq 1-1/n \} | \geq \lambda_2$;
\item[(ii)] $0 \leq \epsilon_n \mathcal{L}^-_{\epsilon_n} w_n(x_n,y) \leq C_S \epsilon_n^{1/2}+ C_S \epsilon_n^{1/2}\sum_{j=1}^r(Z_{j,y} h^{\epsilon_n}( x_n,y))^2$ on $ H_{2S}$;
\item[(iii)] $| \{y \in H_S  \, : \, 0 < w_n (x_n,y)< 1-1/n \}  | \leq 1/n$.
\end{itemize}
By the Banach-Alaoglu theorem, there exists $w\in L^\infty(H_{2S})$ and a subsequence (not relabeled) such that $w(x_n,\cdot) \rightharpoonup^* w$ in $L^\infty(H_{2S})$.  

Let $\chi \in C^\infty_0( H_{2S})$ depend only on the Hamiltonian $H$ and satisfy $\chi \equiv 1$ in a neighborhood of $H_S$. It follows from the lower bound in (ii) and computations similar to those in the proof of Proposition~\ref{prop:apriorigreen} that
\begin{equation} \label{eq:Zjwn}
 \sum_{j=1}^r \| Z_{j,y}(\chi w_n) \|_{L^2(H_{2S})}^2 \lesssim_S \int_{H_{ 2S}} |w_n|^2 \, dy  \lesssim_{S} 1,
\end{equation}
where the constant in the above is independent of $n$. Similarly, using that $h$ satisfies \eqref{eq:Lstarh}, standard estimates show that
\begin{equation} \label{eq:Zjh}
	 \sum_{j=1}^r \|Z_{j,y} h^\epsilon(x,\cdot)\|_{L^2(H_{2S})}^2 \lesssim_S \|q_{t_0}(x,\cdot)\|_{L^2(H_{4S})}^2 + \|h^\epsilon(x,\cdot)\|_{L^2(H_{4S})}^2 \lesssim_S 1 
\end{equation}
for any $x \in H_R$, where we used \eqref{eqn:unifh} in the last inequality. If $\psi\in C_0^\infty(H_{2S})$ and $Z_{0, \epsilon} = \epsilon Z-Z_0 $, then (ii), \eqref{eq:Zjwn}, and \eqref{eq:Zjh} together imply
\begin{align*}
\bigg|\int Z_{0, \epsilon_n} (\chi w_n) \psi \, dy \bigg| &\leq \bigg| \int (Z_{0, \epsilon_n} \chi)w_n \psi \, dy\bigg|+ \bigg|\int \chi \psi Z_{0, \epsilon_n}w_n \, dy \bigg| \\
&\leq \bigg| \int (Z_{0, \epsilon_n} \chi)w_n \psi \, dy\bigg|+ \bigg|\int \chi \psi\epsilon_n\mathcal{L}_{\epsilon_n}^- w_n \, dy \bigg| + \bigg|\int  \chi \psi\epsilon_n\textstyle{\sum_{j=1}^r} Z_j^2 w_n \, dy   \bigg| \\
& \lesssim_S \| \psi\|_{L^2(H_{2S}) }+\epsilon_n^{1/2}  \| \psi\|_{L^\infty(H_{2S})}\sum_{j=1}^r\| Z_{j,y}h^\epsilon \|^2_{L^2(H_{2S})}\\
&\qquad  + \epsilon_n\sum_{j=1}^r( \|Z_j (\chi w_n) \|_{L^2(H_{2S})} +1)(\| Z_j \psi\|_{L^2(H_{2S})} + \|\psi\|_{L^\infty(H_{2S})})\\
& \lesssim \| \psi\|_{L^2(H_{2S})} + \epsilon_n \sum_j \| Z_j \psi \|_{L^2(H_{2S})} + \epsilon_n^{1/2}\| \psi\|_{L^\infty(H_{2S})}.  
  \end{align*}
 From the previous bound, \eqref{eq:Zjwn} and $\|\chi w_n\|_{L^\infty(H_{2S})} \le 1$, we may apply Theorem~\ref{thm:Hssmooth3} with $\delta = 1$ to obtain $s\in (0,1)$ independent of $n$ such that
  \begin{align*}
  	\sup_{n\geq 1} \| \chi w_n \|_{H^s(\RR^d)} \lesssim 1. 
  \end{align*}
Thus, we can upgrade the convergence (by passing to another subsequence) to $L^p(H_{2S})$ for some $p>2$.   In particular, $w_n \rightarrow w$ in measure.  It thus follows that 
\begin{itemize}
\item $0 \leq w \leq 1$;
\item $| \{ w=0 \} \cap H_{ S} | \geq \lambda_1$;
\item $|\{ w=1\} \cap H_{ S} |\geq \lambda_2$;
\item $| \{ 0<w<1 \} \cap H_{ S} |=0.$
\end{itemize} 
In other words, $w$ is a non-constant characteristic function on $H_S$.

\emph{Claim}. For every $1\le j \le r$,  $Z_j w = 0$ in $L^2(H_S)$.  

\emph{Proof of Claim}.  Fix some $1 \le j \le r$. From \eqref{eq:Zjwn} and the Banach-Alaoglu theorem, $Z_j w \in L^2(H_{S})$. It is sufficient to show that each point in $H_S$ has a neighborhood on which $Z_j w = 0$ almost everywhere. Fix $y_0 \in H_S$, and for $\delta > 0$ and $n \in \N$ let $C^n_\delta \subseteq \RR^n$ denote the open cube of side length $2\delta$ centered at the origin. By the straightening theorem for vector fields, there exist a neighborhood $V$ of $y_0$ and diffeomorphism $\Psi: V \to C^d_\delta$ for some $\delta > 0$ such that 
$$ D \Psi(y) Z_j(y) = e_1$$
for every $y \in V$. Let $\widetilde{w} = w \circ \Psi^{-1}$ and observe that $\partial_{y_1}\widetilde{w} =(Z_j w) \circ \Psi^{-1}$. As $|D \Psi^{-1}|$ is bounded, it follows that $\partial_{y_1}\widetilde{w} \in L^2(C^d_\delta)$ and hence by Fubini's theorem we have that $y_1 \mapsto \partial_{y_1} \widetilde{w}(y_1,y_2,\ldots, y_d) \in L^2(-\delta, \delta)$ for a.e. $(y_2,\ldots, y_d) \in C_\delta^{d-1}$. Since $H^1$ functions in one dimension cannot have jump discontinuities and $y_1 \mapsto \widetilde{w}(y_1,y_2,\ldots, y_d)$ is a characteristic function for fixed $(y_2,\ldots, y_d) \in C_\delta^{d-1}$, we conclude that $\partial_{y_1} \widetilde{w}(y_1,y_2,\ldots, y_d) = 0$ in $L^2(-\delta,\delta)$ for a.e. $(y_2,\ldots, y_d) \in C_\delta^{d-1}$. Therefore, using Fubini's theorem again, we obtain $\partial_{y_1} \widetilde{w} = 0$ in $L^2(C_\delta^d)$. Thus, $Z_j w = 0$ in $L^2(V)$ and the claim follows.

It is straightforward to see that $Z_{0,0}w = 0$ in the sense of distributions, and hence from the claim we have that $w$ is a distributional solution to the PDE
\begin{equation} \label{eq:wPDE}
Z_{0,0}w + \sum_{j=1}^r Z_j^2 w = 0
\end{equation}
on $H_S$. It can be shown that $\{Z_{0,\epsilon}, Z_1,\ldots, Z_r\}$ satisfying the uniform H\"{o}rmander condition with respect to $\epsilon \in (0,1)$ implies that $Z_{0,0} + \sum_{j=1}^r Z_j^2$ is hypoelliptic. Thus, by Assumption~\ref{assump:V} (V5) and \eqref{eq:wPDE} we have that $w$ is smooth on $H_S$. This contradicts the fact $w$ is a non-constant characteristic function on $H_S$.
 \end{proof}
 
 Given the previous result, we can now state and prove the main upper bound on the $w_k$'s that in turn produces a uniform lower bound on $h^\epsilon$.

 \begin{lemma}
 \label{lem:smallup}
Suppose that Assumptions~\ref{assump:H}-\ref{assump:BR} hold.  Let $S\gg R$ be as in the statement of Lemma~\ref{lem:posmeas} and fix $\kappa >0$.  Then, there exist $\theta_* \in (0,1/2)$, $\epsilon_* \in (0,1/16)$ and $K\in \mathbf{N}$, all only depending on $\kappa$ and $S$, such that whenever $\epsilon \in (0, \epsilon_*)$ there exists $k_* \in \mathbf{N}$ with $k_* \leq K$ for which 
\begin{align}
\sup_{x\in H_R} \int_{H_{S}} |w_{k_*}(x,y)|^2 \,  dy \leq \kappa. 
\end{align}
\end{lemma}

\begin{proof}
Pick $\epsilon_0 >0$, $\theta_* \in (0,1/2)$ and $\beta >0$ such that the conclusions of Lemma~\ref{IVL} are satisfied with $\lambda_1 = c_{R,S}$ and $\lambda_2=\kappa$.  Let $K$ be the smallest natural number exceeding $1+2|H_{S}|/\beta$.  By  Lemma~\ref{lem:L-}, pick $\bar{\epsilon}(K, \theta_*, R)<1/16$ so that~\eqref{ineq_on_L-} is met when $\epsilon \in (0, \bar{\epsilon})$ and $k\leq K$.  Let $\epsilon_* = \epsilon_0 \wedge \bar{\epsilon}$.  We have left to show that for every $\epsilon \in (0, \epsilon_*)$ there exists $k_* \leq K$ so that 
\begin{align*}
\sup_{x\in H_R} \int_{H_{S}} | w_{k_*}(x,y)|^2 \leq \kappa. 
\end{align*}
If this does not hold, then there exists $\epsilon' \in (0, \epsilon_*)$ and $x' \in H_R$ such that 
\begin{align*}
\int_{H_{S}} |w_{k}^{\epsilon'}(x', y)|^2 \, dy >\kappa
\end{align*} 
for all $k\leq K$.  For simplicity, we adopt the minor abuse of notation and set $w_{k}(y)= w_{k}^{\epsilon'}(x', y)$.  Recalling that $0\leq w_{k} \leq 1$, for $k\leq K-1$ we have that 
\begin{align*}
\kappa < \int_{H_{S}} |w_{k+1} |^2 \, dy &\leq \int_{H_{ S} \cap \{ w_{k+1} >0\}} 1 \, dy\\
&\leq | \{ y \in H_S\, : \, w_{k}(y) \geq 1-\theta_*\} |
\end{align*} 
as $\{y \, : \, w_{k+1}(y) >0 \} \subset \{y \, : \, w_{k}(y) \geq 1-\theta_*\}$.  Since
\begin{align}
| \{ y \in H_S \, : \, w_{k}(y) =0 \}| \geq c_{R,S},
\end{align} 
applying Lemma~\ref{IVL}, for every $k\leq K-1$ we have 
\begin{align*}
| \{ 0< w_{k} < 1-\theta_*\}\cap H_{ S} \}| \geq \beta >0.  
\end{align*}
This is a contradiction to the choice of $K$ since $$\{ 0 < w_k < 1-\theta_*\}\cap \{ 0 < w_{j}< 1-\theta_*\}  = \emptyset$$ for $j\neq k$, $j,k \in \{1,2,\ldots, K-1\}$ and hence
 $|H_{S} | \geq (K-1)\beta\geq 2 |H_{ S}|$.  
\end{proof}
 
Given the previous result, we can now conclude Theorem~\ref{thm:lowerbta}.
\begin{proof}[Proof of Theorem~\ref{thm:lowerbta}]
Let $R>0$ and $S\gg R$ be as in the statement of Lemma~\ref{lem:posmeas}.  For any $x\in H_R$ fixed, we recall that $y \mapsto w_k(x,y)$ satisfies $\mathcal{L}_{\epsilon, y}^{-}w_k \geq 0$.  Thus, by Theorem~\ref{thm:moser2}, there exists a constant $C>0$ independent of $k$ and $\epsilon \in (0,1)$ such that 
\begin{align*}
\| w_k \|_{L^\infty(H_R \times H_S)}^2 \leq C \sup_{x\in H_R} \int_{H_S} w_k(x,y)^2 \, dy. 
\end{align*}
Pick $\kappa >0$ such that $C\kappa =1/4$.  Applying Lemma~\ref{lem:smallup}, there exist $\theta_* \in (0,1/2)$, $\epsilon_* \in (0,1/16)$, and $K\in \N$ depending only on $\kappa$ and $S$ such that for all $\epsilon \in (0, \epsilon_*)$ there exists $1\leq k_* \leq K$ such that 
\begin{align*}
\|w_{k_*} \|_{L^\infty(H_R \times H_S)}\leq \frac{1}{2},\end{align*}
which implies that
\begin{align*}
\frac{\delta_{R,S}}{2}\Big(\frac{\theta_*}{4} \Big)^{K} \leq  \inf_{(x,y) \in H_R \times H_S} h^\epsilon(x,y)
\end{align*}
holds uniformly over $\epsilon \in (0,\epsilon_*)$.  This gives the result.  
\end{proof}

\bibliographystyle{plain}
\bibliography{QE}

\end{document}